\documentclass{article}
\usepackage[T1]{fontenc}
\usepackage{amsthm, fullpage}
\usepackage{amsmath}
\usepackage{amssymb}
\usepackage{amsfonts}
\usepackage{eucal}
\usepackage[all]{xy}
\usepackage[frenchb, english]{babel}
\usepackage{pslatex}
\usepackage[pagebackref]{hyperref}

\newtheorem{prop}{Proposition}[subsection]
\newtheorem{theo}[prop]{Théor\`eme}
\newtheorem*{theo**}{Théorème}
\newtheorem{coro}[prop]{Corollaire}
\newtheorem*{conj*}{Conjecture}
\newtheorem{lemm}[prop]{Lemme}
\newtheorem{lemm*}{Lemme}[prop]

\theoremstyle{definition}
\newtheorem{conj}[prop]{Conjecture}
\newtheorem{vide}[prop]{}
\newtheorem{defi}[prop]{Définition}
\newtheorem*{defi*}{Définition}

\theoremstyle{remark}
\newtheorem{rema}[prop]{Remarques}

\newtheorem{nota}[prop]{Notations}

\numberwithin{equation}{prop}

\newcommand{\riso}{ \overset{\sim}{\longrightarrow}\, }

\renewcommand{\sp}{\mathrm{sp}}

\newcommand{\FF}{{\mathcal{F}}}

\newcommand{\E}{{\mathcal{E}}}
\newcommand{\G}{{\mathcal{G}}}
\renewcommand{\H}{{\mathcal{H}}}

\newcommand{\D}{{\mathcal{D}}}

\newcommand{\PP}{{\mathcal{P}}}

\renewcommand{\O}{{\mathcal{O}}}

\newcommand{\V}{\mathcal{V}}
\renewcommand{\S}{\mathcal{S}}

\newcommand{\Y}{\mathcal{Y}}

\newcommand{\X}{\mathfrak{X}}

\newcommand{\U}{\mathfrak{U}}

\renewcommand{\P}{\mathbb{P}}

\renewcommand{\L}{\mathbb{L}}
\newcommand{\R}{\mathbb{R}}
\newcommand{\Q}{\mathbb{Q}}
\newcommand{\Z}{\mathbb{Z}}
\newcommand{\N}{\mathbb{N}}

\newcommand{\hdag}{  \phantom{}{^{\dag} }    }

\begin{document}
\selectlanguage{frenchb}

\title{Sur la préservation de la surconvergence par l'image directe d'un morphisme propre et lisse}
\author{Daniel Caro \footnote{L'auteur a bénéficié du soutien du réseau européen TMR \textit{Arithmetic Algebraic Geometry}
(contrat numéro UE MRTN-CT-2003-504917).}}

\date{}

\maketitle

\begin{abstract}
\selectlanguage{english}
Up to a translation in the language of arithmetic $\D$-modules, 
we prove a conjecture of Berthelot on the preservation of the overconvergence under
the direct image by a smooth proper morphism of varieties over a perfect field of characteristic $p>0$.
\end{abstract}

\selectlanguage{frenchb}
\date
\tableofcontents

\section*{Introduction}
Soit $\V$ un anneau de valuation discrète complet, 
de corps résiduel parfait $k$ de caractéristique $p>0$, de corps des
fractions $K$ de caractéristique $0$.
Soit $b\colon Y' \rightarrow Y$ un morphisme propre et lisse de $k$-variétés, i.e.,
de $k$-schémas séparés et de type fini.
Berthelot conjectura en 1986 dans \cite[4.3]{Berig_ini} que l'image directe par $b$
du $F$-isocristal surconvergent sur $Y' $ {\it constant},
i.e la cohomologie relative rigide $\R b _{\mathrm{rig}*} (Y'/K)$,
a pour faisceaux de cohomologie des $F$-isocristaux surconvergents sur $Y$.
Cette conjecture avait été validée par Berthelot dans le cas relevable lorsque $Y$ est lisse
(voir \cite[4, Théorème 5]{Berig_ini}).
Plus précisément, il a vérifié sa conjecture
lorsqu'il existe un morphisme $a \colon \X ' \rightarrow \X$ de $\V$-schémas formels propres et
un ouvert $\Y$ de $\X$ tels que $a  ^{-1} (\Y) \rightarrow \Y$ soit un morphisme
de $\V$-schémas formels lisses relevant $b $.
La preuve de ce cas repose sur le théorème de finitude de Kiehl
pour les morphismes propres en géométrie analytique rigide.

Ensuite, dans \cite[4]{Tsuzuki-BaseChangeCoh},
Tsuzuki a étendu cette conjecture
du cas constant au cas général de la manière suivante
(en fait il la formule plus généralement en remplaçant la base $\V$ par un triplet comprenant un $\V$-schéma formel) :
{\og soient $a \colon X' \rightarrow X$ un morphisme propre de $k$-variétés,
$Y$ un ouvert de $X$ tel que le morphisme induit $b \colon Y' := a  ^{-1} (Y) \rightarrow Y$
soit de plus lisse. 
Alors, pour tout ($F$-)isocristal $E'$ sur $Y'$ surconvergent le long de $X' \setminus Y'$,
les faisceaux de cohomologie de
$\R a _{\mathrm{rig}*} (E'/K)$
sont des
($F$-)isocristaux sur $Y$ surconvergents le long de $X \setminus Y$.\fg}
On la nommera encore conjecture de Berthelot (avec coefficients si on veut préciser).
Rappelons que lorsque $X$ est propre, $\R a _{\mathrm{rig}*} (E'/K)$ ne dépend canoniquement que
de $b $ et se note par conséquent $\R b _{\mathrm{rig}*} (E'/K)$.
Via un théorème de changement de base de la cohomologie rigide relative,
Tsuzuki a validé cette conjecture dans le contexte suivant :
soient $f\colon \PP' \rightarrow \PP$ un morphisme propre de $\V$-schémas formels séparés de type fini,
$X$ un sous-schéma fermé de la fibre spéciale de $\PP$, $Y$ un ouvert de $X$,
$X' := f^{-1} (X)$, $Y' := f^{-1} (Y)$ tels que
$f$ soit lisse au voisinage de $Y'$ et $\PP$ soit lisse au voisinage de $Y$.
Dans ce cas, pour tout ($F$-)isocristal $E'$ sur $Y'$ surconvergent le long de $X' \setminus Y'$,
les faisceaux de cohomologie de
$\R f _{\, \mathrm{rig}*} (E'/K)$
sont des
($F$-)isocristaux sur $Y$ surconvergent le long de $X \setminus Y$
(voir les théorèmes \cite[4.1.1 et 4.1.4]{Tsuzuki-BaseChangeCoh} de Tsuzuki).

Plus récemment, dans le cas où $\V$ est modérément ramifié,
\'Etesse a validé cette conjecture de Berthelot avec coefficients dans le cas absolu (i.e. $X$ est propre)
si l'une des deux conditions est validée : 
$b$ est relevable sur $\V$ ou $Y'$ est une intersection complète
relative dans des espaces projectifs sur $Y$ (voir \cite{etesse-2008}).

On dispose enfin de deux autres versions (l'une est plus forte que l'autre) de Shiho
de cette conjecture de Berthelot avec coefficients (voir les deux conjectures \cite[5.3 et 5.5]{shiho-logRC-RCI}).
Dans \cite{shiho-logRC-RCII}, Shiho résout sa conjecture la plus faible.
Enfin, lorsque $b$ n'est plus forcément propre et lisse, Shiho a vérifié la surcohérence générique (i.e., devient
un isocristal surconvergent sur un ouvert dense) de la cohomologie rigide relative avec coefficients. 
\bigskip

Nous nous proposons d'apporter un éclairage nouveau sur ces questions via la théorie de Berthelot des $\D$-modules
arithmétiques. Rappelons d'abord comment la théorie des $\D$-modules
arithmétiques est reliée à la théorie des isocristaux surconvergents.
Soient $X$ une $k$-variété et $Y$ un ouvert de $X$ tels que $(Y,X)$ soit $d$-réalisable,
i.e., tels qu'il existe
un $\V$-schéma formel $\PP $ séparé et lisse,
un diviseur $T$ de la fibre spéciale de $\PP$ et une immersion fermée $X \hookrightarrow \PP$
vérifiant $Y = X \setminus T$.
D'après \cite{caro-pleine-fidelite}, nous bénéficions d'une équivalence
entre la catégorie $(F\text{-})\mathrm{Isoc}^{\dag} (Y,X/K)$ des $(F\text{-})$isocristaux surconvergents sur $(Y,X)/K$
et la catégorie $(F\text{-})\mathrm{Isoc}^{\dag \dag} (Y,X/K)$ des $(F\text{-})$isocristaux surcohérents sur $(Y,X)/K$.
Les objets de la catégorie $(F\text{-})\mathrm{Isoc}^{\dag \dag} (Y,X/K)$ sont des
$(F\text{-})\D$-modules arithmétiques sur $(Y,X)/K$ vérifiant certaines hypothèses de finitude.
Pour obtenir cette équivalence, nous utilisons le morphisme de spécialisation $\sp \colon \PP _K \to \PP$,
où $\PP _K$ est la fibre générique au sens de Raynaud (l'espace analytique rigide canoniquement associé)
de $\PP$. Ce morphisme de spécialisation $\sp$ permet de relier le monde de la géométrie analytique rigide 
dans lequel vivent les isocristaux surconvergents et celui de la géométrie formelle dans lequel vivent les $\D$-modules arithmétiques.
Via cette équivalence, nous obtenons une traduction de la conjecture de Berthelot en terme de $\D$-modules arithmétiques que nous détaillons ci-dessous.

\bigskip
Dans cet article, nous prouvons une traduction dans le langage des $\D$-modules arithmétiques de la conjecture de Berthelot. 
Plus précisément, 
soient $X$ une $k$-variété et $Y$ un ouvert de $X$ tels que $(Y,X)$ soit proprement $d$-réalisable
i.e., tels qu'il existe
un $\V$-schéma formel $\PP $ propre et lisse,
un diviseur $T$ de la fibre spéciale de $\PP$ et une immersion (non nécessairement fermée) $X \hookrightarrow \PP$
vérifiant $Y = X \setminus T$. On désigne par $\D ^{\dag} _{\PP}(\hdag T) _{\Q}$ le faisceau sur $\PP$ des opérateurs différentiels d'ordre infini, de niveau fini avec des singularités surconvergentes le long de $T$ (voir \cite{Be1}).
On vérifie que la sous-catégorie pleine de 
$(F\text{-})D ^\mathrm{b} _\mathrm{surcoh} (\D ^{\dag} _{\PP}(\hdag T) _{\Q} ) $ des complexes à support dans $X$ 
ne dépend pas du choix du plongement $X \hookrightarrow \PP$ et du diviseur $T$ de $P$.
On la note alors
$(F\text{-})D ^\mathrm{b} _\mathrm{surcoh} (\D ^\dag _{(Y,X)/K}) $.
Ses objets sont 
les $(F\text{-})$complexes surcohérents de $\D$-modules arithmétiques sur $(Y,X)/K$.
On définit de même (on traite dans un premier temps le cas où $Y$ est lisse afin d'utiliser l'équivalence entre isocristaux surconvergents et 
isocristaux surcohérents exprimée dans le paragraphe ci-dessus)
la sous-catégorie
$(F\text{-})D ^\mathrm{b} _\mathrm{isoc} (\D ^\dag _{(Y,X)/K})$
de $(F\text{-})D ^\mathrm{b} _\mathrm{surcoh} (\D ^\dag _{(Y,X)/K}) $
des $(F\text{-})$complexes surcohérents de $\D$-modules arithmétiques sur $(Y,X)/K$ dont les espaces de cohomologie sont
des $(F\text{-})$isocristaux surcohérents sur $(Y,X)/K$.
Pour vérifier cette indépendance canonique, nous prouvons la propriété suivante de stabilité de la surcohérence : 
pour tout morphisme $f\colon \PP' \to \PP$ de $\V$-schémas formels séparés et lisses tel que $T':=f ^{-1} (T)$ soit un diviseur de la fibre spéciale de $\PP'$, pour tout $(F\text{-})$complexe à cohomologie $\D ^{\dag} _{\PP'}(\hdag T') _{\Q}$-surcohérente $\E'$ à support propre sur $\PP$, 
l'image directe de $\E'$ par $f$ est à cohomologie $\D ^{\dag} _{\PP}(\hdag T) _{\Q}$-surcohérente (voir \ref{surhol-conjA}).  
Cela correspond à une version surcohérente de 
la conjecture de Berthelot sur la stabilité de l'holonomie par image directe énoncée dans \cite[5.3.6]{Beintro2}.

Nous établissons dans ce papier la version suivante de la conjecture de Berthelot (voir \ref{b+proprelisse-Yqcq}): 
soient $(Y',X')$ et $(Y,X)$ deux couples de $k$-variétés proprement $d$-réalisables,
$a \colon X' \to X$ un morphisme propre tels que $ a  ^{-1} (Y) =Y'$ et tels que le morphisme induit $Y' \to Y$ soit propre et lisse.
Alors, le foncteur image directe par $a $ se factorise sous la forme :
\begin{equation}
\notag
a _{+} \colon  (F\text{-})D ^\mathrm{b} _\mathrm{isoc} (\D ^\dag _{(Y',X')/K})
\rightarrow
(F\text{-})D ^\mathrm{b} _\mathrm{isoc} (\D  ^\dag _{(Y,X)/K}).
\end{equation}
En fait, si on ne se préoccupe pas de l'indépendance par rapport 
aux choix (e.g. du plongement $X \hookrightarrow \PP$ et du diviseur $T$ de $P$),
nous prouvons une version légèrement étendue via \ref{b+proprelisse-varlisse-gen}, i.e. on remplace la notion de proprement $d$-plongeabilité 
par la notion de $d$-plongeabilité.

{\bf Remerciements.} Je remercie vivement Kiran Kedlaya pour une question posée lors de la conférence en l'honneur de Gilles Christol qui m'a incitée à m'intéresser à cette conjecture
de Berthelot via les $\D$-modules arithmétiques. Je remercie Jean-Yves Etesse pour une question analogue lors de cette même conférence. 
Je remercie Nobuo Tsuzuki pour les discussions notamment sur la comparaison entre les versions formelles et rigides des conjectures de Berthelot lors d'une invitation à Sendai. 
\\

{\bf Notations.}
Tout au long de cet article, nous garderons les notations
suivantes : soit $\V$ un anneau de valuation discrète complet,
de corps résiduel parfait $k$ des caractéristique $p>0$, de corps de
fractions $K$ de caractéristique $0$, d'uniformisante $\pi$.
Les $\V$-schémas formels seront notés par des lettres calligraphiques ou
gothiques et leur fibre spéciale par les lettres romanes
correspondantes.
On fixe $s\geq 1$ un entier naturel et $F$ désigne la puissance
$s$-ième de l'endomorphisme de Frobenius. Les modules sont par défaut des modules à gauche.
Si $\E$ est un faisceau abélien, $\E _\Q$ désignera $\E \otimes _\Z \Q$. 
On note $m$ un entier positif.
Si $f\colon \X ' \rightarrow \X$ est un morphisme de $\V$-schémas formels lisses,
on notera par des lettres droites les fibres spéciales et $f \colon X' \rightarrow X$
sera si aucune confusion n'est à craindre le morphisme induit. 
Sauf mention du contraire, on supposera (sans nuire à la généralité) les $k$-schémas réduits.
En général, lorsqu'un diviseur est vide, on évite de l'indiquer, e.g. dans les opérations cohomologiques correspondantes.

{\bf Convention.}
Lorsque nous dirons {\og surholonome\fg} nous voudrons dire {\og surholonome après tout changement de la base $\V$\fg} (voir \cite{surcoh-hol}).

\section{Rappels et compléments sur les $\D$-modules arithmétiques}

Nous donnons dans cette section quelques rappels sur les $\D$-modules arithmétiques.
Nous définissons aussi quelques catégories de $\D$-modules que nous utiliserons par la suite.

\subsection{Opérations cohomologiques}

Donnons d'abord quelques précisions sur les notations que nous adopterons concernant les foncteurs image directe et image inverse extraordinaire.

\begin{vide}
[Image inverse extraordinaire]
\label{nota-tau}
Soient $f, f '\colon \PP' \rightarrow \PP$ deux morphismes de $\V$-schémas formels lisses
induisant le même morphisme $P' \to P$ au niveau des fibres spéciales.
On dispose des foncteurs images inverses extraordinaires
$f ^{!}$, $f ^{\prime !}\colon  D ^\mathrm{b} _\mathrm{coh} (\smash{\D} ^{\dag} _{\PP,\Q})
\to
D ^\mathrm{b}  (\smash{\D} ^{\dag} _{\PP',\Q})$.
Via \cite[2.1.5]{Be2}, on bénéficie de l'isomorphisme canonique de foncteurs de la forme $f^! \riso f ^{\prime !}$. Ces isomorphismes vérifient les formules de transitivité usuelles. 
De même, pour $m$ assez grand, i.e. $p ^m > e/(p-1)$ avec $e$ l'indice de ramification de $\V$,
en désignant par $f^{(m)!}$ et $ f ^{\prime (m) !}$ les foncteurs 
$D ^\mathrm{b} _\mathrm{coh} (\smash{\widehat{\D}} ^{(m)} _{\PP})
\to
D ^\mathrm{b}  (\smash{\widehat{\D}} ^{(m)} _{\PP'})$
images inverses extraordinaires de niveau $m$ par $f$ et $f'$,
on dispose de l'isomorphisme de la forme
$f^{(m)!} \riso f ^{\prime (m) !}$.
De tels (et plus généralement avec des singularités surconvergentes le long d'un diviseur) isomorphismes seront notés $\tau$.
\end{vide}

\begin{vide}
[Image directe]
\label{defi-f+-f+(m)}
Soit $f\colon \PP' \rightarrow \PP$ un morphisme propre de $\V$-schémas formels lisses.
Les foncteurs images directes par $f$ de niveau $m$ préservent la cohérence et seront notés (s'il n'y a pas de confusion possible entre les deux)
$f ^{(m)} _+\colon D ^\mathrm{b} _\mathrm{coh} (\smash{\widehat{\D}} ^{(m)} _{\PP'})
\rightarrow
D ^\mathrm{b} _\mathrm{coh} (\smash{\widehat{\D}} ^{(m)} _{\PP})$,
$f ^{(m)} _+\colon D ^\mathrm{b} _\mathrm{coh} (\smash{\widehat{\D}} ^{(m)} _{\PP',\Q})
\rightarrow
D ^\mathrm{b} _\mathrm{coh} (\smash{\widehat{\D}} ^{(m)} _{\PP,\Q})$.
Ces deux foncteurs commutent à $-\otimes _\Z \Q$, i.e.,
pour tout
$\E'\in D ^\mathrm{b} _\mathrm{coh} (\smash{\widehat{\D}} ^{(m)} _{\PP'})$,
$f ^{(m)} _+ (\E _\Q) \riso f ^{(m)} _+ (\E') _\Q$.
S'il n'y a pas de risque de confusion et pour alléger les notations,
nous écrirons $f  _+$ au lieu de $f ^{(m)} _+$.
Le foncteur image directe par $f$ sera noté
$f_+\colon D ^\mathrm{b} _\mathrm{coh} (\smash{\D} ^{\dag} _{\PP'})
\rightarrow
D ^\mathrm{b} _\mathrm{coh} (\smash{\D} ^{\dag} _{\PP})$, de même avec des $\Q$.

Soient $f, f '\colon \PP' \rightarrow \PP$ deux morphismes de $\V$-schémas formels lisses
induisant le même morphisme $P' \to P$ au niveau des fibres spéciales.
On dispose de même que pour l'image inverse extraordinaire 
des isomorphismes de recollement $f _{+} \riso f ' _{+}$ satisfaisant aux formules de transitivité usuelles (voir \cite{Be2}). 
\end{vide}

\begin{vide}
[Foncteur cohomologique local]
\label{Gamma=u+circu!}
Soit $\PP$ un $\V$-schéma formel lisse, $X$ un sous-schéma fermé de $P$. On notera
$\R \underline{\Gamma} ^\dag _{X}$ le foncteur cohomologique local défini dans \cite[2]{caro_surcoherent} (conjecturalement égal à celui défini par Berthelot). 
Rappelons que lorsqu'il existe un morphisme $u\colon \X \hookrightarrow \PP$ de $\V$-schémas formels lisses dont la réduction induite modulo $\pi$ est l'immersion fermée canonique $X \hookrightarrow P$, d'après le dernier point de \cite[1.15]{caro_surholonome}, 
$\R \underline{\Gamma} ^\dag _{X}  \riso u _{+} u ^{!} $.

\end{vide}

\subsection{Catégories de $\D$-modules arithmétiques sur les $d$-cadres}

\begin{defi}
\label{defi-cad}
Un {\og $d$-cadre $(\PP, T,X,Y)$\fg} est la donnée d'un $\V$-schéma formel séparé et lisse $\PP$, 
d'un diviseur $T$ de $P$, d'un sous-schéma fermé $X$ de $P$ tels que
$Y=X \setminus T$.
 Un $d$-cadre $(\PP, T,X,Y)$
est {\og lisse en dehors du diviseur\fg} si $Y$ est lisse.

Un morphisme de $d$-cadres (lisses en dehors du diviseur) 
$\theta \colon  (\PP', T',X',Y')\to (\PP, T,X,Y)$
est la donnée d'un morphisme $f\colon  \PP' \to \PP$ tel que $f$ induise un morphisme $X'\to X$ tel que $f (Y') \subset Y$.
Si $a\colon X' \to X$, $b \colon  Y' \to Y$ sont les morphismes induits par $f$, le morphisme $\theta$ se note aussi $(f,a,b)$.
\end{defi}

Nous verrons dans les prochains chapitres que les catégories définies ci-dessous dans \ref{nota-surhol-T-pre}, 
\ref{nota-6.2.1dev} et \ref{nota-surhol-T} sont sous certaines conditions indépendants des choix faits.

\begin{nota}
[Cohérence et surcohérence]

\label{nota-surhol-T-pre}
Soit $(\PP, T,X,Y)$ un $d$-cadre 
(voir les conventions de \ref{defi-cad}).

$\bullet$ On note $(F\text{-})\mathrm{Coh} (\PP, T,X/K)$ la catégorie des $(F\text{-})\D ^{\dag} _{\PP}(\hdag T) _{\Q}$-modules cohérents à support dans $X$. 
On note $(F\text{-})\mathrm{Surcoh} (\PP, T,X/K)$ la catégorie des $(F\text{-})\D ^{\dag} _{\PP}(\hdag T) _{\Q}$-modules surcohérents à support dans $X$ (la surcohérence est définie dans \cite{caro_surcoherent}). 

$\bullet$ On note $(F\text{-})D ^\mathrm{b} _\mathrm{coh}  (\PP, T, X/K)$ la sous-catégorie pleine de 
$(F\text{-})D ^\mathrm{b} _\mathrm{coh} (\D ^\dag _{\PP} (\hdag T) _\Q)$ des $(F\text{-})$complexes à support dans $X$. 
On note $(F\text{-})D ^\mathrm{b} _\mathrm{surcoh}  (\PP, T, X/K)$ la sous-catégorie pleine de 
$(F\text{-})D ^\mathrm{b} _\mathrm{surcoh} (\D ^\dag _{\PP} (\hdag T) _\Q)$ des $(F\text{-})$complexes à support dans $X$.

  $\bullet$ Lorsque le diviseur $T$ est vide, on omet de l'indiquer dans les précédentes catégories.

\end{nota}

\begin{rema}
Avec les notations de \ref{nota-surhol-T-pre}, lorsque le diviseur $T$ n'est pas vide,
un objet de $\mathrm{Surcoh} (\PP, T,X/K)$ n'est pas toujours $\D ^{\dag} _{\PP,\Q}$-surcohérent (sauf si le module est muni
d'une structure de Frobenius: voir \cite{caro-Tsuzuki-2008}). Il ne faut donc pas confondre avec les catégories introduites dans
\cite{caro_surcoherent} ou \cite{caro_surholonome} (voir aussi \ref{nota-surhol-T}).
Contrairement à celles introduites dans 
\cite{caro_surcoherent},
les catégories $\mathrm{Surcoh} (\PP, T,X/K)$ sont suffisamment grosses pour contenir tous les isocristaux surcohérents
sur $(Y,X)/K$ lorsque $Y$ est lisse. 
\end{rema}

Rappelons les notations (voir \cite[3.1.2 et 3.5.10]{caro-pleine-fidelite}) et définitions
(voir \cite[5.4.5 et 5.4.7]{caro-pleine-fidelite}) suivantes:

\begin{defi}
  \label{defi-part-surcoh}
Soit $(\PP, T,X,Y)$ un $d$-cadre lisse en dehors du diviseur
(voir les conventions de \ref{defi-cad}). 

\begin{itemize}
\item On note $\mathrm{Isoc} ^{\dag \dag} (\PP, T, X/K)$ la sous-catégorie pleine de $\mathrm{Surcoh}  (\PP, T, X/K)$
des $\D ^{\dag} _{\PP}(\hdag T) _{\Q}$-modules surcohérents $\E$ tels que, en posant $\U := \PP \setminus T$, 
le module $\E |\U$ soit dans l'image essentielle du foncteur $\sp _{Y\hookrightarrow \U, +}$ (défini dans \cite{caro-construction}, ce qui est décrit ici dans \ref{prop-donnederecol-dag}).
Les objets de $\mathrm{Isoc} ^{\dag \dag} (\PP, T, X/K)$ sont {\og les isocristaux partiellement surcohérents sur $(\PP, T, X/K)$\fg}.
Lorsque $\PP$ est propre, on omet le qualificatif {\og partiellement\fg}. 

\item La catégorie $(F\text{-})\mathrm{Isoc} ^{\dag \dag} (\PP, T, X/K)$ ne dépend, à équivalence canonique de catégories près, ni du choix de l'immersion fermée
$X \hookrightarrow \PP$ et ni de celui du diviseur $T$ tel que $Y = X \setminus T$ (mais seulement de $(Y,X)/K$).
On note alors
sans ambiguïté 
$(F\text{-})\mathrm{Isoc}^{\dag \dag} (Y,X/K)$
à la place de $(F\text{-})\mathrm{Isoc} ^{\dag \dag} (\PP, T, X/K)$.
Ses objets sont les {\og $(F\text{-})$isocristaux partiellement surcohérents sur $(Y,X)/K$\fg} 
ou simplement {\og $(F\text{-})$isocristaux surcohérents sur $(Y,X)/K$\fg}.

\item Lorsque $X$ est propre, la catégorie 
$(F\text{-})\mathrm{Isoc}^{\dag \dag} (Y,X/K)$ ne dépend pas non plus du choix de la compactification propre $X$ de $Y$.
On la note alors
$(F\text{-})\mathrm{Isoc}^{\dag \dag} (Y/K)$.
Ses objets sont les {\og $(F\text{-})$isocristaux surcohérents sur $Y/K$\fg}.
\end{itemize}

\end{defi}

\begin{nota}
[Isocristaux]
\label{nota-6.2.1dev}
Soit $(\PP, T,X,Y)$ un $d$-cadre lisse en dehors du diviseur.
On note $(F\text{-})D ^\mathrm{b} _\mathrm{isoc} (\PP, T, X/K)$ la sous-catégorie pleine de 
$(F\text{-})D ^\mathrm{b} _\mathrm{coh} (\D ^\dag _{\PP} (\hdag T) _\Q)$ 
des $(F\text{-})$complexes dont les espaces de cohomologie appartiennent à 
$(F\text{-})\mathrm{Isoc} ^{\dag \dag} (\PP, T, X/K)$.
Rappelons que lorsque $T$ est vide, on n'indique pas i.e. on écrit
$(F\text{-})D ^\mathrm{b} _\mathrm{isoc} (\PP, X/K)$ au lieu de 
$(F\text{-})D ^\mathrm{b} _\mathrm{isoc} (\PP, \emptyset, X/K)$.
\end{nota}

\begin{nota}
On note $\mathrm{oub} _{T}$ le foncteur oubli canonique $(F\text{-})D ^\mathrm{b} (\D ^\dag _{\PP} (\hdag T) _\Q) \to (F\text{-})D ^\mathrm{b}  (\D ^\dag _{\PP,\Q})$ 
et le foncteur extension $(\hdag T) = \D ^\dag _{\PP} (\hdag T) _\Q \otimes _{\D ^\dag _{\PP,\Q}} - \colon (F\text{-})D ^\mathrm{b} _\mathrm{coh} (\D ^\dag _{\PP,\Q} )\to (F\text{-})D ^\mathrm{b} _\mathrm{coh} (\D ^\dag _{\PP} (\hdag T) _\Q)$.
\end{nota}

\begin{nota}
[Surholonomie]

\label{nota-surhol-T}

Pour la notion de surholonomie, on se reportera à \cite{caro_surholonome}.

$\bullet$ On notera $(F\text{-})\mathrm{Surhol} (\PP, T, X/K)$ la catégorie des 
$(F\text{-})\D ^\dag _{\PP,\Q}$-modules surholonomes (après tout changement de la base $\V$ d'après la convention énoncée au début de ce papier)
$\E$ à support dans $X$ tels que le morphisme canonique $\E \to \E (\hdag T)$ soit un isomorphisme. 

$\bullet$
On désigne par $(F\text{-})D ^\mathrm{b} _\mathrm{surhol} (\D ^\dag _{\PP,\Q})$ la sous-catégorie pleine de 
$(F\text{-})D ^\mathrm{b}  (\D ^\dag _{\PP,\Q})$ des 
 $(F\text{-})$complexes surholonomes (après tout changement de la base $\V$).

$\bullet$ On note $(F\text{-})D ^\mathrm{b} _\mathrm{surhol} (\D ^\dag _{\PP} (\hdag T) _\Q)$  
la sous-catégorie 
  pleine de $(F\text{-})D ^\mathrm{b} _\mathrm{coh} (\D ^\dag _{\PP} (\hdag T) _\Q)$ des $(F\text{-})$complexes $\FF$ tels que 
  $oub _{T} (\FF)\in (F\text{-})D ^\mathrm{b} _\mathrm{surhol} (\D ^\dag _{\PP,\Q})$.

$\bullet$ On note $(F\text{-})D ^\mathrm{b} _\mathrm{surhol}  (\PP, T, X/K)$ la sous-catégorie pleine de $(F\text{-})D ^\mathrm{b} _\mathrm{surhol} (\D ^\dag _{\PP} (\hdag T) _\Q)$ des $(F\text{-})$complexes à support dans $X$. 

  $\bullet$ Lorsque le diviseur $T$ est vide, on omet de l'indiquer dans les précédentes catégories.
 \end{nota}

Énoncé d'abord le premier résultat d'indépendance le plus immédiat qui nous sera utile pas la suite: 
\begin{lemm}
[Indépendance par rapport à $X$]
\label{indt-X}
Soit $(\PP, T,X,Y)$ un $d$-cadre (resp. un $d$-cadre lisse en dehors du diviseur). 
Posons $\mathfrak{C}$ pour 
$(F\text{-})D ^\mathrm{b} _\mathrm{surcoh}$ 
ou $\mathrm{Surcoh}$
ou 
$(F\text{-})D ^\mathrm{b} _\mathrm{surhol}$
ou 
$(F\text{-})\mathrm{Surhol}$
(resp. 
$(F\text{-})D ^\mathrm{b} _\mathrm{isoc}$
ou 
$(F\text{-})\mathrm{Isoc} ^{\dag \dag}$).
Soit $X'$ la clôture de $X \setminus T$ dans $P$. 
Alors $\mathfrak{C}(\PP, T,X/K)= \mathfrak{C}(\PP, T,X'/K)$.
\end{lemm}

\begin{proof}
Soit $\E$ un objet de $\mathfrak{C}(\PP, T,X/K)$. 
On dispose dans $\mathfrak{C}(\PP, T,X/K)$ de la flèche 
$\R \underline{\Gamma} ^\dag _{X'} (\E)  \to \E$.
Par \cite[4.3.12]{Be1}, pour prouver que cette flèche est un isomorphe, 
il suffit de le vérifier 
en dehors de $T$, ce qui est immédiat.  
\end{proof}

\begin{rema}
On dispose des factorisations 
$\mathrm{oub} _{T} 
\colon 
(F\text{-})D ^\mathrm{b} _\mathrm{surhol}  (\PP, T, X/K)
\to
(F\text{-})D ^\mathrm{b} _\mathrm{surhol}  (\PP,X/K)$
et, par stabilité de la surholonomie,
 $(\hdag T) 
 \colon 
 (F\text{-})D ^\mathrm{b} _\mathrm{surhol}  (\PP,X/K)
 \to 
 (F\text{-})D ^\mathrm{b} _\mathrm{surhol}  (\PP, T, X/K)$.
 
 Par contre, on prendra garde au fait que l'on ne dispose pas de la factorisation 
$\mathrm{oub} _{T} 
\colon 
D ^\mathrm{b} _\mathrm{surcoh}  (\PP, T, X/K)
\not \to
D ^\mathrm{b} _\mathrm{surcoh}  (\PP,X/K)$.
En effet, il est faux en général (i.e. sans structure de Frobenius) 
qu'un $\D ^\dag _{\PP} (\hdag T) _\Q$-surcohérent soit
$\D ^\dag _{\PP,\Q}$-surcohérent. 
\end{rema}

\begin{rema}
\label{rema-cv-surhol}
\begin{itemize}
\item Soient $\U$ un $\V$-schéma formel séparé et lisse, $Y$ un sous-schéma fermé de $U$ lisse sur $k$. 
Remarquons que grâce à \cite[7.3]{caro_surholonome}, 
un objet de $D ^\mathrm{b} _\mathrm{coh} (\D ^\dag _{\U,\Q})$ dont les espaces de cohomologie sont dans l'image essentielle 
du foncteur $\sp _{Y\hookrightarrow \U,+}$ (voir \ref{prop-donnederecol-dag}) est alors un objet de
$D ^\mathrm{b} _\mathrm{surhol} (\D ^\dag _{\U,\Q})$ (nous n'avons pas besoin de structure de Frobenius). 
On dispose ainsi de l'inclusion canonique $(F\text{-})D ^\mathrm{b} _\mathrm{isoc} (\U, Y/K) \subset
(F\text{-})D ^\mathrm{b} _\mathrm{surhol} (\U, Y/K) $.

\item Soit $(\PP, T,X,Y)$ un $d$-cadre lisse en dehors du diviseur.
Si on rajoute des structure de Frobenius, on obtient plus généralement, grâce à \cite{caro-Tsuzuki-2008},
l'inclusion canonique
$F\text{-}D ^\mathrm{b} _\mathrm{isoc} (\PP, T, X/K) \subset
F\text{-}D ^\mathrm{b} _\mathrm{surhol} (\PP, T, X/K)$.
\end{itemize}

\end{rema}

\begin{rema}
\label{rema-surhol-T}
On dispose pour tout entier $r$ de la factorisation 
$\mathcal{H} ^{r}\colon  (F\text{-})D ^\mathrm{b} _\mathrm{coh}(\PP, T, X/K) 
\to 
\mathrm{Coh} (\PP, T, X/K) $.
De même pour la surcohérence ou la surholonomie (voir \ref{surhol-espcoh}).
\end{rema}

\begin{lemm}
\label{cat-T-sansT}
On désigne par $(F\text{-})D ^\mathrm{b} _{\mathrm{surhol} , T}(\D ^\dag _{\PP,\Q})$
  la sous-catégorie pleine de $(F\text{-})D ^\mathrm{b} _\mathrm{surhol} (\D ^\dag _{\PP,\Q})$
  des $(F\text{-})$complexes $\E$ tels que le morphisme $\E \to \E (\hdag T)$ canonique de $(F\text{-})D ^\mathrm{b}  (\D ^\dag _{\PP,\Q})$ soit un isomorphisme. 
  
Avec les notations de \ref{nota-surhol-T},
les foncteurs $\mathrm{oub} _T$ et $(\hdag T)$ induisent des équivalences quasi-inverses entre les catégories
$(F\text{-})D ^\mathrm{b} _{\mathrm{surhol} , T}(\D ^\dag _{\PP,\Q})$ 
et $(F\text{-})D ^\mathrm{b} _\mathrm{surhol} (\D ^\dag _{\PP} (\hdag T) _\Q)$.
Dans la pratique, on pourra identifier ces deux catégories. 
\end{lemm}

 \begin{proof}
 Soit $\E$ un objet de $(F\text{-})D ^\mathrm{b} _{\mathrm{surhol} , T}(\D ^\dag _{\PP,\Q})$. 
  Le morphisme canonique $\E \to \E (\hdag T)$ est par définition un isomorphisme. D'où : $oub _{T} (\E (\hdag T)) \in (F\text{-})D ^\mathrm{b} _\mathrm{surhol} (\D ^\dag _{\PP,\Q})$. 
Cela entraîne que le foncteur $(\hdag T)$ se factorise sous la forme 
  $(\hdag T)\colon (F\text{-})D ^\mathrm{b} _{\mathrm{surhol} , T}(\D ^\dag _{\PP,\Q}) \to (F\text{-})D ^\mathrm{b} _\mathrm{surhol} (\D ^\dag _{\PP} (\hdag T) _\Q)$.

D'un autre côté, soit $\FF$ un objet de $(F\text{-})D ^\mathrm{b} _\mathrm{surhol} (\D ^\dag _{\PP} (\hdag T) _\Q)$. Comme $\FF$ est $\D ^\dag _{\PP} (\hdag T) _\Q$-cohérent et $ oub _{T} (\FF)$ est (en particulier) $\D ^\dag _{\PP,\Q}$-cohérent, on déduit de \cite[4.3.12]{Be1} que 
le morphisme canonique $ oub _{T} (\FF) \to  (\hdag T) (oub _{T}(\FF ))$ est un isomorphisme. 
On obtient ainsi la factorisation 
$oub _{T}$ : 
$(F\text{-})D ^\mathrm{b} _\mathrm{surhol} (\D ^\dag _{\PP} (\hdag T) _\Q)
\to 
(F\text{-})D ^\mathrm{b} _{\mathrm{surhol} , T}(\D ^\dag _{\PP,\Q})$.
    
Ainsi, nous avons vérifiés que les foncteurs $oub _{T}$ et $(\hdag T)$ induisent les équivalences quasi-inverses de la forme
$oub _{T}\colon (F\text{-})D ^\mathrm{b} _\mathrm{surhol} (\D ^\dag _{\PP} (\hdag T) _\Q) \cong (F\text{-})D ^\mathrm{b} _{\mathrm{surhol} , T}(\D ^\dag _{\PP,\Q})$
et
$(\hdag T)\colon (F\text{-})D ^\mathrm{b} _{\mathrm{surhol} , T}(\D ^\dag _{\PP,\Q}) \cong (F\text{-})D ^\mathrm{b} _\mathrm{surhol} (\D ^\dag _{\PP} (\hdag T) _\Q)$.

\end{proof}

\begin{vide}
On dispose du foncteur oubli
pleinement fidèle $\mathrm{oub} _T\colon  
\smash{\underset{\longrightarrow}{LD}} ^\mathrm{b} _{\Q,\mathrm{qc}}  
(\widehat{\D} ^{(\bullet)} _{\X} ( T) )
\to 
\smash{\underset{\longrightarrow}{LD}} ^\mathrm{b} _{\Q,\mathrm{qc}}  
(\widehat{\D} ^{(\bullet)} _{\X} )$
et du foncteur extension 
$(\hdag T)\colon 
\smash{\underset{\longrightarrow}{LD}} ^\mathrm{b} _{\Q,\mathrm{qc}}  
(\widehat{\D} ^{(\bullet)} _{\X} )
\to 
\smash{\underset{\longrightarrow}{LD}} ^\mathrm{b} _{\Q,\mathrm{qc}}  
(\widehat{\D} ^{(\bullet)} _{\X} (T))$.
Par \cite[1.1.8]{caro_courbe-nouveau}, 
pour tout $\FF \in \smash{\underset{\longrightarrow}{LD}} ^\mathrm{b} _{\Q,\mathrm{qc}}  
(\widehat{\D} ^{(\bullet)} _{\X} ( T) )$
le morphisme canonique 
$(\hdag T) \circ \mathrm{oub} _T (\FF) \to \FF$ est un isomorphisme
(par associativité du produit tensoriel complété, on s'était ramené au cas où $\FF$ était associé à 
$\O _{\X} (\hdag T) _{\Q}$).

Notons $\smash{\underset{\longrightarrow}{LD}} ^\mathrm{b} _{\Q,\mathrm{qc}, T}  
(\widehat{\D} ^{(\bullet)} _{\X})$ la sous-catégorie pleine de
$\smash{\underset{\longrightarrow}{LD}} ^\mathrm{b} _{\Q,\mathrm{qc}}  
(\widehat{\D} ^{(\bullet)} _{\X} )$ 
des complexes $\E$ tel que le morphisme canonique
$ \E \to \mathrm{oub} _T \circ (\hdag T) (\E)$ soit un isomorphisme. 

\end{vide}

Le lemme \ref{cat-T-sansT} ci-dessus bénéficie d'un analogue quasi-cohérent via le lemme ci-dessous: 
\begin{lemm}
\label{cat-T-sansT-qc}
Les foncteurs $\mathrm{oub} _T$ et $(\hdag T)$ induisent des équivalences quasi-inverses entre les catégories
$ \smash{\underset{\longrightarrow}{LD}} ^\mathrm{b} _{\Q,\mathrm{qc}}  
(\widehat{\D} ^{(\bullet)} _{\X} ( T))$
et 
$\smash{\underset{\longrightarrow}{LD}} ^\mathrm{b} _{\Q,\mathrm{qc}, T}  
(\widehat{\D} ^{(\bullet)} _{\X})$.
\end{lemm}

\begin{proof}
1) Vérifions d'abord que le foncteur $\mathrm{oub} _T$ induit la factorisation requise. 
Soit $\FF \in \smash{\underset{\longrightarrow}{LD}} ^\mathrm{b} _{\Q,\mathrm{qc}}  
(\widehat{\D} ^{(\bullet)} _{\X} ( T))$.
Comme le morphisme canonique
$(\hdag T) \circ \mathrm{oub} _T (\FF) \to \FF$ est un isomorphisme, 
on obtient l'isomorphisme canonique 
$\mathrm{oub} _T \circ (\hdag T) \circ \mathrm{oub} _T (\FF) \riso \mathrm{oub} _T(\FF)$.
Comme le morphisme canonique composé 
$\mathrm{oub} _T(\FF) 
\to 
\mathrm{oub} _T \circ (\hdag T) \circ \mathrm{oub} _T (\FF) \to \mathrm{oub} _T(\FF)$
est l'identité, il en résulte que 
le premier morphisme canonique
$\mathrm{oub} _T(\FF) 
\to 
\mathrm{oub} _T \circ (\hdag T) \circ \mathrm{oub} _T (\FF)$
est un isomorphisme. 
Cela signifie que 
$\mathrm{oub} _T  (\FF) 
\in 
\smash{\underset{\longrightarrow}{LD}} ^\mathrm{b} _{\Q,\mathrm{qc}, T}  
(\widehat{\D} ^{(\bullet)} _{\X})$.

2) On sait déjà que le foncteur canonique 
$(\hdag T) \circ \mathrm{oub} _T \to Id$ est un isomorphisme. 
Réciproquement, il est tautologique que, pour tout 
$\E \in \smash{\underset{\longrightarrow}{LD}} ^\mathrm{b} _{\Q,\mathrm{qc}, T}  
(\widehat{\D} ^{(\bullet)} _{\X})$, le morphisme canonique 
$\E \to \mathrm{oub} _T \circ (\hdag T) (\E)$ est un isomorphisme.
\end{proof}

\begin{vide}
  [Théorème de Berthelot-Kashiwara]
\label{Berthelot-Kashiwara}
La version arithmétique due à Berthelot du théorème de Kashiwara s'écrit de la façon suivante (les versions surcohérentes ou surholonomes se déduisent de manière immédiate du cas cohérent) :
soient $u\colon \X \hookrightarrow \PP$ une immersion fermée de $\V$-schémas formels lisses, $T$ un diviseur de $X$ tel que $T \cap X$ soit un diviseur de $X$.

\begin{enumerate}
    \item  Pour tout $\D ^\dag _{\PP} (\hdag T) _{\Q}$-module cohérent
    $\E$ à support dans $X$,
    pour tout $\D ^\dag _{\X } (\hdag T \cap X) _{ \Q}$-module cohérent $\FF$,
    pour tout entier $j \not =0$, 
    $\mathcal{H} ^j {u } _+ (\E ) =0$ et $\mathcal{H} ^j  u  ^!(\FF) =0$.
    \item Posons
    $\mathfrak{C}=(F\text{-})\mathrm{Coh}$
    ou 
$\mathfrak{C}=(F\text{-})\mathrm{Surcoh}$
ou 
$\mathfrak{C}=(F\text{-})\mathrm{Surhol}$
ou 
$\mathfrak{C}=
(F\text{-})D ^\mathrm{b} _\mathrm{coh} $
ou 
$\mathfrak{C}=
(F\text{-})D ^\mathrm{b} _\mathrm{surcoh} $
ou 
$\mathfrak{C}=
(F\text{-})D ^\mathrm{b} _\mathrm{surhol} $.
Les foncteurs ${u } _+$ et $ u  ^!$ induisent des équivalences quasi-inverses entre 
$\mathfrak{C} (\PP, T, X/K)$ et $\mathfrak{C} (\X, T\cap X,  X/K)$.
      \end{enumerate}

\end{vide}

\subsection{Recollement d'isocristaux: cas de la compactification partielle lisse}

Nous donnerons dans la section \ref{section-analog-m} une variante de niveau $m$ (fixé) des procédures de recollement de \cite[2.5]{caro-construction}. Comme nous procéderons par analogie et comme nous utiliserons aussi le cas où le niveau n'est pas fixé (i.e. la situation de \cite[2.5]{caro-construction}), rappelons-en quelques points. 
Dans cette section, nous fixerons les notations suivantes:
\begin{vide}
[Notations]
\label{nota-cohPXT}
Soient $\PP $ un $\V$-schéma formel séparé et lisse,
$T$ un diviseur de $P$, 
$X$ une sous-variété lisse fermée de $P$ tels que $Y := X \setminus T$ soit dense dans $X$ (i.e. $T\cap X$ est un diviseur de $X$).

Fixons $(\PP _{\alpha}) _{\alpha \in \Lambda}$ un recouvrement d'ouverts de $\PP$.
On note $\PP _{\alpha \beta}:= \PP _\alpha \cap \PP _\beta$,
$\PP _{\alpha \beta \gamma}:= \PP _\alpha \cap \PP _\beta \cap \PP _\gamma$,
$X _\alpha := X \cap P _\alpha$,
$X_{\alpha \beta } := X _\alpha \cap X _\beta$ et
$X_{\alpha \beta \gamma } := X _\alpha \cap X _\beta \cap X _\gamma $.
On suppose de plus que pour tout $\alpha\in \Lambda$, $X _\alpha$ est affine
(par exemple lorsque le recouvrement $(\PP _{\alpha}) _{\alpha \in \Lambda}$ est affine).
Comme $P$ est  séparé, pour tous $\alpha,\beta ,\gamma \in \Lambda$,
$X_{\alpha \beta }$ et $X_{\alpha \beta \gamma }$ sont donc affines.

Pour tout triplet $(\alpha, \, \beta,\, \gamma)\in \Lambda ^3$, choisissons
$\X _\alpha$ (resp. $\X _{\alpha \beta}$, $\X _{\alpha \beta \gamma}$)
des $\V$-schémas formels lisses relevant $X _\alpha$
(resp. $X _{\alpha \beta}$, $X _{\alpha \beta \gamma}$),
$p _1 ^{\alpha \beta}\colon \X  _{\alpha \beta} \rightarrow \X _{\alpha}$
(resp. $p _2 ^{\alpha \beta}\colon \X  _{\alpha \beta} \rightarrow \X _{\beta}$)
des relèvements de
$X  _{\alpha \beta} \rightarrow X _{\alpha}$
(resp. $X  _{\alpha \beta} \rightarrow X _{\beta}$).
Rappelons que grâce à Elkik (\cite{elkik} de tels relèvements existent bien.

De même, pour tout triplet $(\alpha,\,\beta,\,\gamma )\in \Lambda ^3$, on choisit des relèvements
$p _{12} ^{\alpha \beta \gamma}\colon \X  _{\alpha \beta \gamma} \rightarrow \X  _{\alpha \beta} $,
$p _{23} ^{\alpha \beta \gamma}\colon \X  _{\alpha \beta \gamma} \rightarrow \X  _{\beta \gamma} $,
$p _{13} ^{\alpha \beta \gamma}\colon \X  _{\alpha \beta \gamma} \rightarrow \X  _{\alpha \gamma} $,
$p _1 ^{\alpha \beta \gamma}\colon \X  _{\alpha \beta \gamma} \rightarrow \X  _{\alpha} $,
$p _2 ^{\alpha \beta \gamma}\colon \X  _{\alpha \beta \gamma} \rightarrow \X  _{\beta} $,
$p _3 ^{\alpha \beta \gamma}\colon \X  _{\alpha \beta \gamma} \rightarrow \X  _{\gamma} $
induisant les morphismes canoniques au niveau des fibres spéciales.
\end{vide}

\begin{defi}
Pour tout $\alpha \in \Lambda$, donnons-nous $\E _\alpha$,
un $(F\text{-})\smash{\D} ^{\dag} _{\X _{\alpha} } (\hdag T  \cap X _{\alpha}) _{\Q}$-module cohérent.
Une \textit{donnée de recollement} sur la famille $(\E _{\alpha})_{\alpha \in \Lambda}$ est 
la donnée pour tous $\alpha,\,\beta \in \Lambda$ d'un isomorphisme
$(F\text{-})\smash{\D} ^{\dag} _{\X _{\alpha \beta} }(\hdag T  \cap X _{\alpha \beta}) _{ \Q}$-linéaire de la forme
$ \theta _{  \alpha \beta} \colon   p _2  ^{\alpha \beta !} (\E _{\beta}) \riso p  _1 ^{\alpha \beta !} (\E _{\alpha}),$
ceux-ci vérifiant la condition de cocycle :
$\theta _{13} ^{\alpha \beta \gamma }=
\theta _{12} ^{\alpha \beta \gamma }
\circ
\theta _{23} ^{\alpha \beta \gamma }$,
où $\theta _{12} ^{\alpha \beta \gamma }$, $\theta _{23} ^{\alpha \beta \gamma }$
et $\theta _{13} ^{\alpha \beta \gamma }$ sont définis par les diagrammes commutatifs
\begin{equation}
  \label{diag1-defindonnederecol}
\xymatrix  @R=0,3cm {
{  p _{12} ^{\alpha \beta \gamma !} p  _2 ^{\alpha \beta !}  (\E _\beta )}
\ar[r] ^-{\tau} _{\sim}
\ar[d] ^-{p _{12} ^{\alpha \beta \gamma !} (\theta _{\alpha \beta})} _{\sim}
&
{p _2 ^{\alpha \beta \gamma!}  (\E _\beta )}
\ar@{.>}[d] ^-{\theta _{12} ^{\alpha \beta \gamma }}
\\
{ p _{12} ^{\alpha \beta \gamma !}  p  _1 ^{\alpha \beta !}  (\E _\alpha)}
\ar[r]^{\tau} _{\sim}
&
{p _1 ^{\alpha \beta \gamma!}(\E _\alpha),}
}
%
%
\xymatrix  @R=0,3cm {
{  p _{23} ^{\alpha \beta \gamma !} p  _2 ^{\beta \gamma!}  (\E _\gamma )}
\ar[r] ^-{\tau} _{\sim}
\ar[d] ^-{p _{23} ^{\alpha \beta \gamma !} (\theta _{ \beta \gamma})} _{\sim}
&
{p _3 ^{\alpha \beta \gamma!}  (\E _\gamma )}
\ar@{.>}[d] ^-{\theta _{23} ^{\alpha \beta \gamma }}
\\
{ p _{23} ^{\alpha \beta \gamma !}  p  _1 ^{ \beta \gamma !}  (\E _\beta)}
\ar[r]^{\tau} _{\sim}
&
{p _2 ^{\alpha \beta \gamma!}(\E _\beta),}
}
%
%
\xymatrix  @R=0,3cm {
{  p _{13} ^{\alpha \beta \gamma !} p  _2 ^{\alpha \gamma !}  (\E _\gamma )}
\ar[r] ^-{\tau} _{\sim}
\ar[d] ^-{p _{13} ^{\alpha \beta \gamma !} (\theta _{\alpha \gamma})} _{\sim}
&
{p _3 ^{\alpha \beta \gamma!}  (\E _\gamma )}
\ar@{.>}[d]^{\theta _{13} ^{\alpha \beta \gamma }}
\\
{ p _{13} ^{\alpha \beta \gamma !}  p  _1 ^{\alpha \gamma !}  (\E _\alpha)}
\ar[r]^{\tau} _{\sim}
&
{p _1 ^{\alpha \beta \gamma!}(\E _\alpha),}
}
\end{equation}
où les isomorphismes de la forme $\tau $ désignent les isomorphismes canoniques de recollement  (voir \ref{nota-tau}).
\end{defi}

\begin{defi}
\label{defindonnederecoldag}
On construit la catégorie $(F\text{-})\mathrm{Coh} (X,\, (\X _\alpha) _{\alpha \in \Lambda},\, T\cap X/K)$ (resp. $(F\text{-})\mathrm{Isoc} ^{\dag \dag} (X,\, (\X _\alpha) _{\alpha \in \Lambda},\, T\cap X/K)$) de la manière
suivante :

\begin{itemize}
\item Un objet est une famille $(\E _\alpha) _{\alpha \in \Lambda}$
de $(F\text{-})\smash{\D} ^{\dag} _{\X _{\alpha} } (\hdag T  \cap X _{\alpha}) _{\Q}$-modules cohérents (resp. de $(F\text{-})\smash{\D} ^{\dag} _{\X _{\alpha} } (\hdag T  \cap X _{\alpha}) _{\Q}$-modules cohérents, $\O _{\X _{\alpha} } (\hdag T  \cap X _{\alpha}) _{\Q}$-cohérents)
munie d'une donnée de recollement $ (\theta _{\alpha\beta}) _{\alpha ,\beta \in \Lambda}$.

\item Un morphisme
$((\E _{\alpha})_{\alpha \in \Lambda},\, (\theta _{\alpha\beta}) _{\alpha ,\beta \in \Lambda})
\rightarrow
((\E ' _{\alpha})_{\alpha \in \Lambda},\, (\theta '_{\alpha\beta}) _{\alpha ,\beta \in \Lambda})$
est une famille de morphismes $(F\text{-})\smash{\D} ^{\dag} _{\X _{\alpha} } (\hdag T  \cap X _{\alpha}) _{\Q}$-linéaire $f _\alpha\colon \E _\alpha \rightarrow \E '_\alpha$
commutant aux données de recollement,
i.e., telle que le diagramme suivant soit commutatif :
\begin{equation}
  \label{diag2-defindonnederecol}
\xymatrix  @R=0,3cm {
{ p _2  ^{\alpha \beta !} (\E _{\beta}) }
\ar[d] ^-{p _2  ^{\alpha \beta !} (f _{\beta}) }
\ar[r] ^-{\theta _{\alpha\beta}} _{\sim}
&
{  p  _1 ^{\alpha \beta !} (\E _{\alpha}) }
\ar[d] ^-{p  _1 ^{\alpha \beta !} (f _{\alpha})}
\\
{p _2  ^{\alpha \beta !} (\E '_{\beta})  }
\ar[r]^{\theta '_{\alpha\beta}} _{\sim}
&
{ p  _1 ^{\alpha \beta !} (\E '_{\alpha})  .}
}
\end{equation}

\end{itemize}
\end{defi}

\begin{vide}
[Recollement : cas de la compactification partielle lisse]
\label{prop-donnederecol-dag}
$\bullet$ Nous gardons les notations de \ref{defindonnederecoldag}.
D'après \cite[2.5.4]{caro-construction}, la catégorie $(F\text{-})\mathrm{Coh} (\PP, T,X/K)$ (voir la définition de \ref{nota-surhol-T-pre}) est canoniquement
isomorphe à la catégorie $(F\text{-})\mathrm{Coh} (X,\, (\X _\alpha) _{\alpha \in \Lambda},\, T\cap X)$.
Plus précisément, d'après la preuve de \cite[2.5.4]{caro-construction}, on dispose des deux foncteurs quasi-inverses canoniques
\begin{gather}
\label{def-Loc-dag}
\mathcal{L}oc\, :
(F\text{-})\mathrm{Coh} (\PP, T,X/K)
\rightarrow
(F\text{-})\mathrm{Coh} (X,\, (\X _\alpha) _{\alpha \in \Lambda},\, T\cap X/K), 
\\
\label{def-Recol-dag}
\mathcal{R}ecol
\colon 
(F\text{-})\mathrm{Coh} (X,\, (\X _\alpha) _{\alpha \in \Lambda},\, T\cap X/K)
\rightarrow
(F\text{-})\mathrm{Coh} (\PP, T,X/K),
\end{gather}
dont les constructions sont analogues à celles que l'on donnera au niveau $m$ dans 
\ref{defi-loc} et \ref{defi-reloc}.

$\bullet$ La catégorie $(F\text{-})\mathrm{Isoc} ^{\dag} (Y,X/K)$ des $(F\text{-})$isocristaux surconvergents sur $(Y,X)/K$ est canoniquement équivalente à $(F\text{-})\mathrm{Isoc} ^{\dag \dag} (X,\, (\X _\alpha) _{\alpha \in \Lambda},\, T\cap X/K)$. En composant cette équivalence avec le foncteur $\mathcal{R}ecol$ de \ref{def-Recol-dag}, on obtient alors le foncteur pleinement fidèle $\sp _{X\hookrightarrow \PP, T,+}\colon  
(F\text{-})\mathrm{Isoc} ^{\dag} (Y,X/K) \to  (F\text{-})\mathrm{Coh} (\PP, T,X/K)$ (voir \cite[2.5]{caro-construction}). 
Grâce à \cite[6.1.4]{caro_devissge_surcoh}, on vérifie que son image essentielle est égale à 
$(F\text{-})\mathrm{Isoc} ^{\dag \dag} (\PP, T, X/K)$ (notations de \ref{defi-part-surcoh}).
\end{vide}

\section{Sur la préservation de la surcohérence par image directe}
Le but de ce chapitre est de démontrer le théorème \ref{surhol-conjA}. 
Nous traitons d'abord dans une première partie le cas de la compactification partielle (i.e. $X$) lisse. Dans la deuxième partie, on s'y ramène dans le cas des isocristaux surcohérents par descente génériquement finie et étale. 
Le cas général s'établit alors par dévissage en isocristaux surcohérents. 
\subsection{Indépendance: cas de la compactification partielle lisse}

Dans le cas de la compactification lisse, on peut utiliser le théorème de Berthelot-Kashiwara (et on étend ainsi \cite[3.2.4]{caro_surcoherent}
pour la surcohérence {\og faible\fg}), 
ce qui permet d'obtenir le lemme suivant fondateur.

\begin{lemm}
[Indépendance par rapport au diviseur]
\label{lemme-coh-PXTindtP}
Soient $\PP$ un $\V$-schéma formel séparé et lisse, $T$ et $ T'$ deux diviseurs de $P$,
$X$ un sous-schéma fermé lisse de $P$
tel que $T\cap X= T'\cap X$. 
Posons
    $\mathfrak{C}=(F\text{-})\mathrm{Coh}$
    ou 
$\mathfrak{C}=(F\text{-})\mathrm{Surcoh}$
ou 
$\mathfrak{C}=(F\text{-})\mathrm{Surhol}$
ou 
$\mathfrak{C}=
(F\text{-})D ^\mathrm{b} _\mathrm{coh} $
ou 
$\mathfrak{C}=
(F\text{-})D ^\mathrm{b} _\mathrm{surcoh} $
ou 
$\mathfrak{C}=
(F\text{-})D ^\mathrm{b} _\mathrm{surhol} $.
On obtient alors les égalités 
$\mathfrak{C}(\PP, T, X/K)=\mathfrak{C} (\PP, T', X/K)$.
\end{lemm}

\begin{proof}
Considérons d'abord le cas des modules.
Le cas cohérent 
a été traitée dans \cite[5.1.3]{caro-pleine-fidelite}. 
En utilisant respectivement les versions surcohérentes ou surholonomes du théorème de Berthelot-Kashiwara (voir \ref{Berthelot-Kashiwara}), 
on procède de manière identique pour les autres cas. Le cas des complexes se déduit du cas des modules (e.g. on se rappelle que le foncteur extension $(\hdag D)$ est exact pour les modules cohérents lorsque $D$ est un diviseur). 
\end{proof}

\begin{lemm}
[Indépendance par rapport au schéma formel: cas des modules]
\label{coh-PXTindtP}
  Soit $(f,Id,Id) \colon  (\PP', T',X,Y)\to (\PP, T,X,Y)$ un morphisme de $d$-cadres avec $X$ lisse.

\begin{enumerate}
\item 
\label{coh-PXTindtP-i} 
Pour tout $\E \in (F\text{-})\mathrm{Surcoh}(\PP, T, X/K)$,
pour tout $\E '\in (F\text{-})\mathrm{Surcoh} (\PP', T', X/K)$,
pour tout $j \in \Z\setminus \{0\}$,
$$\mathcal{H} ^j (\R \underline{\Gamma} ^\dag _X f ^! (\E) ) =0,
\mathcal{H} ^j (f_+(\E')) =0.$$
De même en remplaçant {\og $\mathrm{Surcoh}$\fg} 
par {\og $\mathrm{Surhol}$\fg}.

\item 
\label{coh-PXTindtP-ii} 
Les foncteurs $\R \underline{\Gamma} ^\dag _X  f ^! $ et
$f _+$ induisent alors des équivalences quasi-inverses 
entre les catégories $(F\text{-})\mathrm{Surcoh} (\PP, T, X/K)$ et $(F\text{-})\mathrm{Surcoh} (\PP', T', X/K)$.
De même en remplaçant {\og $\mathrm{Surcoh}$\fg} 
par {\og $\mathrm{Surhol}$\fg}.
\end{enumerate}

\end{lemm}

\begin{proof}
Via \ref{lemme-coh-PXTindtP},
la preuve est identique à celle de \cite[5.1.4]{caro-pleine-fidelite}.
\end{proof}

Pour les complexes, la situation est plus délicate car on ne peut plus recoller pour valider les équivalences quasi-inverses: 

\begin{lemm}
[Indépendance par rapport au schéma formel: cas des complexes]
Soit $(f,Id,Id) \colon  (\PP', T',X,Y)\to (\PP, T,X,Y)$ un morphisme de $d$-cadres avec $X$ lisse. 
\label{coh-PXTindtPbis}
\begin{enumerate}

\item 
\label{coh-PXTindtP-iii}
On dispose des factorisations
$\R \underline{\Gamma} ^\dag _X  f ^! 
\colon 
(F\text{-})D ^\mathrm{b} _\mathrm{coh}  (\PP, T, X/K)
 \to
(F\text{-})D ^\mathrm{b} _\mathrm{coh}  (\PP', T', X/K)$
et
$f _+\colon 
(F\text{-})D ^\mathrm{b} _\mathrm{coh}  (\PP', T', X/K)
 \to
(F\text{-})D ^\mathrm{b} _\mathrm{coh}  (\PP, T, X/K)$.
De même en remplaçant l'indice {\og coh\fg} 
par respectivement {\og surcoh\fg} ou
{\og surhol\fg} ou {\og isoc\fg}.

\item 
\label{coh-PXTindtP-iv} 
Lorsque $f$ est soit propre ou soit une immersion ouverte, les factorisations $\R \underline{\Gamma} ^\dag _X  f ^! $ et
$f _+$ de \ref{coh-PXTindtP-iii} induisent alors des équivalences quasi-inverses.\end{enumerate}
\end{lemm}

\begin{proof}
Afin d'alléger les notations, le cas avec structure de Frobenius étant analogue, nous omettrons d'indiquer {\og $(F\text{-})$\fg} dans toutes les catégories. De plus, comme $P$, $P'$, $X$ sont lisses, 
on peut les supposer intègres. Le cas où $Y$ est vide implique que les catégories sont nulles. 
On peut en outre supposer $Y$ dense dans $X$. 
Dans ce cas, $f ^{-1}( T)$ est un diviseur de $P'$. 
Par \ref{lemme-coh-PXTindtP}, comme $f ^{-1}( T) \cap X = T' \cap X$, on se ramène au cas
$T' = f ^{-1}( T)$.

I) Vérifions à présent \ref{coh-PXTindtPbis}.\ref{coh-PXTindtP-iii}.
Soient $\E \in D ^\mathrm{b} _\mathrm{coh}  (\PP, T, X/K)$
et $\E ' \in D ^\mathrm{b} _\mathrm{coh}  (\PP', T', X/K)$.
Le fait que 
$\R \underline{\Gamma} ^\dag _X  f ^! (\E) 
\in 
D ^\mathrm{b} _\mathrm{coh}  (\PP', T', X/K)$ est local en $\PP'$ et
que $f _{+} (\E') \in D ^\mathrm{b} _\mathrm{coh}  (\PP, T, X/K)$ est local en $\PP$.
Dans les deux cas, on peut supposer $X$ affine et lisse. 
Il existe alors une immersion fermée de $\V$-schémas formels lisses de la forme
$u'\colon  \X \hookrightarrow \PP'$ relevant $X \hookrightarrow \PP'$. Posons $u:=f  \circ u '$.
Dans ce cas, 
$\R \underline{\Gamma} ^\dag _X  f ^! (\E) 
\riso u ' _{+} \circ u ^{\prime !}\circ  f ^! (\E) 
\riso u ' _{+} \circ u ^{ !} (\E) $.
Comme $\E$ est à support dans $X$, d'après le théorème de Berthelot-Kashiwara, 
$u ^{ !} (\E) $ est cohérent.
Il en résulte $\R \underline{\Gamma} ^\dag _X  f ^! (\E) 
\in 
D ^\mathrm{b} _\mathrm{coh}  (\PP', T', X/K)$. 
Enfin, le théorème de Berthelot-Kashiwara donne que
$u ^{\prime !} (\E')$ est cohérent et
$\E' \riso u ' _{+} \circ u ^{\prime !} (\E')$.
Il en résulte $f _{+} (\E') \riso u _{+} \circ u ^{\prime !} (\E')$
et donc que $f _{+} (\E')  \in D ^\mathrm{b} _\mathrm{coh}  (\PP, T, X/K)$.
On procède de manière identique pour les cas respectives.

II) Traitons à présent \ref{coh-PXTindtPbis}.\ref{coh-PXTindtP-iv} dans le cas non respectif.

1) Vérifions d'abord \ref{coh-PXTindtP-iv} lorsque $f$ est propre. 

1.i) Comme $f$ est propre, 
pour tout complexe $\E$ de $D ^\mathrm{b} _\mathrm{coh}  (\PP, T, X/K)$,
on dispose par adjonction du morphisme canonique
$f _+ \circ \R \underline{\Gamma} ^\dag _X  \circ f ^! (\E) \rightarrow \E$ (en effet, il s'agit d'appliquer \cite[1.2.10]{caro_courbe-nouveau} pour les complexes 
parfaits $\E$ et $\R \underline{\Gamma} ^\dag _X  \circ f ^! (\E) $ au morphisme canonique $\R \underline{\Gamma} ^\dag _X  \circ f ^! (\E) \rightarrow f ^! (\E) $).
Le fait que ce morphisme soit un isomorphisme est local en $\PP$. 
On peut donc supposer $\PP$ affine (et donc $X$). 
Il existe alors une immersion fermée de $\V$-schémas formels lisses de la forme
$u'\colon  \X \hookrightarrow \PP'$ relevant $X \hookrightarrow \PP'$. Posons $u:=f  \circ u '$.
Comme le morphisme canonique $\R \underline{\Gamma} ^\dag _{X }\to Id$ est canoniquement isomorphe au morphisme d'adjonction $u' _{ +} \circ u_{ } ^{\prime!} \to Id$ (voir \ref{Gamma=u+circu!}), 
on obtient par transitivité du morphisme d'adjonction que 
$f _+ \circ \R \underline{\Gamma} ^\dag _X  \circ f ^! (\E) \rightarrow \E$ est canoniquement isomorphe 
au morphisme d'adjonction 
$u_{ +} \circ u  ^{!} (\E) \to \E $.
Via le théorème de Kashiwara-Berthelot (voir \ref{Berthelot-Kashiwara}), celui-ci est un isomorphisme. 
Ainsi, le morphisme
$f _+ \circ \R \underline{\Gamma} ^\dag _X  \circ f ^! (\E) \rightarrow \E$
fonctoriel en $\E$ est un isomorphisme. 

1.ii) Réciproquement, 
soit $\E' \in D ^\mathrm{b} _\mathrm{coh}  (\PP', T', X/K)$.
Comme $f$ est propre, on dispose par adjonction du morphisme canonique 
$\E'\to   f ^!  \circ f _+  (\E') $. En lui appliquant le foncteur $\R \underline{\Gamma} ^\dag _X $, comme $\E'$ est à support dans $X$, on obtient le morphisme canonique
$\E'\to \R \underline{\Gamma} ^\dag _X\circ   f ^! \circ f _+  (\E') $. Le fait que ce morphisme soit un isomorphisme est local. 
On peut donc supposer qu'il existe une immersion fermée de $\V$-schémas formels lisses de la forme
$u'\colon  \X \hookrightarrow \PP'$ relevant $X \hookrightarrow \PP'$ et on pose $u:=f  \circ u '$.
Via le théorème de Berthelot-Kashiwara (voir \ref{Berthelot-Kashiwara}), 
il existe $\FF \in D ^\mathrm{b} _\mathrm{coh}  (\X, T\cap T, X/K)$ tel que 
$\E' \riso u _{+} ' (\FF)$. On ramène donc à supposer $\E' =u _{+} ' (\FF)$.
Par transitivité des morphismes d'adjonction de la forme $Id \to f ^! f _+$, 
on vérifie la commutativité du carré du diagramme ci-dessous
\begin{equation}
\notag
\xymatrix{
{u ' _+ (\FF) } 
\ar@{=}[dr] 
& 
{u ' _+ u ^{\prime !}  u ' _+ (\FF)  } 
\ar[l] _-{\mathrm{adj}}
\ar[r] ^-{\mathrm{adj}}
&
{u ' _+ u ^{\prime !} f ^! f _+ u ' _+ (\FF)  } 
\ar[d] ^-{\sim}
\\ 
{}
&
{u ' _+ (\FF) }  
\ar[r] ^-{\mathrm{adj}} _-{\sim}
\ar[u] ^-{\mathrm{adj}} _-{\sim}
& 
{u ' _+ u ^{!}  u _+ (\FF)  .} 
}
\end{equation}
La commutativité du triangle résulte des propriétés des morphismes d'adjonction. 
Le morphisme composé du haut est par définition 
le morphisme $\E'\to \R \underline{\Gamma} ^\dag _X\circ   f ^! \circ f _+  (\E') $.
Celui-ci est donc un isomorphisme.

iii) Via i) et ii), on a donc établit que 
les foncteurs 
$\R \underline{\Gamma} ^\dag _X  f ^! $ 
et
$f _+$ entre les catégories
$D ^\mathrm{b} _\mathrm{coh}  (\PP, T, X/K)$
et 
$D ^\mathrm{b} _\mathrm{coh}  (\PP', T', X/K)$
sont quasi-inverses. 

2) Supposons que $f$ soit une immersion ouverte.  
Soient $\E \in D ^\mathrm{b} _\mathrm{coh}  (\PP, T, X/K)$,
$\E' \in D ^\mathrm{b} _\mathrm{coh}  (\PP', T', X/K)$.
Grâce à \ref{coh-PXTindtP}.\ref{coh-PXTindtP-i} qui donne l'acyclicité des 
foncteurs $\R \underline{\Gamma} ^\dag _X\circ   f ^!$ et $ f _+$, 
on obtient les isomorphismes canoniques $\R \underline{\Gamma} ^\dag _X\circ   f ^! (\E) \riso f ^{*} (\E)= f ^{-1} (\E)$ et 
$f _{+} (\E')\riso f _*(\E')$ (i.e., il est inutile de dériver le foncteur $f _*$).
Or, on dispose des morphismes d'adjonction :
$\E \to f _{*} f ^{-1} (\E)$ et $f ^{-1}f _{*} (\E') \to  \E'$. Il est immédiat que le morphisme 
$f ^{-1}f _{*} (\E') \to   \E'$ est un isomorphisme. 
Pour établir que les foncteur $f _{*}$ et $f ^{-1}$ induisent des équivalences quasi-inverses entre $D ^\mathrm{b} _\mathrm{coh}  (\PP, T, X/K)$
et $D ^\mathrm{b} _\mathrm{coh}  (\PP', T', X/K)$, 
il suffit de vérifier que 
le morphisme canonique $\E \to f _{*} f ^{-1} (\E)$ est un isomorphisme.
Comme cela est local en $\PP$, on peut donc supposer  
qu'il existe une immersion fermée de $\V$-schémas formels lisses de la forme
$u'\colon  \X \hookrightarrow \PP'$ relevant $X \hookrightarrow \PP'$. Posons $u:=f  \circ u '$.
D'après le théorème Berthelot-Kashiwara (voir \ref{Berthelot-Kashiwara}), il suffit d'établir que $\E  \to f _{*} f  ^{-1} (\E  )$ devient un isomorphisme après application du fonceur $u ^{ !}$. Comme c'est le cas après application de $f  ^{-1}$ et comme
$u  ^! \riso  u ^{\prime ! }  \circ  f ^{-1} $, on en déduit le résultat.

III) Comme les catégories des cas respectifs sont des sous-catégories pleines de celles du cas non respectif, avec de plus \ref{coh-PXTindtPbis}.\ref{coh-PXTindtP-iii}, les cas respectifs de \ref{coh-PXTindtPbis}.\ref{coh-PXTindtP-iv} résulte du cas non respectif. 
\end{proof}

\begin{conj}
\label{conj-Db(coh)}
Il est conjectural que le foncteur canonique 
$(F\text{-})D ^\mathrm{b}  (\mathrm{Coh} (\PP, T, X/K))\to 
(F\text{-})D ^\mathrm{b} _\mathrm{coh}  (\PP, T, X/K)$ 
soit une équivalence de catégories. 
Si telle était le cas, alors l'hypothèse que $f $ soit propre ou soit une immersion ouverte
dans \ref{coh-PXTindtPbis}.\ref{coh-PXTindtP-iv} deviendrait superflue.
Nous pourrions alors définir les catégories correspondantes de \ref{nota-Dsurcv-cpart}
en remplaçant l'hypothèse {\og proprement $d$-réalisable \fg} 
par {\og $d$-réalisable \fg}.

\end{conj}

\subsection{Isocristaux surcohérents: descente au cas de la compactification partielle lisse}

Le lemme \ref{casliss631dev} ci-dessous est une légère extension du lemme \cite[5.3.1.2]{caro-pleine-fidelite}.
Ce lemme \ref{casliss631dev} sera utile afin d'obtenir le théorème le théorème \ref{surhol-conjA}.
Ce dernier nous permettra ensuite d'établir le théorème \ref{b+proprelisse-varlisse-gen}.

\begin{lemm}
\label{casliss631dev}
  Soient $f\colon \PP' \rightarrow \PP $ un morphisme de $\V$-schémas formels séparés et lisses,
  $u ' \colon X'\hookrightarrow P'$ une immersion fermée avec $X'$ intègre, 
$T'$ un diviseur de $P'$ ne contenant pas $X'$ tel que $Y':= X' \setminus T'$ soit lisse.
  Supposons $f  \circ u' $ propre.
Soit $\E' \in (F\text{-})D ^\mathrm{b} _\mathrm{isoc} (\PP', T', X'/K)$.

\begin{enumerate}
\item \label{casliss631dev1} 
Soit $D$ est un diviseur de $P$ tel que 
$D':=f ^{-1}(D)$ soit un diviseur de $P'$ inclus dans $T'$. 
Si $\E' $ est en outre $\D ^\dag _{\PP'} (\hdag D') _{\Q}$-surcohérent alors 
$f _+ (\E') \in (F\text{-})D ^\mathrm{b} _\mathrm{surcoh} (\D ^\dag _{\PP} (\hdag D) _{\Q})$.

\item \label{casliss631dev2}
Si $\E' \in (F\text{-})D ^\mathrm{b} _\mathrm{surhol} (\D ^\dag _{\PP',\Q})$ (e.g. si $T'$ est vide d'après la remarque \ref{rema-cv-surhol})
alors 
$f _+ (\E') \in (F\text{-})D ^\mathrm{b} _\mathrm{surhol} (\D ^\dag _{\PP,\Q})$.

\end{enumerate}

\end{lemm}

\begin{proof}
Commençons par quelques éléments valables dans les deux cas. 
D'après \cite[5.3.1]{caro-pleine-fidelite}, il résulte du théorème de désingularisation de de Jong
qu'il existe un diagramme commutatif de la forme
  \begin{equation}
  \label{diag-casliss631dev}
  \xymatrix {
  {X''} \ar[r] ^-{u''} \ar[d] _{a '}  &
  {\P ^N _{\PP'}} \ar[r] ^{\P ^N _f} \ar[d] ^-{q'} & {\P ^N _{\PP}} \ar[d]^-{q} \\
  {X'} \ar[r] ^-{u'}  & {\PP'} \ar[r] ^-f & {\PP,}
   }
\end{equation}
  où $X''$ est lisse sur $k$, $q$ et $q'$ sont les projections canoniques, $u''$ est une immersion fermée, 
  $a  ^{\prime -1} (T'\cap X')$ est un diviseur à croisements normaux strict de $X''$,
  $a '$ est propre, surjectif, génériquement fini et étale. 
  Posons $T'':= q ^{\prime -1} (T')$, 
$\widetilde{\PP}:=\P ^N _{\PP}$.
Plus précisément, d'après la preuve de \cite[5.3.1]{caro-pleine-fidelite}, 
il existe une immersion $\iota\colon X''\hookrightarrow \P ^N$ tel que $u''=( \iota , u' \circ a)$.
On remarque alors que $\P ^N _f \circ u''= (\iota, f \circ u' \circ a)$. 
Comme $q$ et $f \circ u' \circ a$ sont propres, on en déduit que l'immersion $\P ^N _f \circ u''$ est fermée.

Traitons à présent le premier cas. Comme la surcohérence est préservée par les foncteurs $\mathcal{H} ^{l}$ avec $l$ un entier, 
quitte à utiliser la deuxième suite spectrale d'hypercohomologie du foncteur
$f _+$, on peut supposer que $\E'$ est un module.
D'après  \cite[5.3.1]{caro-pleine-fidelite}, 
si on pose
$\E'':= \R \underline{\Gamma} _{X''} ^\dag \circ q ^{\prime !} (\E')$
alors 
$\E''\in (F\text{-})\mathrm{Isoc} ^{\dag \dag} (\P ^N _{\PP'}, T'', X''/K)$ et 
$\E'$ est un facteur direct de $q ' _+(\E'')$.
Cela implique 
$f _+ (\E')$ est un facteur direct de
$f _+  \circ q '_+ (\E'')$.
On se ramène ainsi à étudier les propriétés de 
$f _+  \circ q '_+ (\E'')$.
Or, par transitivité de l'image directe 
$f _+  \circ q '_+ (\E'') \riso q _+ \circ \phi _+  (\E'')$, avec $\phi :=\P ^N _f$.
Posons $D'':= q ^{\prime -1} (D')$, $\widetilde{D}:=q ^{-1} (D)$.
Comme $q $ est propre, $q _+$ préserve la surcohérence.
D'après les préliminaires de la preuve, 
il suffit donc d'établir que $\phi _+  (\E'')$  
est $\D ^\dag _{\widetilde{\PP}} (\hdag \widetilde{D}) _{\Q}$-surcohérent.
Par préservation de la surcohérence par image inverse extraordinaire et foncteur cohomologique local à support strict dans un sous-schéma fermé,
on vérifie que $\E''$ est $\D ^\dag _{\P ^N _{\PP'}} (\hdag D'') _{\Q}$-surcohérent.
Comme $D'' \cap X''=\widetilde{D}\cap X''$, en posant $Y'':= X'' \setminus D''$, 
on dispose ainsi du morphisme de $d$-cadres
$( \phi, Id, Id ) \colon (\P ^N _{\PP'}, D'', X'', Y'') \to (\widetilde{\PP}, \widetilde{D},  X'', Y'') $.
Comme $X''$ est lisse, 
il résulte de \ref{coh-PXTindtP}.\ref{coh-PXTindtP-ii}
que $\phi _+  (\E'')$  
est $\D ^\dag _{\widetilde{\PP}} (\hdag \widetilde{D}) _{\Q}$-surcohérent (à support dans $X$). 
D'où le résultat.

Comme la surholonomie est en particulier 
préservée par image directe par un morphisme propre,
par image inverse extraordinaire et foncteur cohomologique local à support strict dans un sous-schéma fermé, 
on traite le second cas de manière identique.

\end{proof}

\subsection{Stabilité de la surcohérence et surholonomie, conjectures de Berthelot}

\begin{vide}
[Rappel]
  \label{surhol-espcoh}
 Soit $\PP$ un $\V$-schéma formel séparé et lisse.
 Soit $\E \in (F \text{-} )D ^\mathrm{b}  (\smash{\D} ^\dag _{\PP,\Q})$.
 D'après \cite{surcoh-hol}, 
 le complexe $\E$ appartient à $( F \text{-})D ^\mathrm{b} _\mathrm{surhol} (\smash{\D} ^\dag _{\PP,\Q})$ si et seulement si,
pour tout $j\in \Z$, $\mathcal{H} ^j (\E) $ est un $(F\text{-})\smash{\D} ^\dag _{\PP,\Q}$-module
  surholonome. Cela est encore valable en remplaçant la surholonomie par la surcohérence (e.g. voir la fin de \cite[3.1.1]{caro_surcoherent}).
 \end{vide}

\begin{theo}
\label{surhol-conjA}
  Soit $f\colon \PP' \rightarrow \PP$ un morphisme de $\V$-schémas formels séparés et lisses.

\begin{enumerate}
\item    Soit $D$ un diviseur de $P$ tel que $D ' := f ^{-1} (D)$ soit un diviseur de $P'$.  
Si $\E '\in (F\text{-})D ^\mathrm{b} _\mathrm{surcoh} (\smash{\D} ^\dag _{\PP'} (\hdag D') _{\Q})$ est à support propre sur $P$,
  alors $f _+ (\E' ) \in (F\text{-})D ^\mathrm{b} _\mathrm{surcoh} (\smash{\D} ^\dag _{\PP} (\hdag D) _{\Q})$.
\item Si $\E '\in (F\text{-})D ^\mathrm{b} _\mathrm{surhol} (\smash{\D} ^\dag _{\PP',\Q})$ est à support propre sur $P$,
  alors $f _+ (\E' ) \in (F\text{-})D ^\mathrm{b} _\mathrm{surhol} (\smash{\D} ^\dag _{\PP,\Q})$.
\end{enumerate}

\end{theo}

\begin{proof}
Vérifions le premier cas (resp. le second cas).
  Comme il s'agit de reprendre essentiellement la preuve de \cite[8.8]{caro_surholonome},
  nous omettons les vérifications déjà obtenues. Soit $X'$ le support de $\E'$.
Il s'agit de procéder par récurrence sur $\dim X'$.
Lorsque $\dim X' =0$, cela résulte du théorème de Berthelot-Kashiwara de \ref{Berthelot-Kashiwara}
(utilisé pour l'immersion fermée $X ' \hookrightarrow \PP'$ qui se relève).
En utilisant 
des triangles distingués de Mayer-Vietoris,
on se ramène au cas où $X'$ est intègre
(de manière identique au début de l'étape $3$ de la preuve de \cite[8.8]{caro_surholonome}).
De plus, comme la surcohérence (resp. la surholonomie d'après \ref{surhol-espcoh}) est préservée par les foncteurs $\mathcal{H} ^{l}$ avec $l$ un entier, 
quitte à utiliser la deuxième suite spectrale d'hypercohomologie du foncteur
$f _+$, on peut supposer que $\E'$ est un module.
D'après \cite[6.2.1]{caro-pleine-fidelite}, il existe alors un diviseur $T'\supset D'$ de $P'$
tel que $Y':= X'\setminus T'$ soit affine lisse et dense dans $X'$ et
$\E' (\hdag T') \in (F\text{-})\mathrm{Isoc} ^{\dag \dag} (\PP',T', X'/K)$.
Par stabilité de la surcohérence (resp. surholonomie) par le fonceur $(\hdag T')$,
$\E' (\hdag T')$ est
$\D ^\dag _{\PP'} (\hdag D') _{\Q}$-surcohérent (resp. surholonome).
D'après \ref{casliss631dev}, il en résulte alors que $f _+ (\E' (\hdag T'))$ est $\D ^\dag _{\PP} (\hdag D) _{\Q}$-surcohérent
(resp. surholonome).
Or, par hypothèse de récurrence,
$f _+ (\R \underline{\Gamma} ^\dag _{T'} (\E'))$ est $\D ^\dag _{\PP} (\hdag D) _{\Q}$-surcohérent
(resp. surholonome).
 On conclut alors en appliquant $f _+$ au triangle de localisation en $T'$ de $\E'$.

\end{proof}

\begin{rema}
\label{rema-surhol-conjA}
  Nous avions dans la preuve de \cite[8.8]{caro_surholonome} (que l'on a reprise pour établir \ref{surhol-conjA}) 
  d'abord traité le cas où le support est lisse. 
  Par contre dans la preuve de \ref{surhol-conjA}, nous avons éviter de traité ce cas.
  Cela a été possible car ce cas où le support est lisse n'a été utilisé (dans la preuve de \cite[8.8]{caro_surholonome})
que dans deux cas : lorsque $\dim X' =0$ ou lorsque
$\E'(\hdag T')$ est associé à un isocristal surconvergent sur $X' \setminus T'$.
\end{rema}

On obtient grâce à \ref{surhol-conjA} une légère amélioration de \cite[8.8]{caro_surholonome} : 

\begin{coro}
\label{conjB->conjA}
Soient $f\colon \PP'\to \PP$ un morphisme de $\V$-schémas formels séparés et lisses.
Considérons les deux conjectures suivantes de Berthelot (voir \cite[5.3.6]{Beintro2}) : 
\begin{enumerate}
\item[$A)$] 
\label{conjA} 
Si $\E' \in F \text{-}D ^\mathrm{b} _\mathrm{hol} (\smash{\D} ^\dag _{\PP',\Q})$ est à support propre sur $\PP$, alors 
$f_{+} (\E') \in F \text{-}D ^\mathrm{b} _\mathrm{hol} (\smash{\D} ^\dag _{\PP,\Q})$.

\item[$B)$] Si $\E \in F \text{-}D ^\mathrm{b} _\mathrm{hol} (\smash{\D} ^\dag _{\PP,\Q})$, alors 
$f ^{!}(\E) \in F \text{-}D ^\mathrm{b} _\mathrm{hol} (\smash{\D} ^\dag _{\PP',\Q})$.
\end{enumerate}
La conjecture $B)$ implique alors la conjecture $A)$.
\end{coro}

\begin{proof}
D'après \cite[8.3]{caro_surholonome}, la conjecture $B)$ implique que les notions d'holonomie et de surholonomie (avec structure de Frobenius) coïncident. 
Il suffit alors d'appliquer \ref{surhol-conjA}.
\end{proof}

\section{Préservation de la surconvergence}

\subsection{\label{section-analog-m}Recollement de niveau $m$ sur les $k$-variétés lisses réalisables}

Nous reprenons dans cette section les constructions de 
\ref{defindonnederecoldag} et \ref{prop-donnederecol-dag}
en les adaptant au cas des $\D$-modules arithmétiques de niveau $m$ (et en oubliant les diviseurs).
Toutes les constructions et tous les résultats de cette section restent valable avec des structures de Frobenius. 
\\

Nous garderons de plus les notations suivantes :
  \label{notat-construc}
 on se donne $\U$ un $\V$-schéma formel séparé et lisse,
 $Y$ un sous-schéma fermé lisse de $U$
 et $v \colon Y \hookrightarrow U$ l'inclusion canonique.
On fixe $(\U _{\alpha}) _{\alpha \in \Lambda}$ un recouvrement d'ouverts de $\U$.
On note alors $\U _{\alpha \beta}:= \U _\alpha \cap \U _\beta$,
$\U _{\alpha \beta \gamma}:= \U _\alpha \cap \U _\beta \cap \U _\gamma$,
$Y _\alpha := Y \cap U _\alpha$,
$Y_{\alpha \beta } := Y _\alpha \cap Y _\beta$ et
$Y_{\alpha \beta \gamma } := Y _\alpha \cap Y _\beta \cap Y _\gamma $.
On suppose de plus que pour tout $\alpha\in \Lambda$, $Y _\alpha$ soit affine.
Comme $U$ est  séparé, pour tous $\alpha,\beta ,\gamma \in \Lambda$,
$Y_{\alpha \beta }$ et $Y_{\alpha \beta \gamma }$ sont alors affines.

Pour tout triplet $(\alpha, \, \beta,\, \gamma)\in \Lambda ^3$, choisissons
$\Y _\alpha$ (resp. $\Y _{\alpha \beta}$, $\Y _{\alpha \beta \gamma}$)
des $\V$-schémas formels lisses relevant $Y _\alpha$
(resp. $Y _{\alpha \beta}$, $Y _{\alpha \beta \gamma}$).
Soient
$p _1 ^{\alpha \beta}\colon \Y  _{\alpha \beta} \rightarrow \Y _{\alpha}$
(resp. $p _2 ^{\alpha \beta}\colon \Y  _{\alpha \beta} \rightarrow \Y _{\beta}$)
des relèvements de
$Y  _{\alpha \beta} \rightarrow Y _{\alpha}$
(resp. $Y  _{\alpha \beta} \rightarrow Y _{\beta}$).
De même, pour tout triplet $(\alpha,\,\beta,\,\gamma )\in \Lambda ^3$, on choisit des relèvements
$p _{12} ^{\alpha \beta \gamma}\colon \Y  _{\alpha \beta \gamma} \rightarrow \Y  _{\alpha \beta} $,
$p _{23} ^{\alpha \beta \gamma}\colon \Y  _{\alpha \beta \gamma} \rightarrow \Y  _{\beta \gamma} $,
$p _{13} ^{\alpha \beta \gamma}\colon \Y  _{\alpha \beta \gamma} \rightarrow \Y  _{\alpha \gamma} $,
$p _1 ^{\alpha \beta \gamma}\colon \Y  _{\alpha \beta \gamma} \rightarrow \Y  _{\alpha} $,
$p _2 ^{\alpha \beta \gamma}\colon \Y  _{\alpha \beta \gamma} \rightarrow \Y  _{\beta} $,
$p _3 ^{\alpha \beta \gamma}\colon \Y  _{\alpha \beta \gamma} \rightarrow \Y  _{\gamma} $,
$v _{\alpha}\colon \Y _{\alpha } \hookrightarrow \U _{\alpha }$,
$v _{\alpha \beta}\colon \Y _{\alpha \beta} \hookrightarrow \U _{\alpha \beta}$
et
$v _{\alpha \beta \gamma}\colon \Y _{\alpha \beta \gamma } \hookrightarrow \U _{\alpha \beta \gamma}$
induisant les morphismes canoniques au niveau des fibres spéciales.
Tous ces relèvements seront supposés fixés par la suite.

\begin{defi}\label{defindonnederecol}
Pour tout $\alpha \in \Lambda$, donnons-nous $\E _\alpha$,
un $\smash{\widehat{\D}} ^{(m)} _{\Y _{\alpha} }$-module cohérent.
On appelle \textit{donnée de recollement} sur $(\E _{\alpha})_{\alpha \in \Lambda}$,
la donnée pour tous $\alpha,\,\beta \in \Lambda$ d'un isomorphisme
$\smash{\widehat{\D}} ^{(m)} _{\Y _{\alpha \beta}  }$-linéaire
$ \theta _{  \alpha \beta} \colon   p _2  ^{\alpha \beta !} (\E _{\beta}) \riso p  _1 ^{\alpha \beta !} (\E _{\alpha}),$
ceux-ci vérifiant la condition de cocycle :
$\theta _{13} ^{\alpha \beta \gamma }=
\theta _{12} ^{\alpha \beta \gamma }
\circ
\theta _{23} ^{\alpha \beta \gamma }$,
où $\theta _{12} ^{\alpha \beta \gamma }$, $\theta _{23} ^{\alpha \beta \gamma }$
et $\theta _{13} ^{\alpha \beta \gamma }$ sont définis par les diagrammes commutatifs analogues à \ref{diag1-defindonnederecol}. 
\end{defi}

\begin{defi}
\label{defi-coh-m-alpah}

a) On construit la catégorie $\mathrm{Coh} ^{(m)} (Y,\, (\Y _\alpha) _{\alpha \in \Lambda}/\V)$
de la manière suivante :

\begin{itemize}
  \item  un objet est une famille $(\E _\alpha) _{\alpha \in \Lambda}$
de $\smash{\widehat{\D}} ^{(m)} _{\Y _{\alpha} }$-modules cohérents $\E _\alpha$ 
sans $p$-torsion
munie d'une donnée de recollement $ (\theta _{\alpha\beta}) _{\alpha ,\beta \in \Lambda}$,

\item  un morphisme
$((\E _{\alpha})_{\alpha \in \Lambda},\, (\theta _{\alpha\beta}) _{\alpha ,\beta \in \Lambda})
\rightarrow
((\E ' _{\alpha})_{\alpha \in \Lambda},\, (\theta '_{\alpha\beta}) _{\alpha ,\beta \in \Lambda})$
est une famille de morphismes $f _\alpha\colon \E _\alpha \rightarrow \E '_\alpha$
commutant aux données de recollement, i.e., telle que le diagramme analogue à \ref{diag2-defindonnederecol} soit commutatif.

\end{itemize}

b) On désigne par $\mathrm{Cris} ^{(m)} (Y,\, (\Y _\alpha) _{\alpha \in \Lambda}/\V)$ la sous-catégorie pleine de $\mathrm{Coh} ^{(m)} (Y,\, (\Y _\alpha) _{\alpha \in \Lambda}/\V)$
des familles $(\E _\alpha) _{\alpha \in \Lambda}$
de $\smash{\widehat{\D}} ^{(m)} _{\Y _{\alpha} }$-modules cohérents topologiquement quasi-nilpotents $\E _\alpha$
sans $p$-torsion, 
cohérents sur $\O _{\Y _{\alpha} }$.
\end{defi}

\begin{defi}
\label{defi-coh-m-U}
On désigne par $\mathrm{Coh} ^{(m)} (\U,Y/\V)$ (resp. $\mathrm{Cris} ^{(m)} (\U,Y/\V)$) la catégorie des
$\smash{\widehat{\D}} ^{(m)} _{\U}$-modules cohérents $\FF$ 
sans $p$-torsion 
à support dans $Y$ tels que,
pour tout $\alpha \in \Lambda$,
il existe un $\smash{\widehat{\D}} ^{(m)} _{\Y _\alpha}$-module cohérent $\E _\alpha$
(resp. un $\smash{\widehat{\D}} ^{(m)} _{\Y _{\alpha} }$-module cohérent topologiquement quasi-nilpotent $\E _\alpha$
sans $p$-torsion, 
cohérent sur $\O _{\Y _{\alpha} }$)
et un isomorphisme $\smash{\widehat{\D}} ^{(m)} _{\U _\alpha}$-linéaire :
$v_{\alpha+} (\E _\alpha) \riso \FF |\U _\alpha$.
\end{defi}

\begin{rema}
\label{CohUYV-local}
Les deux catégories 
$\mathrm{Coh} ^{(m)} (\U,Y/\V)$ et
 $\mathrm{Cris} ^{(m)} (\U,Y/\V)$
 définies dans \ref{defi-coh-m-U} ne dépendent pas du choix du recouvrement ouvert
$(\U _{\alpha}) _{\alpha \in \Lambda}$ de $\U$ tel que $Y \cap U _\alpha \hookrightarrow \U _\alpha$ se relève, ni de ses relèvements.
En effet, cela résulte du fait que le foncteur image directe est canoniquement indépendant des relèvements (voir \ref{defi-f+-f+(m)}).
De plus, la vérification qu'un $\smash{\widehat{\D}} ^{(m)} _{\U}$-module cohérent à support dans $Y$
appartienne à l'une de ces catégories est locale en $\U$ et se ramène à supposer $\U$ affine et en particulier que $Y\hookrightarrow \U$ se relève (on peut alors choisir le recouvrement avec $\Lambda$ de cardinal $1$ de $\U$).
\end{rema}

\begin{rema}
On dispose d'une équivalence canonique 
entre la catégorie des $m$-cristaux sur $Y$ cohérents sur $\O _{Y/\S}$ et celle des 
$\widehat{\D} ^{(m)} _{\Y}$-modules cohérents  
topologiquement quasi-nilpotents et $\O _{\Y}$-cohérents 
(voir par exemple \cite[3.1]{E-LS2}). Cela justifie la notation $\mathrm{Cris} ^{(m)} (Y,\, (\Y _\alpha) _{\alpha \in \Lambda}/\V)$ de \ref{defi-coh-m-U}.

\end{rema}

\begin{lemm}
\label{F=u!u+F}
Soient $u\colon\Y \hookrightarrow \U$ une immersion fermée de $\V$-schémas formels lisses
et $u _{i}\colon Y _i \hookrightarrow U _i$ l'immersion fermée de $\V/\pi ^{i+1} \V$-schémas lisses induit. 
Soit $\FF$ un $\smash{\widehat{\D}} ^{(m)} _{\Y}$-module cohérent sans $p$-torsion.
On dispose alors de l'isomorphisme canonique
\begin{equation}
\label{F=u!u+F-iso}
\FF 
\riso
\mathcal{H} ^{0} u ^{!} \circ u _{+} (\FF).
\end{equation}

\end{lemm}

\begin{proof}
 Notons $\FF _i := \smash{\D} ^{(m)} _{Y_i} \otimes ^{\L} _{\smash{\widehat{\D}} ^{(m)} _{\Y}}  \FF \riso \FF / \pi ^{i+1}\FF$. 
 On dispose de l'isomorphisme
 $\smash{\D} ^{(m)} _{U_i} \otimes ^{\L} _{\smash{\widehat{\D}} ^{(m)} _{\U}} 
 u _{+} (\FF) 
 \riso
 u _{i+} (\FF _i)$
 (cela découle par exemple de l'isomorphisme \cite[3.5.1.2]{Beintro2}, ou alors on procède comme pour la vérification du théorème de changement de base de l'image directe de \cite[2.4.2]{Beintro2}).
 De plus, d'après \cite[3.5.1.2]{Beintro2}, on dispose de l'isomorphisme canonique 
 $u _{+} (\FF ) \riso \underset{\underset{i}{\longleftarrow}}{\lim\,}  u _{i +} (\FF_i ) $. 
Posons, $\E :=  u _{+} (\FF) $ 
et
$\E _i := \smash{\D} ^{(m)} _{U_i} \otimes ^{\L} _{\smash{\widehat{\D}} ^{(m)} _{\U}} 
 u _{+} (\FF) $.  
D'après \cite[1.2.5]{surcoh-hol}, on dispose de plus de l'isomorphisme canonique
$\mathcal{H} ^0 u ^{!} ( {\E}) 
\riso
\underset{\underset{i}{\longleftarrow}}{\lim\,} 
\mathcal{H} ^0 u _i ^! ({\E} _i) $.
Or, d'après \cite[1.2.5]{surcoh-hol}, on dispose du monomorphisme canonique 
$\FF _i  
\hookrightarrow
\mathcal{H} ^{0} u _i ^{!} \circ u _{i +} (\FF_i ).$
En passant à la limite projective, on obtient d'après ce qui précède
le morphisme canonique
$\FF 
\to
\mathcal{H} ^{0} u ^{!} \circ u _{+} (\FF)$. 
Soit $\Y'$ un ouvert affine de $\Y$.
Il résulte de \cite[2.3.1]{surcoh-hol} que ce morphisme 
induit l'isomorphisme
$\FF (\Y' ) \riso \mathcal{H} ^{0} u ^{!} \circ u _{+} (\FF) (\Y')$.
D'où le résultat. 
\end{proof}

\begin{rema}
\label{rema-F=u!u+F}
Avec les notations \ref{F=u!u+F},
on prendra garde au fait que l'analogue arithmétique à niveau $m$ constant du théorème de Kashiwara est inexact. Plus précisément, 
un $\smash{\widehat{\D}} ^{(m)} _{\U}$-module cohérent sans $p$-torsion à support dans $Y$
n'est pas forcément dans l'image essentielle du foncteur $v ^{(m)} _{+}\colon  D ^\mathrm{b} _\mathrm{coh} (\smash{\widehat{\D}} ^{(m)} _{\Y})
\to D ^\mathrm{b} _\mathrm{coh} (\smash{\widehat{\D}} ^{(m)} _{\U})$. Cette remarque explique 
pourquoi dans la définition des catégories 
$\mathrm{Coh} ^{(m)} (\U,Y/\V)$ (resp. $\mathrm{Cris} ^{(m)} (\U,Y/\V)$) de \ref{defi-coh-m-U} 
on fait l'hypothèse que
$\E |\U _\alpha $ soit dans l'image essentielle de $v _{\alpha +}$
afin d'obtenir les équivalences de catégories de \ref{eqcatmVcoh}.

\end{rema}

\begin{lemm}
\label{lemmsstorsion}
  Soit $v\colon \Y \hookrightarrow \U$ une immersion fermée de $\V$-schémas formels lisses.
Soit $\E$ un $\smash{\widehat{\D}} ^{(m)} _{\Y}$-module cohérent sans $p$-torsion.
  Alors $ v_+ (\E)$ est un $\smash{\widehat{\D}} ^{(m)} _{\U}$-module cohérent sans $p$-torsion.
\end{lemm}

\begin{proof}
L'assertion étant locale, on se ramène au cas où
$\U$ est muni de coordonnées locales $t _1, \dots, t _n$ telles
  que $t _{r+1}, \dots, t _n$ engendrent l'idéal définissant $v$.
Par préservation de la cohérence par l'image directe d'un morphisme propre, on obtient que
$ v_+ (\E)$ est un $\smash{\widehat{\D}} ^{(m)} _{\U}$-module cohérent.
Il reste à établir que
$ v_+ (\E)=
v _* (\smash{\widehat{\D}} ^{(m)} _{\U\leftarrow \Y} \otimes _{\smash{\widehat{\D}} ^{(m)} _{\Y}} \E)$
est sans $p$-torsion.
Pour tout entier $i$, notons $U _{i}$ et $Y _{i}$ les réductions de $\U$ et $\Y$ modulo $\pi ^{i+1}$.
Or, comme pour tout entier $i$ le $\D ^{(m)} _{Y _i}$-module à droite
$\D ^{(m)} _{U _i \leftarrow Y _i}$ est plat (car libre), on déduit de \cite[3.2.4]{Be1} que
$\smash{\widehat{\D}} ^{(m)} _{\U\leftarrow \Y}$ est plat sur $\smash{\widehat{\D}} ^{(m)} _{\Y}$.
D'où le résultat.
\end{proof}

\begin{prop}
\label{cohUY-eqcatrel}
Supposons qu'il existe 
$u \colon \Y\hookrightarrow \U$ 
un relèvement de $Y\hookrightarrow \U$.
Les foncteurs $u _+$ et $\mathcal{H} ^{0} u ^{!} $ induisent des équivalences quasi-inverses
entre la catégorie des $\smash{\widehat{\D}} ^{(m)} _{\Y}$-modules cohérents sans $p$-torsion
et $\mathrm{Coh} ^{(m)} (\U,Y/\V)$. 
\end{prop}

\begin{proof}
Par définition, avec \ref{CohUYV-local} et \ref{lemmsstorsion}, on sait déjà que le foncteur $u _+$ est bien défini et essentiellement surjectif. 
Soient $\E, \E'$ deux  $\smash{\widehat{\D}} ^{(m)} _{\Y}$-modules cohérents sans $p$-torsion
et $\phi\,:\, u _+(\E')  \to u _+(\E') $ un morphisme de $\mathrm{Coh} ^{(m)} (\U,Y/\V)$, i.e. un morphisme 
$\smash{\widehat{\D}} ^{(m)} _{\U}$-linéaire.
Considérons le diagramme canonique 
\begin{equation}
\notag
\xymatrix @C=1,8cm {
{u _{+} (\E')}
\ar[d] ^-{u _{+}(\psi)}
\ar[r] ^-{\sim} 
& 
{u _{+}\circ \mathcal{H} ^{0} u ^{!} \circ u _{+} (\E')}
\ar[d] ^-{u _{+}\circ \mathcal{H} ^{0} u ^{!} (\phi)}
\ar[r] 
& 
{u _{+} (\E')}
\ar[d] ^-{\phi}
\\ 
{ u _{+} (\E)} 
\ar[r] ^-{\sim} 
& 
{u _{+}\circ \mathcal{H} ^{0} u ^{!}\circ u _{+} (\E)} 
\ar[r] 
& 
{u _{+} (\E),} 
}
\end{equation}
dont les flèches horizontales de gauche sont les images par $u _+$ des isomorphismes de \ref{F=u!u+F} 
et dont les flèches horizontales de droite se construisent de manière analogue 
par passage à la limite projective à partir de \cite[1.2.6]{surcoh-hol} (au lieu de \cite[1.2.5]{surcoh-hol})
et dont le morphisme $\psi \colon \E' \to \E$ est l'unique morphisme induisant le carré commutatif de gauche sans $u _+$. 
Comme les morphismes composés horizontaux sont l'identité, on en déduit le résultat  (et il en résulte aussi que tous les morphismes horizontaux sont
des isomorphismes). 
\end{proof}

\begin{lemm}
\label{cohssi=cohi}
Pour tout entier $i$, notons $U _{i}$ la réduction de $\U$ modulo $\pi ^{i+1}$.
Soit $\E$ un $\smash{\D} ^{(m)} _{U _i}$-module.
Le module $\E$ est $\smash{\D} ^{(m)} _{U _i}$-cohérent 
si et seulement s'il est $\smash{\widehat{\D}} ^{(m)} _{\U}$-cohérent.
\end{lemm}

\begin{proof}
Le lemme étant locale en $\U$, on 
peut supposer que $\U$ est affine. Notons alors respectivement $E$, $\smash{D} ^{(m)} _{U _i}$,
$\smash{\widehat{D}} ^{(m)} _{\U}$ les sections globales de 
$\E$, $\smash{\D} ^{(m)} _{U _i}$,
$\smash{\widehat{\D}} ^{(m)} _{\U}$.
Comme $\smash{D} ^{(m)} _{U _i}$ est un $\smash{\widehat{D}} ^{(m)} _{\U}$-module de type fini,
d'après le théorème $A$ de Berthelot sur les $\smash{\widehat{\D}} ^{(m)} _{\U}$-cohérents (voir \cite[3]{Be1}),
le module $\G :=\smash{\widehat{\D}} ^{(m)} _{\U} \otimes _{\smash{\widehat{D}} ^{(m)} _{\U}} \smash{D} ^{(m)} _{U _i}$
est un $\smash{\widehat{\D}} ^{(m)} _{\U}$-module cohérent. 
Si $\U '$ est un ouvert affine de $\U$, 
comme 
$\smash{\widehat{\D}} ^{(m)} _{\U'} \otimes _{\smash{\widehat{D}} ^{(m)} _{\U'}} 
\smash{\widehat{D}} ^{(m)} _{\U '} \otimes _{\smash{\widehat{D}} ^{(m)} _{\U}} \smash{D} ^{(m)} _{U _i} 
\riso
\G |\U'
$,
d'après le théorème $A$ de Berthelot sur les $\smash{\widehat{\D}} ^{(m)} _{\U'}$-cohérents,
il en résulte que le morphisme canonique
 $\smash{\widehat{D}} ^{(m)} _{\U '} \otimes _{\smash{\widehat{D}} ^{(m)} _{\U}} \smash{D} ^{(m)} _{U _i} \to \G (\U')$
est un isomorphisme. Or $\smash{\widehat{D}} ^{(m)} _{\U '} \otimes _{\smash{\widehat{D}} ^{(m)} _{\U}} \smash{D} ^{(m)} _{U _i}$
est canoniquement isomorphe à $\smash{D} ^{(m)} _{U '_i}= \smash{\D} ^{(m)} _{U _i} (U ' _i)$, où $U' _i$ est l'ouvert affine de $U _i$ correspondant à $\U'$. 
On en déduit que le morphisme canonique 
$\smash{\widehat{\D}} ^{(m)} _{\U} \otimes _{\smash{\widehat{D}} ^{(m)} _{\U}} \smash{D} ^{(m)} _{U _i}
\to \smash{\D} ^{(m)} _{U _i}$
est un isomorphisme (en d'autres termes, on a établi que $ \smash{\D} ^{(m)} _{U _i}$ est $\smash{\widehat{\D}} ^{(m)} _{\U}$-cohérent). 
On dispose alors du diagramme commutatif: 
\begin{equation}
\notag
\xymatrix{
{\smash{\widehat{\D}} ^{(m)} _{\U} \otimes _{\smash{\widehat{D}} ^{(m)} _{\U}} E} 
\ar[r] ^-{\sim}
\ar[d] ^-{}
& 
{\smash{\widehat{\D}} ^{(m)} _{\U} \otimes _{\smash{\widehat{D}} ^{(m)} _{\U}} 
\smash{D} ^{(m)} _{U _i} \otimes _{\smash{D} ^{(m)} _{U _i}} E} 
\ar[r] ^-{\sim}
& 
{\smash{\D} ^{(m)} _{U _i} \otimes _{\smash{D} ^{(m)} _{U _i}} E} 
\ar[d] ^-{}
\\ 
{\E} 
\ar@{=}[rr]
&& {\E, } }
\end{equation}
dont la flèche du haut à droite est un isomorphisme d'après ce que l'on vient d'établir. 
La $\smash{\D} ^{(m)} _{U _i}$-cohérence (resp. $\smash{\widehat{\D}} ^{(m)} _{\U}$-cohérence) de 
$\E$ signifie d'après le théorème de type $A$ correspondant que la flèche de droite (resp. de gauche) est un isomorphisme. 
\end{proof}

\begin{lemm}
\label{u+UandUi}
Soient $u\colon\Y \hookrightarrow \U$ une immersion fermée de $\V$-schémas formels lisses
et $u _{i}\colon Y _i \hookrightarrow U _i$ l'immersion fermée de $\V/\pi ^{i+1} \V$-schémas lisses induit. 
\begin{enumerate}
\item Soit $\G$ un $\smash{\D} ^{(m)} _{Y _i}$-module cohérent. 
On dispose alors de l'isomorphisme canonique de $\smash{\D} ^{(m)} _{U _i}$-modules cohérents :
\begin{equation}
\notag
u _{+} (\G) := 
u _{*}(\smash{\widehat{\D}} ^{(m)} _{\U\hookleftarrow \Y} 
\otimes _{\smash{\widehat{\D}} ^{(m)} _{\Y}} \G )
\riso
u _{*}(\smash{\D} ^{(m)} _{U _i \hookleftarrow Y _i} 
\otimes _{\smash{\D} ^{(m)} _{Y _i}} \G )
=: u _{i+} (\G) .
\end{equation}

\item Soit $\E$ un $\smash{\D} ^{(m)} _{U _i}$-module cohérent. 
On dispose alors de l'isomorphisme canonique de complexes de $\smash{\D} ^{(m)} _{Y _i}$-modules :
\begin{equation}
\notag
u ^{!} (\E) := 
\smash{\widehat{\D}} ^{(m)} _{\Y \hookrightarrow \U} 
\otimes ^{\L} _{u ^{-1} \smash{\widehat{\D}} ^{(m)} _{\U}} u ^{-1}\E 
\riso
\smash{\D} ^{(m)} _{Y _i \hookrightarrow U_i} 
\otimes ^{\L} _{u ^{-1} \smash{\D} ^{(m)} _{U_i}} u ^{-1}\E 
=: u ^{!} _{i} (\E).
\end{equation}

\end{enumerate}

\end{lemm}

\begin{proof}
Comme $\smash{\widehat{\D}} ^{(m)} _{\U\hookleftarrow \Y} 
\otimes _{\smash{\widehat{\D}} ^{(m)} _{\Y}} 
 \smash{\D} ^{(m)} _{Y _i}
\riso
\smash{\D} ^{(m)} _{U _i \hookleftarrow Y _i} $,
 on dispose des isomorphismes: 
\begin{equation}
\notag
\smash{\widehat{\D}} ^{(m)} _{\U\hookleftarrow \Y} 
\otimes _{\smash{\widehat{\D}} ^{(m)} _{\Y}} \G 
\riso
\smash{\widehat{\D}} ^{(m)} _{\U\hookleftarrow \Y} 
\otimes _{\smash{\widehat{\D}} ^{(m)} _{\Y}} 
 \smash{\D} ^{(m)} _{Y _i}
\otimes _{\smash{\D} ^{(m)} _{Y _i}} \G 
\riso
\smash{\D} ^{(m)} _{U _i \hookleftarrow Y _i} 
\otimes _{\smash{\D} ^{(m)} _{Y _i}} \G .
\end{equation}
D'où la première assertion. 
De plus, comme $\smash{\widehat{\D}} ^{(m)} _{\Y \hookrightarrow \U} $ n'a pas de $\pi ^{i +1}$-torsion, 
via la résolution plate $\smash{\widehat{\D}} ^{(m)} _{\U} \overset{\pi ^{i +1}}{\longrightarrow}\smash{\widehat{\D}} ^{(m)} _{\U}$
de 
$\smash{\D} ^{(m)} _{U _i}$,
on obtient l'isomorphisme canonique
$ \smash{\widehat{\D}} ^{(m)} _{\Y \hookrightarrow \U} 
\otimes ^{\L} _{u ^{-1} \smash{\widehat{\D}} ^{(m)} _{\U}} u ^{-1} \smash{\D} ^{(m)} _{U _i}
\riso
\smash{\D} ^{(m)} _{Y _i \hookrightarrow U_i} $.
D'où
\begin{equation}
\notag
\smash{\widehat{\D}} ^{(m)} _{\Y \hookrightarrow \U} 
\otimes ^{\L} _{u ^{-1} \smash{\widehat{\D}} ^{(m)} _{\U}} u ^{-1}\E 
\riso
\smash{\widehat{\D}} ^{(m)} _{\Y \hookrightarrow \U} 
\otimes ^{\L} _{u ^{-1} \smash{\widehat{\D}} ^{(m)} _{\U}} u ^{-1} \smash{\D} ^{(m)} _{U _i}
\otimes ^{\L} _{u ^{-1} \smash{\D} ^{(m)} _{U_i}} u ^{-1}\E 
\riso
\smash{\D} ^{(m)} _{Y _i \hookrightarrow U_i} 
\otimes ^{\L} _{u ^{-1} \smash{\D} ^{(m)} _{U_i}} u ^{-1}\E .
\end{equation}

\end{proof}

\begin{prop}
\label{3.4.3Be1-gen}
Soient $\U = \U' \cup \U''$ un recouvrement ouvert de $\U$,
$Y' := Y \cap U'$ et $Y'' := Y \cap U''$.
Soient $\E' \in \mathrm{Coh} ^{(m)} (\U',Y'/\V)$
et $\E'' \in \mathrm{Coh} ^{(m)} (\U'',Y''/\V)$, 
$\epsilon \colon \E' _\Q | \U ' \cap \U'' \riso \E'' _\Q | \U ' \cap \U''$ un isomorphisme
$\smash{\widehat{\D}} ^{(m)} _{\U' \cap \U'', \Q }$-linéaire.
Il existe alors $\E \in \mathrm{Coh} ^{(m)} (\U,Y/\V)$ 
prolongeant $\E'$ sur $\U'$ et une inclusion
$\E |\U'' \hookrightarrow p ^{-n}\E''$ pour un certain entier $n\geq 0$ tel que 
le morphisme induit 
$\E _\Q |\U'' \to \E'' _\Q$ soit un isomorphisme prolongeant $\epsilon$.
\end{prop}

\begin{proof}
La preuve est analogue à celle \cite[3.4.3]{Be1}.
Quitte à faire une récurrence sur le nombre d'ouverts affines recouvrant $\U''$, on peut supposer $\U''$ affine. 
Il existe alors $u ^{\prime \prime}\colon \Y'' \hookrightarrow \U''$ un relèvement de $Y'' \hookrightarrow \U''$.
On note $v ^{\prime \prime}\colon \Y'' \cap \U' \hookrightarrow \U''\cap \U'$ l'immersion fermée induite par $u''$.
L'immersion fermée $v ^{\prime \prime}$ est ainsi un relèvement de $Y'' \cap Y' \hookrightarrow \U'' \cap \U'$.
Quitte à multiplier par une puissance de $p$, on peut supposer 
$\epsilon (\E' | \U ' \cap \U'') \subset \E''  | \U ' \cap \U''$.

D'après \ref{cohUY-eqcatrel}, 
il existe un $\smash{\widehat{\D}} ^{(m)} _{\Y'' \cap \U'}$-module cohérent $\FF'$ sans $p$-torsion
et un isomorphisme 
$\iota ' \colon \E' | \U ' \cap \U'' \riso v ^{\prime \prime} _+ (\FF') $. 
D'après \ref{cohUY-eqcatrel}, 
il existe un $\smash{\widehat{\D}} ^{(m)} _{\Y'' }$-module cohérent $\FF''$ sans $p$-torsion
et un isomorphisme 
$\iota '' \colon \E'' \riso u ^{\prime \prime} _+ (\FF'')$. 
D'après \ref{cohUY-eqcatrel}, il existe un morphisme
$\theta\colon \FF'  \to \FF '' |\Y'' \cap \U'$ tel que 
$ v ^{\prime \prime} _+ (\theta)$ rende commutatif le diagramme ci-dessous
\begin{equation}
\label{3.4.3Be1-gen-diag1}
\xymatrix@C=1,5cm{
{\E' |\U' \cap \U''} 
\ar[rr] ^-{\iota '} _-{\sim}
\ar@{^{(}->}[d] ^-{\epsilon}
&& 
{v ^{\prime \prime} _+ (\FF')}
\ar@{.>}[d] ^-{v ^{\prime \prime} _+ (\theta)}
\\ 
{\E'' |\U' \cap \U''} 
\ar[r] ^-{\iota ''|\U' \cap \U''} _-{\sim}  
& 
{ u ^{\prime \prime} _+ (\FF'')|\U' \cap \U''} 
\ar[r] ^-{\sim}  
&
{v ^{\prime \prime} _+ (\FF'' |\U' \cap \Y'').} 
}
\end{equation}
Comme $\epsilon _\Q$ est un isomorphisme, il en est de même de 
$\theta _\Q$. Il existe donc $i$ assez grand tel que $\pi ^{i+1}\FF''  | \U ' \cap \Y''
 \subset 
 \theta (\FF' )$.
 Notons $\FF '' _i := \FF'' / \pi ^{i+1}\FF''$.
D'après \ref{cohssi=cohi},
$\FF '' _i $ est un $\smash{\widehat{\D}} ^{(m)} _{\Y''}$-module cohérent. 
On note $t\colon \FF'' \to \FF '' _i$ la surjection canonique de $\smash{\widehat{\D}} ^{(m)} _{\Y''}$-modules cohérents.
 Soit $\overline{\FF'} _i $ l'image du morphisme 
 $ (t | \U' \cap \Y'' ) \circ \theta \colon \FF' \to \FF '' _i | \U ' \cap \Y'' $.
On remarque que comme $\pi ^{i+1}\FF''  | \U ' \cap \Y''
 \subset 
 \theta (\FF' )$, alors $(t | \U' \cap \Y'' ) ^{-1} (\overline{\FF'} _i ) =\theta (\FF' )$.
Comme $\overline{\FF'} _i $ est un $\smash{\D} ^{(m)} _{Y'' _i\cap U' _i}$-module cohérent, 
de manière analogue à la preuve de \cite[3.4.3]{Be1}, il existe alors $\G _i$ un 
sous-$\smash{\D} ^{(m)} _{Y'' _i}$-module cohérent de $\FF'' _i $ tel que $\overline{\FF'} _i =\G _i | \U ' \cap \Y''$.
Avec \ref{cohssi=cohi}, $\G _i$ est aussi $\smash{\widehat{\D}} ^{(m)} _{\Y''}$-cohérent. 
Il en résulte que $\G := t ^{-1} (\G _i) $ est un 
sous-$\smash{\widehat{\D}} ^{(m)} _{\Y'' }$-module cohérent de $\FF''$. 
Comme $\pi ^{i +1} \FF'' \subset \G \subset \FF''$, on obtient 
$\G _\Q = \FF'' _\Q$.
On a de plus
$\G | \U ' \cap \Y''=(t | \U' \cap \Y'' ) ^{-1} (\overline{\FF'} _i ) =\theta (\FF' )$.
On note alors $\theta '\colon \FF' \to \G | \U ' \cap \Y''$ le morphisme factorisant $\theta$
et on pose
$\H := (\iota ^{\prime \prime}) ^{-1} (u ^{\prime \prime} _+ (\G)) \subset \E ''$.
Comme $\G _\Q = \FF'' _\Q$, alors $\H _\Q = \E'' _\Q$.
On dispose alors du diagramme commutatif
\begin{equation}
\label{3.4.3Be1-gen-diag2}
\xymatrix@C=1,5cm{
{\E' |\U' \cap \U''} 
\ar[rr] ^-{\iota '} _-{\sim}
\ar@/_1,2cm/@{^{(}->}[dd] _-{\epsilon}
\ar@{.>}[d] ^-{}
&& 
{v ^{\prime \prime} _+ (\FF')}
\ar@/^1,2cm/@{^{(}->}[dd] ^-{v ^{\prime \prime} _+ (\theta)}
\ar[d] _-{v ^{\prime \prime} _+ (\theta ')} ^-{\sim}
\\ 
{\H |\U' \cap \U''} 
\ar[r] ^-{\iota ''|\U' \cap \U''} _-{\sim}  
\ar@{^{(}->}[d]
& 
{ u ^{\prime \prime} _+ (\G)|\U' \cap \U''} 
\ar[r] ^-{\sim}  
\ar@{^{(}->}[d]
&
{v ^{\prime \prime} _+ (\G |\U' \cap \Y'')} 
\ar@{^{(}->}[d]
\\
{\E'' |\U' \cap \U''} 
\ar[r] ^-{\iota ''|\U' \cap \U''} _-{\sim}  
& 
{ u ^{\prime \prime} _+ (\FF'')|\U' \cap \U''} 
\ar[r] ^-{\sim}  
&
{v ^{\prime \prime} _+ (\FF'' |\U' \cap \Y'').} 
}
\end{equation}
dont le carré de gauche est commutatif et cartésien par définition de $\H$,
dont le carré de droite est commutatif et cartésien par fonctorialité,
dont le contour est commutatif car égal à celui de \ref{3.4.3Be1-gen-diag1}.
On en déduit que $\epsilon$ se factorise en 
l'isomorphisme $\E' |\U' \cap \U'' \riso \H |\U' \cap \U''$.
Les deux faisceaux $\E'$ et $\H$ se recollent donc en un faisceau que l'on note $\E$. 
Le faisceau $\E$ a les propriétés voulues.
\end{proof}

\begin{vide}
\label{defi-loc}
On construit le foncteur $\mathcal{L}oc ^{(m)}\colon \mathrm{Coh} ^{(m)} (\U,Y/\V)
\rightarrow
\mathrm{Coh} ^{(m)} (Y,\, (\Y _\alpha) _{\alpha \in \Lambda}/\V)$ de la façon suivante :
soit $\E$ un $\smash{\widehat{\D}} ^{(m)} _{\U }$-module cohérent appartenant à $\mathrm{Coh} ^{(m)} (\U, Y/\V)$.
Pour tout $\alpha \in \Lambda $, il existe un $\smash{\widehat{\D}} ^{(m)} _{\Y _\alpha}$-module cohérent $\E _\alpha$ tel que $v_{\alpha+} (\E _\alpha) \riso \E |\U _\alpha$.
Grâce à \ref{cohUY-eqcatrel},
$\mathcal{H} ^0  v _{\alpha } ^! (\E |{\U _\alpha})$ est
un $\smash{\widehat{\D}} ^{(m)} _{\Y _\alpha}$-module cohérent. 
On définit alors l'isomorphisme
$\theta _{\alpha \beta}\colon  p  _2 ^{\alpha \beta !}  \mathcal{H} ^0 v _{\beta } ^! (\E |{\U _\beta})
\riso
p  _1 ^{\alpha \beta !} \mathcal{H} ^0  v _{\alpha } ^! (\E |{\U _\alpha}) $,
comme étant l'unique flèche rendant commutatif le diagramme suivant
\begin{equation}
  \label{coh->cohloc}
  \xymatrix  @R=0,3cm {
{ p  _2 ^{\alpha \beta !} \mathcal{H} ^0  v _{\beta } ^! (\E |{\U _\beta}) }
\ar[r]^-{ \mathcal{H} ^0 \tau} _-{\sim}
\ar@{.>}[d] ^-{\theta _{\alpha \beta}}
&
{\mathcal{H} ^0 v ^! _{\alpha \beta} ( (\E |{\U _\beta}) |{\U _{\alpha \beta}}) }\ar@{=}[d] \\
{ p  _1 ^{\alpha \beta !} \mathcal{H} ^0  v _{\alpha } ^! (\E |{\U _\alpha}) }
\ar[r]^-{ \mathcal{H} ^0 \tau} _-{\sim}
&
{\mathcal{H} ^0 v ^! _{\alpha \beta} ( (\E |{\U _\alpha}) |{\U _{\alpha \beta}}) ,}
}
\end{equation}
où les isomorphismes horizontaux sont les isomorphismes de recollement définis dans \ref{nota-tau} 
auquels on a appliqué le foncteur $ \mathcal{H} ^0 $ (on utilise aussi 
que les foncteurs $p  _1 ^{\alpha \beta !}$ et $p  _2 ^{\alpha \beta !}$ sont exactes
car $p  _1 ^{\alpha \beta }$ et $p  _2 ^{\alpha \beta }$ sont des immersions ouvertes).
Comme lors de la preuve de \cite[2.5.4]{caro-construction}, on vérifie alors que
le foncteur $\mathcal{L}oc ^{(m)}$ est bien défini en posant
\begin{equation}
  \label{egalite-defi-loc}
  \mathcal{L}oc ^{(m)} (\E) :=
((\mathcal{H} ^0 v _{\alpha } ^!
(\E |{\U _\alpha})) _{\alpha \in \Lambda},\, (\theta _{\alpha \beta})_{\alpha , \beta \in \Lambda}).
\end{equation}
On dispose en outre de la factorisation (encore notée) $\mathcal{L}oc ^{(m)}\colon \mathrm{Cris} ^{(m)} (\U,Y/\V)
\rightarrow
\mathrm{Cris} ^{(m)} (Y,\, (\Y _\alpha) _{\alpha \in \Lambda}/\V)$.
\end{vide}

\begin{vide}
\label{defi-reloc}
Soit
$((\E _{\alpha})_{\alpha \in \Lambda},\, (\theta _{\alpha\beta}) _{\alpha ,\beta \in \Lambda})
\in
\mathrm{Coh} ^{(m)} (Y,\, (\Y _\alpha) _{\alpha \in \Lambda}/\V)$.
Comme dans la preuve de \cite[2.5.4]{caro-construction}, on vérifie que
la famille
$(v _{\alpha +}(\E _{\alpha}))_{\alpha \in \Lambda}$ se recolle
en un $\smash{\widehat{\D}} ^{(m)} _{\U}$-module cohérent à support dans $Y$.
On obtient ainsi le foncteur canonique :
$$\mathcal{R}ecol ^{(m)}
\colon 
\mathrm{Coh} ^{(m)} (Y,\, (\Y _\alpha) _{\alpha \in \Lambda}/\V)
\rightarrow
\mathrm{Coh} ^{(m)} (\U,Y/\V).$$
On dispose en outre de la factorisation (encore notée) 
$\mathcal{R}ecol ^{(m)}
\colon 
\mathrm{Cris} ^{(m)} (Y,\, (\Y _\alpha) _{\alpha \in \Lambda}/\V)
\rightarrow
\mathrm{Cris} ^{(m)} (\U,Y/\V).$
\end{vide}

\begin{vide}
\label{eqcatmVcoh}
De manière analogue à la preuve de \cite[2.5.4]{caro-construction} (qui est la version analogue en remplaçant {\og $(m)$\fg}
par {\og $\dag$\fg}), on vérifie que
les deux foncteurs $\mathcal{L}oc ^{(m)}$ et $\mathcal{R}ecol ^{(m)}$ induisent des équivalences quasi-inverses entre les catégories
$\mathrm{Coh} ^{(m)} (\U,Y/\V)$
et
$\mathrm{Coh} ^{(m)} (Y,\, (\Y _\alpha) _{\alpha \in \Lambda}/\V)$
(resp. $\mathrm{Cris} ^{(m)} (\U,Y/\V)$
et
$\mathrm{Cris} ^{(m)} (Y,\, (\Y _\alpha) _{\alpha \in \Lambda}/\V)$).

\end{vide}

\begin{vide}
\label{section-diagcomm-otimesk}
Désignons par $\mathrm{Coh} ( \smash{\D} ^{(m)} _{Y})$ la catégorie des
$\smash{\D} ^{(m)} _{Y}$-modules cohérents, de même en remplaçant $Y$ par $U$.
On bénéficie du foncteur
$- \otimes _{\V} k\colon \mathrm{Coh} ^{(m)} (\U,Y/\V)
\rightarrow
\mathrm{Coh} ( \smash{\D} ^{(m)} _{U})$.

De plus,
soit $((\E _{\alpha})_{\alpha \in \Lambda},\, (\theta _{\alpha\beta}) _{\alpha ,\beta \in \Lambda})
\in \mathrm{Coh} ^{(m)} (Y,\, (\Y _\alpha) _{\alpha \in \Lambda}/\V)$.
En identifiant $\smash{\widehat{\D}} ^{(m)} _{\Y _\alpha} \otimes _\V k $ et
$\smash{\D} ^{(m)} _{Y _\alpha}$,
on remarque alors que
la famille $(\theta _{\alpha\beta}\otimes _{\V} k) _{\alpha ,\beta \in \Lambda}$
est une donnée de recollement (au sens topologique usuel car les isomorphismes de recollement de la forme $\tau$ de \ref{defi-f+-f+(m)} deviennent l'identité)
de la famille $(\E _{\alpha}\otimes _{\V} k)_{\alpha \in \Lambda}$.
On obtient donc par recollement un
$\smash{\D} ^{(m)} _{Y}$-module cohérent.
Ainsi, on dispose du foncteur (noté de façon légèrement abusive)
\begin{equation}
\label{defi-recol-otimes-k}
- \otimes _{\V} k\text{ : }
\mathrm{Coh} ^{(m)} (Y,\, (\Y _\alpha) _{\alpha \in \Lambda}/\V)
\rightarrow
\mathrm{Coh} ( \smash{\D} ^{(m)} _{Y}).
\end{equation}
On vérifie de plus que l'on dispose du diagramme commutatif à isomorphisme canonique près de foncteurs :
\begin{equation}
\label{diagcomm-otimesk}
  \xymatrix {
{\mathrm{Coh} ^{(m)} (Y,\, (\Y _\alpha) _{\alpha \in \Lambda}/\V) }
\ar[r] ^-{\mathcal{R}ecol ^{(m)}}
\ar[d] _-{- \otimes _{\V} k}
&
{\mathrm{Coh} ^{(m)} (\U,Y/\V)  }
\ar[d] ^-{- \otimes _{\V} k}
\\
{ \mathrm{Coh} ( \smash{\D} ^{(m)} _{Y}) }
\ar[r] ^-{v_{+}}
&
{ \mathrm{Coh} ( \smash{\D} ^{(m)} _{U}) }.}
\end{equation}
\end{vide}

\begin{vide}
\label{etsurQ?}
\begin{enumerate}
\item \label{defi-coh-m-alpahQ} On construit la catégorie
$\mathrm{Coh} ^{(m)} (Y,\, (\Y _\alpha) _{\alpha \in \Lambda}/K)$
de manière analogue à
$\mathrm{Coh} ^{(m)} (Y,\, (\Y _\alpha) _{\alpha \in \Lambda}/\V)$
(voir \ref{defi-coh-m-alpah})
en remplaçant la notion de
{\og
$\smash{\widehat{\D}} ^{(m)} _{\Y _{\alpha}}$-module cohérent 
\fg}
par celle de
{\og
$\smash{\widehat{\D}} ^{(m)} _{\Y _{\alpha},\Q}$-module cohérent
\fg}.
Le foncteur $-\otimes _\Z \Q$ induit de cette façon le suivant
$-\otimes _\Z \Q\colon \mathrm{Coh} ^{(m)} (Y,\, (\Y _\alpha) _{\alpha \in \Lambda}/\V)
\rightarrow
\mathrm{Coh} ^{(m)} (Y,\, (\Y _\alpha) _{\alpha \in \Lambda}/K)$.

\item \label{defi-coh-m-UQ} On définit $\mathrm{Coh} ^{(m)} (\U,Y/K)$
la catégorie des
$\smash{\widehat{\D}} ^{(m)} _{\U,\Q}$-modules cohérents $\E$ à support dans $Y$ tels qu'il
existe un élément $\overset{_\circ}{\E} \in \mathrm{Coh} ^{(m)} (\U,Y/\V)$ (voir \ref{defi-coh-m-U})
et un isomorphisme $\smash{\widehat{\D}} ^{(m)} _{\U,\Q}$-linéaire
$\E \riso \overset{_\circ}{\E}  _\Q$.
Par définition, on dispose du foncteur
$-\otimes _\Z \Q\colon \mathrm{Coh} ^{(m)} (\U,Y/\V) \rightarrow \mathrm{Coh} ^{(m)} (\U,Y/K)$.

\item On construit de manière analogue à respectivement \ref{defi-loc} et \ref{defi-reloc} les foncteurs
\begin{gather}
  \notag
  \mathcal{L}oc ^{(m)} _{\Q}\, :
\mathrm{Coh} ^{(m)} (\U,Y/K)
\rightarrow
\mathrm{Coh} ^{(m)} (Y,\, (\Y _\alpha) _{\alpha \in \Lambda}/K).
\\
\mathcal{R}ecol ^{(m)} _{\Q}
\colon 
\mathrm{Coh} ^{(m)} (Y,\, (\Y _\alpha) _{\alpha \in \Lambda}/K)
\rightarrow
\mathrm{Coh} ^{(m)} (\U,Y/K).
\end{gather}
Ceux-ci sont quasi-inverses (cela se vérifie comme pour \cite[2.5.4]{caro-construction} en remplaçant {\og $\dag$\fg}
par {\og $(m)$\fg}).
\item \label{commQ} On dispose des isomorphismes canoniques de foncteurs 
$\mathcal{L}oc ^{(m)} _{\Q} \circ (\Q \otimes _{\Z} -) \riso  (\Q \otimes _{\Z} -) \circ \mathcal{L}oc ^{(m)} $ et $\mathcal{R}ecol^{(m)} _{\Q} \circ (\Q \otimes _{\Z} -) \riso (\Q \otimes _{\Z} -) \circ \mathcal{R}ecol^{(m)}$. 

\item \label{loc-recol-comm-niveaux}
Soit $m' \geq m$ un entier.
On définit le foncteur $\smash{\widehat{\D}} ^{(m')} _{Y,\Q} \otimes _{\smash{\widehat{\D}} ^{(m)} _{Y,\Q}} -\colon  \mathrm{Coh} ^{(m)} (Y,\, (\Y _\alpha) _{\alpha \in \Lambda}/K) \to \mathrm{Coh} ^{(m')} (Y,\, (\Y _\alpha) _{\alpha \in \Lambda}/K)$ en posant, pour tout 
$\E:= ((\E _{\alpha})_{\alpha \in \Lambda},\, (\theta _{\alpha\beta}) _{\alpha ,\beta \in \Lambda})
\in
\mathrm{Coh} ^{(m)} (Y,\, (\Y _\alpha) _{\alpha \in \Lambda}/\V)$,
$$\smash{\widehat{\D}} ^{(m')} _{Y,\Q} \otimes _{\smash{\widehat{\D}} ^{(m)} _{Y,\Q}} \E
:=
(\smash{\widehat{\D}} ^{(m')} _{\Y _\alpha,\Q} \otimes _{\smash{\widehat{\D}} ^{(m)} _{\Y _\alpha,\Q}} 
\E _{\alpha})_{\alpha \in \Lambda},$$ ce dernier étant muni de la donnée de recollement déduite par extension (via les isomorphismes de commutation de l'image inverse extraordinaire par un morphisme lisse au changement de niveau de \cite[3.4.6]{Beintro2}
et via le fait que les isomorphismes de $\tau $ de \ref{nota-tau} commutent au changement de niveau d'après \cite[2.1.2]{caro-construction}).

Il résulte de \cite[3.5.3]{Beintro2} que le foncteur 
$\mathcal{R}ecol ^{(m)} _{\Q}$ commute au changement de niveaux : on dispose de l'isomorphisme canonique de foncteurs 
$(\smash{\widehat{\D}} ^{(m')} _{\U,\Q} \otimes _{\smash{\widehat{\D}} ^{(m)} _{\U,\Q}} -) \circ \mathcal{R}ecol ^{(m)} _{\Q}
\riso \mathcal{R}ecol ^{(m')} _{\Q}  \circ (\smash{\widehat{\D}} ^{(m')} _{Y,\Q} \otimes _{\smash{\widehat{\D}} ^{(m)} _{Y,\Q}} -) $.  

Pour mémoire (en fait nous n'aurons pas besoin de ce dernier isomorphisme par la suite), comme pour tout entier $m$ le foncteur $\mathcal{L}oc ^{(m)} _{\Q}$ est 
quasi-inverse de 
$\mathcal{R}ecol ^{(m)} _{\Q}$, il en résulte 
$(\smash{\widehat{\D}} ^{(m')} _{Y,\Q} \otimes _{\smash{\widehat{\D}} ^{(m)} _{Y,\Q}} -) \circ  \mathcal{L}oc ^{(m)} _{\Q}
\riso 
 \mathcal{L}oc ^{(m')} _{\Q}\circ (\smash{\widehat{\D}} ^{(m')} _{\U,\Q} \otimes _{\smash{\widehat{\D}} ^{(m)} _{\U,\Q}} -) $.

\item  \label{loc-recol-comm-niveauxlim} Avec les notations de \ref{prop-donnederecol-dag}, on construit de même le foncteur 
$(\D ^{\dag} _{Y,\Q} \otimes _{\smash{\widehat{\D}} ^{(m)} _{Y,\Q}} -) \colon  
\mathrm{Coh} ^{(m)} (Y,\, (\Y _\alpha) _{\alpha \in \Lambda}/K) \to \mathrm{Coh}  (Y,\, (\Y _\alpha) _{\alpha \in \Lambda}/K)$.
On déduit de \cite[4.3.8]{Beintro2} que l'on dispose de l'isomorphisme canonique de foncteurs :
$(\D ^{\dag} _{\U,\Q} \otimes _{\smash{\widehat{\D}} ^{(m)} _{\U,\Q}} -) \circ \mathcal{R}ecol ^{(m)} _{\Q}
\riso \mathcal{R}ecol   \circ (\D ^{\dag} _{Y,\Q} \otimes _{\smash{\widehat{\D}} ^{(m)} _{Y,\Q}} -) $.

Cela implique au passage
$(\D ^{\dag} _{Y,\Q} \otimes _{\smash{\widehat{\D}} ^{(m)} _{Y,\Q}} -) \circ  \mathcal{L}oc ^{(m)} _{\Q}
\riso 
 \mathcal{L}oc \circ (\D ^{\dag} _{\U,\Q} \otimes _{\smash{\widehat{\D}} ^{(m)} _{\U,\Q}} -) $.

\end{enumerate}

\end{vide}

\begin{lemm}
\label{stab-ssquot-O-coh}
Soient $\PP$ un $\V$-schéma formel lisse et $\E$ un $\smash{\widehat{\D}} ^{(m)} _{\PP}$-module cohérent.
\begin{enumerate}
\item Supposons $\PP$ affine. Le faisceau $\E$ est $\O _{\PP}$-cohérent si et seulement si $\Gamma (\PP, \E)$ est $\Gamma (\PP, \O _{\PP})$-cohérent.
\item 
\label{stab-ssquot-O-coh-2}
Soit $\E'$ un sous-quotient de $\E$ dans la catégorie des $\smash{\widehat{\D}} ^{(m)} _{\PP}$-modules cohérents.
Si $\E$ est $\O _{\PP}$-cohérent alors $\E'$ est $\O _{\PP}$-cohérent. 
\end{enumerate}

\end{lemm}

\begin{proof}
Établissons d'abord la première assertion. 
Posons $E:= \Gamma (\PP, \E)$, $O _{\PP}:= \Gamma (\PP, \O _{\PP})$ 
et $\smash{\widehat{D}} ^{(m)} _{\PP}:= \Gamma (\PP,\smash{\widehat{\D}} ^{(m)} _{\PP})$.
Si $\E$ est $\O _{\PP}$-cohérent  alors d'après le théorème de type $A$ pour les $\O _{\PP}$-modules cohérents (voir \cite[3]{Be1}), 
$E$ est $O _{\PP}$-cohérent.  Réciproquement, supposons que $E$ soit $O _{\PP}$-cohérent.
Considérons le diagramme canonique commutatif: 
\begin{equation}
\xymatrix{
{\O _{\PP} \widehat{\otimes} _{O _{\PP}} E} 
\ar[d] ^-{\sim}
& 
{\O _{\PP} \otimes _{O _{\PP}} E}  
\ar[l] ^-{\sim}
\ar[d] ^-{}
\ar[r] ^-{}
&
{\E}
\ar@{=}[d] ^-{}
\\ 
{\smash{\widehat{\D}} ^{(m)} _{\PP} \widehat{\otimes}_{\smash{\widehat{D}} ^{(m)} _{\PP}} E}  
& 
{\smash{\widehat{\D}} ^{(m)} _{\PP}\otimes _{\smash{\widehat{D}} ^{(m)} _{\PP}} E} 
\ar[l] ^-{\sim}
\ar[r] ^-{\sim}
&
{\E}
}
\end{equation}
Comme $E$ est $O _{\PP}$-cohérent (resp. 
$\smash{\widehat{D}} ^{(m)} _{\PP}$-cohérent grâce au théorème de type $A$ pour les $\smash{\widehat{\D}} ^{(m)} _{\PP}$-modules cohérents \cite[3.3.9]{Be1}), 
d'après le théorème de type $A$ de Berthelot de \cite[3.3.9]{Be1} valable pour les $\O _{\PP}$-modules cohérents (resp. $\smash{\widehat{\D}} ^{(m)} _{\PP}$-modules cohérents), le morphisme du haut (resp. du bas) à gauche
est un isomorphisme. 
Comme $\E$ est $\smash{\widehat{\D}} ^{(m)} _{\PP}$-cohérent, à nouveau grâce à \cite[3.3.9]{Be1},
on obtient que la flèche du bas à droite est un isomorphisme.
La flèche vertical de gauche est un isomorphisme par définition de $\smash{\widehat{\D}} ^{(m)} _{\PP}$.
Il en résulte qu'il en ait de même de la flèche de droite du haut. Via le théorème de type $A$ pour les $\O _{\PP}$-modules cohérents,
cela signifie que $\E$ est  $\O _{\PP}$-cohérent.

Prouvons à présent la deuxième assertion. Par stabilité de la cohérence, il suffit de traiter le cas où $\E'$ est un sous-$\smash{\widehat{\D}} ^{(m)} _{\PP}$-module cohérent de $\E$. La cohérence étant de nature locale, on se ramène à supposer
$\PP$ affine. D'après le théorème de type $B$ pour les $\smash{\widehat{\D}} ^{(m)} _{\PP}$-modules cohérents,
$ \Gamma (\PP, \E')$ est un sous-module de $ \Gamma (\PP, \E)$.
Comme $O _{\PP}$ est noethérien, $ \Gamma (\PP, \E')$ est $O _{\PP}$ -cohérent. 
On conclut grâce à la première assertion du lemme.
\end{proof}

\begin{lemm}
\label{coro-stab-ssquot-O-coh}
Soit $\E\in \mathrm{Cris} ^{(m)} (\U, Y/\V)$
et $\E' \in \mathrm{Coh} ^{(m)} (\U, Y/\V)$.
Si $\E'$ est un sous-module de $\E$ alors $\E' \in \mathrm{Cris} ^{(m)} (\U, Y/\V)$.
\end{lemm}

\begin{proof}
Le lemme étant local, on peut supposer 
qu'il existe 
$u \colon \Y\hookrightarrow \U$ 
un relèvement de $Y\hookrightarrow \U$.
Avec \ref{cohUY-eqcatrel},
$\mathcal{H} ^{0} u ^{!}  (\E')$ est un sous-$\smash{\widehat{\D}} ^{(m)} _{\Y}$-module cohérent de
$\mathcal{H} ^{0} u ^{!}  (\E)$. Il en résulte que $\mathcal{H} ^{0} u ^{!}  (\E')$
est topologiquement quasi-nilpotent.
Grâce à \ref{stab-ssquot-O-coh}.\ref{stab-ssquot-O-coh-2},
$\mathcal{H} ^{0} u ^{!}  (\E')$ est en outre $\O _{\Y}$-cohérent.
 On conclut grâce à \ref{cohUY-eqcatrel}.

\end{proof}

\begin{lemm}
Pour tout entier positif $m$, avec les notations de \ref{defindonnederecoldag}, 
on dispose de l'inclusion canonique:
\begin{equation}
\label{IsocdagdagsubsetIsoc(m)}
\mathrm{Isoc} ^{\dag \dag} (Y,\, (\Y _\alpha) _{\alpha \in \Lambda}/K) \subset \mathrm{Coh} ^{(m)} (Y,\, (\Y _\alpha) _{\alpha \in \Lambda}/K).
\end{equation}
\end{lemm}

\begin{proof}
On peut supposer $\U = \U _\alpha$ et on pose $\Y = \Y _\alpha$.
D'après \cite[3.1.2]{Be0}, si $\E$ est un $\D ^{\dag} _{\Y,\Q}$-module cohérent, 
$\O _{\Y,\Q}$-cohérent alors $\E$ est aussi $\widehat{\D} ^{(m)} _{\Y,\Q}$-cohérent 
pour tout entier $m$. D'où l'inclusion \ref{IsocdagdagsubsetIsoc(m)}.
\end{proof}

\begin{nota}
On notera 
$\mathrm{Isoc}  ^{(m)} (\U, Y/K)$ l'image essentielle de 
$$\mathrm{Isoc} ^{\dag \dag} (Y,\, (\Y _\alpha) _{\alpha \in \Lambda}/K)
\underset{\ref{IsocdagdagsubsetIsoc(m)}}{\subset}
\mathrm{Coh} ^{(m)} (Y,\, (\Y _\alpha) _{\alpha \in \Lambda}/K)
\underset{\mathcal{R}ecol ^{(m)}}{\cong}
\mathrm{Coh} ^{(m)} (\U, Y/K)
.$$
Le foncteur $-\otimes _{\Z}\Q$ de \ref{etsurQ?}.\ref{defi-coh-m-UQ}
se factorise alors en
$-\otimes _{\Z}\Q \colon \mathrm{Cris} ^{(m)}(\U, Y/\V) \to 
\mathrm{Isoc} ^{(m)}(\U, Y/K)$.
\end{nota}

\begin{theo}
\label{crist-loc-def}
\begin{itemize}
\item Pour tout objet $\E$ de 
$\mathrm{Isoc}  ^{(m)} (\U, Y/K)$, il existe 
un objet $\overset{\circ}{\E}$ de
$\mathrm{Cris} ^{(m)}(\U, Y/\V)$ et un isomorphisme dans 
$\mathrm{Isoc}  ^{(m)} (\U, Y/K)$ de la forme
$\overset{\circ}{\E} _\Q \riso \E$.

\item Pour tout objet $\FF=((\FF _{\alpha})_{\alpha \in \Lambda},\, (\theta _{\alpha\beta}) _{\alpha ,\beta \in \Lambda})$ de 
$\mathrm{Isoc} ^{\dag \dag} (Y,\, (\Y _\alpha) _{\alpha \in \Lambda}/K)$, il existe 
un objet $\overset{\circ}{\FF}=((\overset{\circ}{\FF _{\alpha}})_{\alpha \in \Lambda},\, (\theta _{\alpha\beta}) _{\alpha ,\beta \in \Lambda})$ de
$\mathrm{Cris} ^{(m)} (Y,\, (\Y _\alpha) _{\alpha \in \Lambda}/\V)$ et un isomorphisme dans 
$\mathrm{Coh} ^{(m)} (Y,\, (\Y _\alpha) _{\alpha \in \Lambda}/K)$ de la forme
$\overset{\circ}{\FF} _\Q \riso \FF$.
\end{itemize}
\end{theo}

\begin{proof}
Les deux assertions étant équivalentes, établissons la première. 
On peut supposer $\U$ intègre. Quitte à considérer un sous-recouvrement fini, on peut supposer que $\Lambda$ est fini. 
On procède par récurrence sur le cardinal de $\Lambda$. 
Supposons que ce cardinal soit $1$. 
Dans ce cas, $Y\hookrightarrow \U$ se relève en une immersion fermée de $\V$-schémas formels lisses de la forme
$u\colon \Y \hookrightarrow \U$. 
Posons $\G := \mathcal{L}oc ^{(m)}(\E)=\mathcal{0} u ^{(m)!} (\E)$. 
Alors $\G$ un $\D ^{\dag} _{\Y,\Q}$-module cohérent, 
$\O _{\Y,\Q}$-cohérent et tel que $\E \riso \mathcal{R}ecol ^{(m)} (\G) = u ^{(m)} _{+} (\G)$.
D'après \cite[3.1.2]{Be0}, $\G$ est aussi $\widehat{\D} ^{(m)} _{\Y,\Q}$-cohérent 
pour tout entier $m$. En particulier, $\G$ est topologiquement nilpotent (voir les remarques de \cite[4.4.6]{Be1}).
D'après \cite[3.1.2]{Be0}, il existe de plus un 
$\widehat{\D} ^{(m)} _{\Y}$-module cohérent topologiquement nilpotent
$\overset{\circ}{\G}$ et $\O _{\Y}$-cohérent et un isomorphisme $\widehat{\D} ^{(m)} _{\Y,\Q}$-linéaire de la forme
$\overset{\circ}{\G} _{\Q}\riso \G$.
Grâce à \cite[3.4.5]{Be1}, quitte à quotienter $\overset{\circ}{\G }$ par son sous-module de $p$-torsion,
on peut choisir le module $\overset{\circ}{\G }$ sans $p$-torsion.
D'où le lemme dans ce cas.

Supposons à présent le théorème vrai pour $\Lambda $ de cardinal strictement plus petit. 
Soit $\E \in \mathrm{Isoc}  ^{(m)} (\U, Y/K)$.
Choisissons $\alpha _0$ un élément de $\Lambda$ et posons 
$\U' := \cup _{\alpha \not = \alpha _0} \U _{\alpha}$
et $\U'' := \U _{\alpha _0}$.
Posons $Y': =Y \cap U'$, $Y'': =Y \cap U''$, $\E':= \E |\U'$ et $\E'':= \E |\U''$.
Par hypothèse de récurrence, 
il existe
$\overset{\circ}{\E'} \in \mathrm{Cris} ^{(m)}(\U', Y'/\V)$ et un isomorphisme dans 
$\mathrm{Isoc}  ^{(m)} (\U', Y'/K)$ de la forme
$\overset{\circ}{\E'} _\Q \riso \E'$.
De même, il existe
$\overset{\circ}{\E''} \in \mathrm{Cris} ^{(m)}(\U'', Y''/\V)$ et un isomorphisme dans 
$\mathrm{Isoc}  ^{(m)} (\U'', Y''/K)$ de la forme
$\overset{\circ}{\E''} _\Q \riso \E''$.
Comme $\E'$ et $\E''$ se recolle en $\E$, on dispose 
de l'isomorphisme canonique $\epsilon \colon \overset{\circ}{\E'} _\Q | \U ' \cap \U'' \riso \overset{\circ}{\E''} _\Q | \U ' \cap \U''$.
D'après \ref{3.4.3Be1-gen}, 
il existe alors $\overset{\circ}{\E}\in \mathrm{Coh} ^{(m)} (\U,Y/\V)$ 
prolongeant $\overset{\circ}{\E'}$ sur $\U'$ et une inclusion
$\overset{\circ}{\E} |\U'' \hookrightarrow p ^{-n} \overset{\circ}{\E''}$ pour un certain entier $n\geq 0$ tel que 
le morphisme induit 
$\overset{\circ}{\E} _\Q |\U'' \to \overset{\circ}{\E''} _\Q$ soit un isomorphisme prolongeant $ \epsilon$.
D'après \ref{coro-stab-ssquot-O-coh},
on en déduit que $\overset{\circ}{\E} |\U'' \in \mathrm{Cris} ^{(m)}(\U'', Y''/\V)$.
Comme on a aussi $\overset{\circ}{\E} |\U' =\overset{\circ}{\E'} \in \mathrm{Cris} ^{(m)}(\U', Y'/\V)$, 
il en résulte que $\overset{\circ}{\E}  \in \mathrm{Cris} ^{(m)}(\U, Y/\V)$.
On dispose de plus de l'isomorphisme $\overset{\circ}{\E} _\Q \riso \E$.
\end{proof}

\subsection{Préservation de la convergence par un morphisme propre et lisse de $k$-variétés lisses}

\begin{prop}
\label{b+cohDOstab}
Soit $b\colon Y' \rightarrow Y$ un morphisme propre et lisse de $k$-variétés lisses.
Le foncteur $ b ^{(m)} _{+}$ se factorise de la manière suivante :
$$b_{+}^{(m)} \colon 
D ^\mathrm{b} _\mathrm{coh} (\smash{\D} ^{(m)} _{Y'})
\cap D ^\mathrm{b} _\mathrm{coh} ( \O _{Y'})
\rightarrow
D ^\mathrm{b} _\mathrm{coh} (\smash{\D} ^{(m)} _{Y})
\cap D ^\mathrm{b} _\mathrm{coh} ( \O _{Y}),$$
où, par abus de notations, nous avons noté 
$D ^\mathrm{b} _\mathrm{coh} (\smash{\D} ^{(m)} _{Y})
\cap D ^\mathrm{b} _\mathrm{coh} ( \O _{Y})$ pour désigner la sous-catégorie pleine de 
$D ^\mathrm{b} _\mathrm{coh} (\smash{\D} ^{(m)} _{Y})$ des complexes à cohomologie $\O _{Y}$-cohérente, de même avec des primes.
\end{prop}

\begin{proof}
Comme $b $ est propre, on sait déjà que l'on dispose de la factorisation
$b_{+}^{(m)} \colon 
D ^\mathrm{b} _\mathrm{coh} (\smash{\D} ^{(m)} _{Y'})
\rightarrow
D ^\mathrm{b} _\mathrm{coh} (\smash{\D} ^{(m)} _{Y})$.
  Soit $\E ' \in D ^\mathrm{b} _\mathrm{coh} (\smash{\D} ^{(m)} _{Y'})
\cap D ^\mathrm{b} _\mathrm{coh} ( \O _{Y'})$.
Il reste à établir que $b ^{(m)} _{+} (\E') \in D ^\mathrm{b} _{\mathrm{coh}} (\O _Y)$.

Soit $F _{Y'} \colon  Y' \rightarrow Y'$ l'endomorphisme absolu de Frobenius.
Comme $F _{Y'} $ est en particulier un morphisme fini,
le théorème de descente du niveau par Frobenius (voir \cite[2]{Be2})
entraîne que le foncteur image inverse $F _{Y'} ^*$
induit une équivalence de catégories entre
$D ^\mathrm{b} _\mathrm{coh} (\smash{\D} ^{(m)} _{Y'})
\cap D ^\mathrm{b} _\mathrm{coh} ( \O _{Y'})$
et
$D ^\mathrm{b} _\mathrm{coh} (\smash{\D} ^{(m+1)} _{Y'})
\cap D ^\mathrm{b} _\mathrm{coh} ( \O _{Y'})$.
Comme il en est de même sur $Y$ et que le foncteur image directe
commute à Frobenius, i.e.
$F _{Y} ^*  \circ b_{+}^{(m)} \riso b_{+}^{(m+1)} \circ F _{Y'} ^* $,
on se ramène alors à traiter le cas où $m =0$.
Comme $b $ est lisse, d'après la remarque de \cite[2.4.6]{Beintro2},
$b ^{(0)} _{+} (\E') \riso
\R b _{*} ( \Omega ^\bullet _{Y'/Y} \otimes _{\O _{Y'}} \E') [d_{Y'/Y}]$.
On a $\Omega ^\bullet _{Y'/Y} \otimes _{\O _{Y'}} \E' \in D ^\mathrm{b} (b ^{-1} \O _Y)$.
On dispose alors de la suite spectrale dans la catégorie des $\O _Y$-modules:
$\mathcal{H} ^{s} \R b _{*} ( \Omega ^{r} _{Y'/Y} \otimes _{\O _{Y'}} \E')
\Rightarrow 
\mathcal{H} ^{n}\R b _{*} ( \Omega ^\bullet _{Y'/Y} \otimes _{\O _{Y'}} \E')$.
Or, comme $b $ est propre et comme 
$\Omega ^{r} _{Y'/Y} \otimes _{\O _{Y'}} \E'$ est $\O _{Y'}$-cohérent,
d'après le théorème de Grothendieck
sur la préservation de la $\O$-cohérence,
cela entraîne que
$\mathcal{H} ^{s} \R b _{*} ( \Omega ^{r} _{Y'/Y} \otimes _{\O _{Y'}} \E')$
est $\O _{Y}$-cohérent.
On en déduit $b ^{(0)} _{+} (\E') \in D ^\mathrm{b} _{\mathrm{coh}} (\O _Y)$.

\end{proof} 

\begin{lemm}
  \label{2.2.14courbe}
Soient $m_0\geq 0$ un entier, $\Y$ un $\V$-schéma formel lisse,
$\E ^{(m _0)}$ un $\smash{\widehat{\D}} ^{(m_0)} _{\Y, \Q}$-module cohérent,
$\O _{\Y, \Q}$-cohérent.
On suppose que,
pour tout entier $m \geq m_0$,
$\smash{\widehat{\D}} ^{(m)} _{\Y, \Q} \otimes _{\smash{\widehat{\D}} ^{(m_0)} _{\Y, \Q}} \E ^{(m _0)}$
est $\O _{\Y, \Q}$-cohérent.
Alors, $\smash{\D} ^{\dag} _{\Y, \Q} \otimes _{\smash{\widehat{\D}} ^{(m_0)} _{\Y, \Q}} \E ^{(m _0)}$
est $\O _{\Y, \Q}$-cohérent.
\end{lemm}

\begin{proof}
  Il suffit de reprendre la preuve de \cite[2.2.14]{caro_courbe-nouveau} à partir de
  {\og
De plus, puisque $\Gamma (X, \E ^{(sm+m_0)})$
  \fg}.
\end{proof}

\begin{prop}
[Préservation de la convergence]
  \label{b+isoccv-pl}
Soient $g\colon \U' \rightarrow  \U$ un morphisme de $\V$-schémas formels séparés et lisses,
$Y$ (resp. $Y'$) un sous-schéma fermé lisse de $U$ (resp. $U'$).
On suppose que $g$ se factorise en un morphisme $b\colon Y' \rightarrow Y$ propre et lisse.

Soit $\E'$ un objet de 
$(F\text{-})D ^\mathrm{b} _\mathrm{isoc} (\U ', Y'/K)$ (voir \ref{nota-6.2.1dev}).
Alors $g _+(\E')$ est un objet de 
$(F\text{-})D ^\mathrm{b} _\mathrm{isoc} (\U , Y/K)$.
\end{prop}

\begin{proof}
Le foncteur $g _+$ commutant à Frobenius, il suffit de vérifier la proposition sans structure de Frobenius.
Or, la catégorie $\mathrm{Isoc} ^{\dag \dag} (\U, Y/K)$ 
est stable par suite spectrale, i.e., une suite spectrale dont les flèches initiales sont dans $\mathrm{Isoc} ^{\dag \dag} (\U, Y/K)$ a pour aboutissement des objects de $\mathrm{Isoc} ^{\dag \dag} (\U, Y/K)$. 
 Quitte à utiliser la deuxième suite spectrale d'hypercohomologie du foncteur $g _+$, 
  on peut donc supposer que $\E'\in \mathrm{Isoc} ^{\dag \dag} (\U', Y'/K)$.   
Le morphisme $g$ se factorise en son graphe $\gamma _g \colon  \U ' \hookrightarrow \U' \times \U$ 
suivi de la projection canonique 
$\U' \times \U \to \U$.
Via le théorème de Berthelot-Kashiwara,
on vérifie que le foncteur $\gamma _{g, +}$ induit une 
  équivalence canonique
  entre les catégories 
    $\mathrm{Isoc} ^{\dag \dag}  (\U ', Y'/K)$ et $\mathrm{Isoc} ^{\dag \dag} (\U '\times \U, Y'/K)$.
  Quitte à remplacer $\U'$ par $\U' \times \U$,
 on se ramène ainsi au cas où $g$ lisse. 
  Comme $\E'$ est à support propre sur $Y$ et donc sur $\U$, il découle de 
  \ref{casliss631dev}
  que $g _+ (\E') \in D ^\mathrm{b} _\mathrm{surhol} (\D ^\dag _{\U,\Q})$ (en particulier $g _+ (\E') \in D ^\mathrm{b} _\mathrm{coh} (\D ^\dag _{\U,\Q})$).

D'après la caractérisation de \cite[2.5.10]{caro-construction} des $\D ^\dag _{\U,\Q}$-modules cohérents qui sont dans l'image essentielle de
$\sp _{Y \hookrightarrow \U,+}$, 
ce qu'il reste à établir est alors local en $\U$.
On peut donc supposer qu'il existe un relèvement $v\colon \Y \hookrightarrow \U$.
Notons $\U'':= \U ' \times _{\U} \Y$, $v ''\colon \U'' \hookrightarrow \U'$ et $h\colon \U'' \rightarrow \Y$
les projections canoniques. Comme $g$ est lisse, $\U'' $ est un $\V$-schéma formel lisse.
Comme $v$ est une immersion fermée, $v''$ l'est aussi.
L'immersion $Y' \hookrightarrow \U''$ est donc aussi fermée.
Via le théorème de Berthelot-Kashiwara (voir \ref{Berthelot-Kashiwara}), 
$v ^{\prime \prime !} (\E ') $ est un $\D ^\dag _{\U'',\Q}$-module surholonome et $ v '' _+ v ^{\prime \prime !} (\E ') \riso \E'$.
De plus, encore via la description \cite[2.5.10]{caro-construction}, on vérifie que 
$v ^{\prime \prime !} (\E ') $ est dans l'image essentielle de
$\sp _{Y' \hookrightarrow \U'',+}$. 
On se ramène ainsi au cas où $g=h$, i.e. $\Y= \U$ et $\U'' = \U'$.
Pour terminer la preuve, il s'agit d'établir $g_+ (\E')\in D ^\mathrm{b} _{\mathrm{coh}} (\O _{\Y,\Q})$.
Nous procédons pour cela de la manière suivante.

$\bullet$ \'Etape $1$.  
Notons $ v '$ l'immersion fermée canonique $Y' \hookrightarrow U'$ et 
soit $(\U '_{\alpha}) _{\alpha \in \Lambda}$ un recouvrement d'ouverts de $\U'$ satisfaisant aux conditions
de \ref{notat-construc}. Nous reprenons les notations de \ref{notat-construc} (en ajoutant des primes).
Comme $\E'$ est dans l'image essentielle de
$\sp _{Y' \hookrightarrow \U',+}$, d'après \cite[2.5.10]{caro-construction},
$\E'$ est à support dans $Y'$ et,
pour tout $\alpha \in \Lambda$,
$\E ' _\alpha := v  _\alpha ^{\prime !} (\E' |\U ' _\alpha)$ est $\O _{\Y ' _\alpha,\Q}$-cohérent.
Avec les notations de \ref{defindonnederecoldag},
il en résulte 
$\mathcal{L}oc (\E')\in 
\mathrm{Isoc} ^{\dag \dag} (Y',\, (\Y ' _\alpha) _{\alpha \in \Lambda}/K)$. 
D'après \ref{crist-loc-def}, pour chaque entier $m$ (entier non fixé parcourant $\N$), on dispose d'une inclusion canonique
$\mathrm{Isoc} ^{\dag \dag} (Y',\, (\Y '_\alpha) _{\alpha \in \Lambda}/K) \subset \mathrm{Coh} ^{(m)} (Y',\, (\Y '_\alpha) _{\alpha \in \Lambda}/K)$
et il existe un élément  $\FF ^{\prime (m)}\in \mathrm{Cris} ^{(m)} (Y',\, (\Y '_\alpha) _{\alpha \in \Lambda}/\V)$ tel que 
 dans
$\mathrm{Coh} ^{(m)} (Y',\, (\Y '_\alpha) _{\alpha \in \Lambda}/K)$
l'on ait l'isomorphisme $\FF ^{\prime (m)} _{\Q } \riso \mathcal{L}oc (\E')$. 

$\bullet$ \'Etape $2$ :
{\it Posons $\G ^{\prime (m)}:= \mathcal{R}ecol ^{(m)} (\FF ^{\prime (m)})$.
Alors $g ^{(m)} _+ (\G ^{\prime (m)}) \in D ^\mathrm{b} _\mathrm{coh} (\smash{\widehat{\D}} ^{(m)} _{\Y})\cap D ^\mathrm{b} _{\mathrm{coh}} (\O _\Y)$.}

\noindent En effet, posons $\FF ^{\prime (m)} _0:= \FF ^{\prime (m)} \otimes _\V k $, $\G ^{\prime (m)} _0:= \G ^{\prime (m)} \otimes _\V k $.
Par \ref{diagcomm-otimesk}, on dispose alors de l'isomorphisme
$\smash{\D} ^{(m)} _{U'}$-linéaire :
$\G ^{\prime (m)} _0\riso v^{\prime (m)} _{+} (\FF ^{\prime (m)} _0)$.
Cela implique par transitivité de l'image directe :
$g ^{(m)} _{+} (\G ^{\prime (m)} _0) \riso b ^{(m)} _{+} (\FF ^{\prime (m)} _0)$.
Comme $b $ est propre et lisse, d'après \ref{b+cohDOstab}, il en résulte 
$g ^{(m)} _{+} (\G ^{\prime (m)} _0) 
\in D ^\mathrm{b} _{\mathrm{coh}} (\D^{(m)} _Y) \cap  D ^\mathrm{b} _{\mathrm{coh}} (\O _Y)$.
Or, 
$g ^{(m)} _+ (\G ^{\prime (m)}) \otimes _\V ^\L k \riso 
g ^{(m)} _{+} (\G ^{\prime (m)} \otimes _\V ^\L k)$. 
Or, comme $\G ^{\prime (m)}$ est sans $p$-torsion, 
$\G ^{\prime (m)} \otimes _\V ^\L k \riso \G ^{\prime (m)} _0$. D'où:
$g ^{(m)} _+ (\G ^{\prime (m)}) \otimes _\V ^\L k  \in D ^\mathrm{b} _{\mathrm{coh}} (\D ^{(m)}_Y) \cap D ^\mathrm{b} _{\mathrm{coh}} (\O _Y)$.
Comme de plus $g ^{(m)} _+ (\G ^{\prime (m)})\in D ^\mathrm{b} _\mathrm{qc} (\smash{\widehat{\D}} ^{(m)} _{\Y})$, d'après la remarque de \cite[3.2.2]{Beintro2} (et de manière analogue en remplaçant $\D$ par $\O$),
cela implique que
$g ^{(m)} _+ (\G ^{\prime (m)}) \in D ^\mathrm{b} _\mathrm{coh} (\smash{\widehat{\D}} ^{(m)} _{\Y})\cap D ^\mathrm{b} _{\mathrm{coh}} (\O _\Y)$.

$\bullet$ \'Etape $3$ : {\it Conclusion}.

a) Comme $g ^{(m)} _+ (\G ^{\prime (m)}) \otimes _\Z \Q \riso
g ^{(m)} _+ (\G ^{\prime (m)} _\Q)$ (où le dernier foncteur $g ^{(m)} _+$ est le foncteur image directe en tant
que $\smash{\widehat{\D}} ^{(m)} _{\U ',\Q}$-module cohérent),
il découle de l'étape $2$ que
$g ^{(m)} _+ (\G ^{\prime (m)} _\Q) \in D ^\mathrm{b} _\mathrm{coh} (\smash{\widehat{\D}} ^{(m)} _{\Y,\Q})\cap  D ^\mathrm{b} _{\mathrm{coh}} (\O _{\Y,\Q})$.
Comme en particulier pour tout entier $m$ le complexe $g ^{(m)} _+ (\G _\Q ^{\prime (m)})\in D ^\mathrm{b} _\mathrm{coh} (\smash{\widehat{\D}} ^{(m)} _{\Y,\Q})$,
on vérifie alors de manière analogue à \cite[3.5.3]{Beintro2}
(l'hypothèse de \cite[3.5.3]{Beintro2} que le morphisme soit propre n'intervient que pour obtenir la cohérence de l'image directe, en effet il suffit d'utiliser \cite[2.4.3]{Beintro2})
pour tous entiers $m' \geq m\colon \smash{\widehat{\D}} ^{(m')} _{\Y,\Q}
 \otimes _{\smash{\widehat{\D}} ^{(m)} _{\Y,\Q}}g ^{(m)} _+ (\G _\Q ^{\prime (m)})
\riso
g ^{(m')} _+ (\smash{\widehat{\D}} ^{(m')} _{\U ',\Q}
 \otimes _{\smash{\widehat{\D}} ^{(m)} _{\U ',\Q}} \G ^{\prime(m)} _\Q)$.
De façon similaire à la preuve de \cite[4.3.8]{Beintro2},
cela entraîne : 
$g_+ (\smash{\D} ^{\dag} _{\U',\Q} \otimes _{\smash{\widehat{\D}} ^{(m)} _{\U',\Q}} \G _\Q ^{\prime (m)}) \riso \smash{\D} ^{\dag} _{\Y,\Q} \otimes _{\smash{\widehat{\D}} ^{(m)} _{\Y,\Q}}g ^{(m)} _+ (\G _\Q ^{\prime (m)})$.

b) Par \ref{etsurQ?}.\ref{commQ},
on obtient l'isomorphisme
$\smash{\widehat{\D}} ^{(m)} _{\U ',\Q}$-linéaire :
$\G ^{\prime (m)} _\Q \riso \mathcal{R}ecol ^{(m)} _{\Q} (\FF ^{\prime (m)}_\Q)$.
D'après \ref{etsurQ?}.\ref{loc-recol-comm-niveaux}, 
pour tout entier $m' \geq m$, il en résulte 
$\smash{\widehat{\D}} ^{(m')} _{\U',\Q} \otimes _{\smash{\widehat{\D}} ^{(m)} _{\U',\Q}} \G ^{\prime (m)} _\Q 
\riso \mathcal{R}ecol ^{(m')} _{\Q}  (\smash{\widehat{\D}} ^{(m')} _{Y',\Q} \otimes _{\smash{\widehat{\D}} ^{(m)} _{Y',\Q}} \FF ^{\prime (m)}_\Q) $.
Or, comme $\FF ^{\prime (m)}_\Q \in \mathrm{Isoc} ^{\dag \dag} (Y',\, (\Y '_\alpha) _{\alpha \in \Lambda}/K) $, 
il résulte de \cite[3.1]{Be0} l'isomorphisme canonique $\FF ^{\prime (m')} _{\Q } \riso \smash{\widehat{\D}} ^{(m')} _{Y',\Q} \otimes _{\smash{\widehat{\D}} ^{(m)} _{Y',\Q}} \FF ^{\prime (m)}_\Q$.
D'où : $\smash{\widehat{\D}} ^{(m')} _{\U',\Q} \otimes _{\smash{\widehat{\D}} ^{(m)} _{\U',\Q}} \G ^{\prime (m)} _\Q 
\riso \mathcal{R}ecol ^{(m')} _{\Q}  ( \FF ^{\prime (m')}_\Q)
\riso \G ^{\prime (m')}_\Q $.
D'après a),  cela implique 
\begin{equation}
 \label{b+isoccv-pl-etapeb}
\smash{\widehat{\D}} ^{(m')} _{\Y,\Q}
 \otimes _{\smash{\widehat{\D}} ^{(m)} _{\Y,\Q}}g ^{(m)} _+ (\G _\Q ^{\prime (m)}) 
 \in D ^\mathrm{b} _\mathrm{coh} (\smash{\widehat{\D}} ^{(m')} _{\Y,\Q})\cap  D ^\mathrm{b} _{\mathrm{coh}} (\O _{\Y,\Q}).
\end{equation}

c) De plus, on déduit de \cite[3.1]{Be0} l'isomorphisme 
$\mathcal{L}oc (\E') \riso \smash{\D} ^{\dag} _{Y',\Q} \otimes _{\smash{\widehat{\D}} ^{(m)} _{Y',\Q}} \FF ^{\prime (m)}_\Q$.
D'après \ref{etsurQ?}.\ref{loc-recol-comm-niveauxlim}, il en résulte l'isomorphisme
$\mathcal{R}ecol \circ \mathcal{L}oc (\E')
\riso  
\smash{\D} ^{\dag} _{\U',\Q} \otimes _{\smash{\widehat{\D}} ^{(m)} _{\U',\Q}} \G ^{\prime (m)} _\Q  $.
Par \ref{prop-donnederecol-dag}, 
on en déduit 
$\E' \riso \smash{\D} ^{\dag} _{\U',\Q} \otimes _{\smash{\widehat{\D}} ^{(m)} _{\U',\Q}} \G ^{\prime (m)} _\Q  $.
D'après a), cela entraîne 
\begin{equation}
 \label{b+isoccv-pl-etapec}
 g _+ (\E ') \riso \smash{\D} ^{\dag} _{\Y,\Q} \otimes _{\smash{\widehat{\D}} ^{(m)} _{\Y,\Q}}g ^{(m)} _+ (\G _\Q ^{\prime (m)}).
\end{equation}

d) Pour tout entier $j$, 
$\mathcal{H} ^{j} (g ^{(0)} _+ (\G _\Q ^{\prime (0)}))$ est un $\smash{\widehat{\D}} ^{(0)} _{\Y,\Q}$-module cohérent, $\O _{\Y,\Q}$-cohérent.
Comme l'extension $\smash{\widehat{\D}} ^{(0)} _{\Y,\Q} \rightarrow \smash{\widehat{\D}} ^{(m)} _{\Y,\Q}$ est plate, il découle de 
\ref{b+isoccv-pl-etapeb} que 
$\smash{\widehat{\D}} ^{(m)} _{\Y,\Q}
 \otimes _{\smash{\widehat{\D}} ^{(0)} _{\Y,\Q}} (\mathcal{H} ^{j} g ^{(0)} _+ (\G _\Q ^{\prime (0)}) )$ est
 $\O _{\Y,\Q}$-cohérent. De plus, comme l'extension $\smash{\widehat{\D}} ^{(0)} _{\Y,\Q} \rightarrow \smash{\D} ^{\dag} _{\Y,\Q}$ est plate, via \ref{b+isoccv-pl-etapec}, on obtient
 $\mathcal{H} ^{j} (g _+ (\E ') ) \riso 
 \smash{\D} ^{\dag} _{\Y,\Q}
 \otimes _{\smash{\widehat{\D}} ^{(0)} _{\Y,\Q}} (\mathcal{H} ^{j} g ^{(0)} _+ (\G _\Q ^{\prime (0)}) )$.
Par \ref{2.2.14courbe}, 
il en résulte que
$\mathcal{H} ^{j} (g _+ (\E ') )$ est $\O _{\Y,\Q}$-cohérent.

\end{proof}

\subsection{Préservation de la surconvergence}

\begin{theo}
\label{b+proprelisse-varlisse-gen}
Soit $\theta =(f, a,b) \colon  (\PP', T',X',Y')\to (\PP, T,X,Y)$ un morphisme de $d$-cadres lisses en dehors du diviseur
tel que 
$a$ soit propre, 
$Y'= a  ^{-1} (Y)$
et $b$ soit lisse.
Soit $\E' $ un objet de 
$(F\text{-})D ^\mathrm{b} _\mathrm{isoc}  (\PP', T', X'/K)$.
Alors $f _{+} (\E')$ est un objet de $(F\text{-})D ^\mathrm{b} _\mathrm{isoc}  (\PP, T, X/K)$.
\end{theo}

\begin{proof}
0) Quitte à utiliser la deuxième suite spectrale d'hypercohomologie du foncteur
$f _+$, on peut supposer que $\E' \in (F\text{-})\mathrm{Isoc} ^{\dag \dag} (\PP', T', X'/K)$.
D'après \ref{indt-X}, on peut supposer que $Y'$ est dense dans $X'$.
En outre, quitte à considérer la décomposition en somme directe de $\E'$ sur les composantes irréductibles de $Y'$, 
on se ramène facilement au cas où $Y'$ est intègre. 

1) Dans un premier temps, on se ramène au cas où $f$ est lisse et $T'= f ^{-1} (T)$. 
Posons $\PP'':= \PP' \times \PP$, $q\colon \PP''\to \PP$ et $q'\colon \PP''\to \PP'$ les projections canoniques.
Le morphisme $f$ se décompose en le graphe de $f$ noté
$\gamma _f\colon  \PP' \hookrightarrow \PP' \times \PP$
suivi de la projection $q$.  Via $\gamma _f$, $X'$ est aussi un sous-schéma fermé de $\PP''$.
D'après \cite[5.4.1]{caro-pleine-fidelite} (utilisé pour les morphismes identités), 
comme $Y' = X' \setminus q ^{-1} (T)=X' \setminus q ^{\prime -1} (T')$,
on a l'égalité $  (F\text{-})\mathrm{Isoc} ^{\dag \dag} (\PP'', q ^{\prime -1} (T'), X'/K)=  (F\text{-})\mathrm{Isoc} ^{\dag \dag} (\PP'', q ^{ -1} (T), X'/K)$.
Toujours d'après \cite[5.4.1]{caro-pleine-fidelite} (utilisé pour $\theta = (\gamma _{f}, Id)$),
comme $\gamma _{f} ^{-1} ( q ^{\prime -1} (T'))=T'$ est un diviseur, 
le foncteur $\gamma _{f, +}$ induit 
une équivalence de catégories entre 
$(F\text{-})\mathrm{Isoc} ^{\dag \dag} (\PP', T', X'/K)$
et
$  (F\text{-})\mathrm{Isoc} ^{\dag \dag} (\PP'', q ^{\prime -1} (T'), X'/K)$. 
Les catégories $ (F\text{-})\mathrm{Isoc} ^{\dag \dag} (\PP'', q ^{ -1} (T), X'/K)$
et 
$(F\text{-})\mathrm{Isoc} ^{\dag \dag} (\PP', T', X'/K)$
sont donc canoniquement équivalentes. 
Quitte à remplacer $\PP'$ par $\PP''$ et $T'$ par $q ^{ -1} (T)$, 
on peut donc supposer $T' = f ^{-1}(T)$ et $f$ est lisse.

2) D'après \cite[5.3.1.2]{caro-pleine-fidelite}, 
comme $u\circ a $ est propre, 
on en déduit que 
$f _+ (\E') \in (F\text{-})D ^\mathrm{b} _\mathrm{surcoh} (\D ^\dag _{\PP} (\hdag T) _\Q)$.

3) Notons $\U := \PP \setminus T$ et $\U':=\PP' \setminus T'$.
Le morphisme induit $g\colon \U' \to \U$ est un prolongement du morphisme propre et lisse $Y'\to Y$ via les immersions fermées $Y\hookrightarrow \U$, $Y'\hookrightarrow \U'$.
Grâce à \ref{b+isoccv-pl},
comme $ \E'|\U' \in (F\text{-})\mathrm{Isoc} ^{\dag \dag} (\U', Y'/K)$,
comme $f _+ (\E')|\U \riso g _{+} ( \E'|\U')$, il en résulte que les espaces de cohomologie de $f _+ (\E')|\U$ proviennent
de $(F\text{-})$isocristaux convergents sur $Y$, i.e.
$f _+ (\E')|\U  \in (F\text{-})D ^\mathrm{b} _\mathrm{isoc}  (\U, Y/K)$. D'après \cite[2.2.12]{caro_courbe-nouveau},  
comme aussi d'après l'étape $2$ on sait 
$f _+ (\E') \in (F\text{-})D ^\mathrm{b} _\mathrm{coh} (\D ^\dag _{\PP} (\hdag T) _\Q)$, 
cela implique que les espaces de cohomologie de 
$f _+ (\E')$ proviennent de $(F\text{-})$isocristaux sur $Y$ surconvergents le long de $X \setminus Y$.
\end{proof}

\section{Catégories sur les couples de variétés et préservation de la surcohérence}

Nous construisons dans ce chapitres des catégories de $\D$-modules arithmétiques 
sur des couples de variétés propremement $d$-réalisables. 
Puis nous traduisons le théorème de préservation de la surcohérence établie 
dans \ref{b+proprelisse-varlisse-gen} dans ce nouveau $d$-cadre (e.g., voir \ref{b+proprelisse}). 

\subsection{Definitions}

 \begin{defi}
\label{defi-(d)plong}
On définit la catégorie des couples de $k$-variétés $d$-réalisables de la manière suivante: 
Soient $X$ une $k$-variété et $Y$ un ouvert de $X$.  
Le couple $(Y,X)$ est un {\og couple de $k$-variétés $d$-réalisable \fg} 
s'il existe un $d$-cadre de la forme $(\PP, T,X,Y)$ 
(voir les conventions de \ref{defi-cad}).

Soient $(Y',X')$ et $(Y,X)$ deux couples de $k$-variétés $d$-réalisables.
Un morphisme $(b,a)\colon (Y',X')\to (Y,X)$ de couples de $k$-variétés $d$-réalisables est un morphisme de variétés
$a \colon X'\to X$ induisant la factorisation $b\colon Y' \to Y$.
On pourra noter abusivement $a$ pour $(b,a)$.
\end{defi}

On aura besoin de la notion un peu plus restrictive :
\begin{defi}
\label{defi-(d)plongprop}
La catégorie des $k$-variétés proprement $d$-réalisables 
est la sous-catégorie pleine de la catégorie des $k$-variétés $d$-réalisables
dont les objets sont les couples 
$(Y,X)$ tels qu'il existe
un $\V$-schéma formel $\PP $ propre et lisse,
un diviseur $T$ de $P$ et une immersion (non nécessairement fermée) $X \hookrightarrow \PP$
vérifiant $Y = X \setminus T$.

Un morphisme $(b,a)\colon (Y',X') \to (Y,X)$ de couples de $k$-variétés $d$-réalisables 
(resp. de couples de $k$-variétés proprement $d$-réalisables) est {\og complet\fg} (resp. {\og propre\fg}) 
lorsque le morphisme sous-jacent $a$ est propre (resp. $a$ et $b$ sont propres).
\end{defi}

 Complétons à présent la définition de \cite[5.4.4]{caro-pleine-fidelite}.

\begin{defi}
\label{defi-var-dplong}
    \begin{itemize}
\item Une $k$-variété $Y$ est {\og $d$-réalisable\fg}  
s'il existe un couple de $k$-variétés $d$-réalisables de la forme $(Y,X)$ avec $X$ propre. 

\item Une $k$-variété $Y$ est {\og proprement $d$-réalisable\fg} 
s'il existe un couple de $k$-variétés proprement $d$-réalisables de la forme $(Y,X)$ avec $X$ propre. 
 \end{itemize}
\end{defi}

\subsection{Indépendance par rapport au schéma formel}

Traitons d'abord à part le cas cohérent:
\begin{prop}
\label{nota-Dsurcv-cpart}
  Soient $X$ une $k$-variété lisse et $Y$ un ouvert de $X$ tels que $(Y,X)$ soit proprement $d$-réalisable (\ref{defi-(d)plongprop}).
Choisissons $\widetilde{\PP} $ un $\V$-schéma formel propre et lisse,
$\widetilde{T}$ un diviseur de $\widetilde{P}$ tels qu'il existe une immersion $X \hookrightarrow \widetilde{\PP}$
induisant l'égalité $Y = X \setminus \widetilde{T}$. 
Choisissons $\PP$ un ouvert de $\widetilde{\PP}$ contenant $X$ tel que
l'immersion $X \hookrightarrow \PP$ canoniquement induite soit fermée. On note alors $T := P \cap \widetilde{T}$ le diviseur de $P$ induit. 

Les catégories 
$(F\text{-})D ^\mathrm{b} _\mathrm{coh}  (\PP, T, X/K)$
ne dépendent pas, à isomorphisme canonique près, 
du choix du $\V$-schéma formel propre et lisse $\widetilde{\PP}$, de l'ouvert $\PP$ de $\widetilde{\PP}$, de l'immersion fermée $X \hookrightarrow \PP$ et du diviseur $\widetilde{T}$ de $\widetilde{P}$ tels 
$T = P \cap \widetilde{T}$ et $Y =X \setminus \widetilde{T}$
(mais seulement de $(Y,X)/K$).
On la notera alors respectivement sans ambiguïté 
$(F\text{-})D ^\mathrm{b} _\mathrm{coh} (\D ^\dag _{(Y,X)/K}) $.
Ses objets sont appelés les {\og $(F\text{-})$complexes cohérents de $\D ^\dag _{(Y,X)/K}$-modules \fg}
ou {\og les $(F\text{-})$complexes cohérents de $\D$-modules arithmétiques sur $(Y,X)/K$ \fg}.
\end{prop}

\begin{proof}
Faisons un tel second choix : soient $\widetilde{\PP}' $ un $\V$-schéma formel propre et lisse,
$\PP '$ un ouvert de $\widetilde{\PP}' $,
$\widetilde{T}'$ un diviseur de $\widetilde{P}'$, 
tels qu'il existe une immersion fermée $X \hookrightarrow \PP'$
induisant l'égalité $Y = X \setminus \widetilde{T}'$. 
Posons $T':=\widetilde{T}' \cap P'$   
et
$\widetilde{\PP}'':=\widetilde{\PP}\times \widetilde{\PP}'$.
Soient $\widetilde{q}\colon \widetilde{\PP}''\to \widetilde{\PP}$,
$\widetilde{q}'\colon \widetilde{\PP}''\to \widetilde{\PP}'$ les projections (propres et lisses) canoniques, 
$\widetilde{T}'' _{1}:= \widetilde{q} ^{-1} (\widetilde{T})$, $\PP '' _{1} := \widetilde{q} ^{-1} (\PP)$, $T'' _{1}:=\widetilde{T}'' _{1}\cap \PP '' _{1}$,
$\widetilde{T}'' _{2}:= \widetilde{q} ^{\prime -1} (\widetilde{T}')$, $\PP '' _{2} := \widetilde{q} ^{\prime -1} (\PP')$, $T'' _{2}:=\widetilde{T}'' _{2}\cap \PP '' _{2}$,
$\widetilde{T}'' := \widetilde{T}'' _{1} \cup \widetilde{T}'' _{2}$, 
$\PP '':= \PP '' _{1} \cap \PP '' _{2}$.
Notons $q\colon \PP'' _{1} \rightarrow \PP$ le morphisme induit $\widetilde{q}$.
D'après \cite[5.4.3]{EGAI}, le graphe de l'immersion $X \hookrightarrow P'$ 
est fermé. Comme l'immersion diagonale $X \hookrightarrow \PP '' _{1}$ se décompose 
en $X \hookrightarrow X \times \widetilde{\PP}' \hookrightarrow \PP \times \widetilde{\PP}' = \PP '' _{1}$,
celle-ci est donc fermée. 
De plus, $Y = X \setminus T'' _{1}$.  
Comme $q$ est propre, 
il résulte alors de \ref{coh-PXTindtPbis}.\ref{coh-PXTindtP-iv} que
les foncteurs $\R \underline{\Gamma} ^\dag _X  q ^! $ et
$q _+$ induisent des équivalences quasi-inverses 
entre les catégories 
$(F\text{-})D ^\mathrm{b} _\mathrm{coh}  (\PP, T, X/K)$ 
et 
$(F\text{-})D ^\mathrm{b} _\mathrm{coh}  (\PP''_{1}, T'' _{1}, X/K)$.
De même, comme on dispose de l'immersion fermée $X \hookrightarrow \PP ''$ et comme $Y = X \setminus T''$, en appliquant \ref{coh-PXTindtPbis}.\ref{coh-PXTindtP-iv} à l'immersion ouverte 
$\PP '' \subset \PP '' _{1}$, on obtient des équivalences quasi-inverses 
entre les catégories 
$(F\text{-})D ^\mathrm{b} _\mathrm{coh}  (\PP'', T'', X/K)$ 
et 
$(F\text{-})D ^\mathrm{b} _\mathrm{coh}  (\PP''_{1}, T'' _{1}, X/K)$.
Ainsi, les catégories 
$(F\text{-})D ^\mathrm{b} _\mathrm{coh}  (\PP'', T'', X/K)$ 
et 
$(F\text{-})D ^\mathrm{b} _\mathrm{coh}  (\PP, T, X/K)$
sont canoniquement équivalentes.
De manière symétrique (i.e. on remplace $1$ par $2$), on vérifie 
que 
les catégories 
$(F\text{-})D ^\mathrm{b} _\mathrm{coh}  (\PP'', T'', X/K)$ 
et 
$(F\text{-})D ^\mathrm{b} _\mathrm{coh}  (\PP', T', X/K)$
sont canoniquement équivalentes.
D'où le résultat.

\end{proof}

Lorsque l'on remplace la notion de cohérence par celle de surcohérence ou de surholonomie,
nous utiliserons 
la stabilité de ces deux notions par les opérations cohomologiques (e.g. image directe, image inverse extraordinaire)
pour obtenir une indépendance telle que celle de la proposition \ref{nota-Dsurcv-cpart}
sans supposer que la compactification partielle $X$ est lisse. 
La première étape, est d'améliorer le précédent lemme \ref{lemme-coh-PXTindtP} en supprimant l'hypothèse de lissité de $X$.
 On répond donc positivement à la conjecture énoncée dans \cite[3.2.5]{caro_surcoherent}:
\begin{lemm}
[Indépendance par rapport au diviseur]
\label{lemme-coh-PXTindtPGEN}
Soient $\PP$ un $\V$-schéma formel séparé et lisse, $T, T'$ deux diviseurs de $P$,
$X$ un sous-schéma fermé de $P$ tels que $T\cap X= T'\cap X$. 
Posons
$\mathfrak{C}=(F\text{-})\mathrm{Surcoh}$
ou 
$\mathfrak{C}=(F\text{-})\mathrm{Surhol}$
ou 
$\mathfrak{C}=
(F\text{-})D ^\mathrm{b} _\mathrm{surcoh} $
ou 
$\mathfrak{C}=
(F\text{-})D ^\mathrm{b} _\mathrm{surhol} $.
Lorsque $X \setminus T$ est lisse, on pose aussi  
$\mathfrak{C}=(F\text{-})\mathrm{Isoc}^{\dag \dag}$
ou
$\mathfrak{C}=
(F\text{-})D ^\mathrm{b} _\mathrm{isoc} $.
On obtient alors les égalités 
$\mathfrak{C} (\PP, T, X/K)=\mathfrak{C}  (\PP, T', X/K)$. 
\end{lemm}

\begin{proof}
Le cas surholonome étant déjà traité dans \cite{caro_surholonome}, 
contentons-nous de prouver les autres cas.
Quitte à considérer $T \cup T'$, on peut supposer que $T \supset T'$.
Soit 
$\FF \in \mathfrak{C} (\PP, T', X/K)$.
Comme $\R \underline{\Gamma} ^\dag _{X} (\FF) \riso \FF$, on obtient les isomorphismes
$$\R \underline{\Gamma} ^\dag _{T} (\FF)
\riso 
\R \underline{\Gamma} ^\dag _{X\cap T} (\FF)
=
\R \underline{\Gamma} ^\dag _{X\cap T'} (\FF)
\riso
\R \underline{\Gamma} ^\dag _{T'} (\FF)
=0.$$
On en déduit que la flèche canonique
$\FF  \to \FF(\hdag T) $
est un isomorphisme.
On a donc vérifié que l'on dispose de l'inclusion:
$\mathfrak{C} (\PP, T', X/K)\subset \mathfrak{C}  (\PP, T, X/K)$ (en effet, 
un $\smash{\D} ^{\dag} _{\X } (\hdag T ) _{\Q}$-module qui est $\smash{\D} ^{\dag} _{\X } (\hdag T ') _{\Q}$-surcohérent
est $\smash{\D} ^{\dag} _{\X } (\hdag T ) _{\Q}$-surcohérent ; c'est la réciproque qui est fausse en générale).

Vérifions à présent l'inclusion inverse (qui est moins immédiate que pour \cite[3.2.4]{caro_surcoherent}). 
Les autres cas en résultant aisément, il suffit 
de traiter le cas où $\mathfrak{C}=(F\text{-})D ^\mathrm{b} _\mathrm{surcoh}$.
Par \ref{indt-X}, quitte à changer $X$, 
on peut supposer $X \setminus T$ dense dans $X$.
Soit $\E \in \mathfrak{C} (\PP, T, X/K)$.
On procède par récurrence lexicographique sur la dimension de $X$ et le nombre des composantes irréductibles de $X$ dimension $\dim X$.
Comme la surcohérence est préservée par les foncteurs espace de cohomologie (i.e. de la forme $\mathcal{H} ^{l}$),
on se ramène à valider le lemme lorsque 
$\mathfrak{C}=(F\text{-})\mathrm{Surcoh}$.
D'après \cite[6.2.1]{caro-pleine-fidelite},   
il existe un diviseur $\widetilde{T}\supset T$ de $P$ tel que,
  en notant $\widetilde{Y} := X \setminus \widetilde{T}$,
  $\widetilde{Y}$ soit intègre, lisse, inclus et dense dans une composante irréductible de $X$ de dimension $\dim X$ et
  $\widetilde{\E}:=\E (\hdag \widetilde{T}) \in (F\text{-})\mathrm{Isoc} ^{\dag \dag}( \PP, \widetilde{T}, X/K)$.
  Notons $\widetilde{X}$ l'adhérence de $\widetilde{Y}$ dans $X$. On obtient 
$  (F\text{-})\mathrm{Isoc} ^{\dag \dag}( \PP, \widetilde{T}, X/K) =  (F\text{-})\mathrm{Isoc} ^{\dag \dag}( \PP, \widetilde{T}, \widetilde{X}/K)$
(voir \ref{indt-X}).

D'après \cite[5.3.1]{caro-pleine-fidelite} (utilisé lorsque le morphisme $f$ est l'identité),
il existe un diagramme commutatif de la forme
  \begin{equation}
  \label{diag-casliss631dev}
  \xymatrix {
  {\widetilde{X}'} \ar[r] ^-{u'} \ar[d] _{a '} & {\P ^N _{P}} \ar[r] \ar[d] &
  {\P ^N _{\PP}} \ar[d] ^-{q} \\
  {\widetilde{X}} \ar[r] ^-{u} & {P} \ar[r] & {\PP}
   }
\end{equation}
  où $u$ est l'immersion fermée canonique, 
  $\widetilde{X}'$ est lisse sur $k$, $q$ est la projection canonique, $u'$ est une immersion fermée, 
  $a  ^{\prime -1} (T\cap \widetilde{X})$ est un diviseur à croisements normaux strict de $\widetilde{X}'$,
  $a '$ est propre, surjectif, génériquement fini et étale. 
  Posons $\widetilde{T}':= q ^{-1} (\widetilde{T})$, $T'':=q ^{-1} (T')$, 
  $\E':= \R \underline{\Gamma} _{\widetilde{X}'} ^\dag \circ q ^{!} (\E)$
  et
  $\widetilde{\E}':= \R \underline{\Gamma} _{\widetilde{X}'} ^\dag \circ q ^{!} (\widetilde{\E})$.
Or, d'après le cas de la compactification lisse déjà traité via le lemme \ref{lemme-coh-PXTindtP},
on a l'égalité
$\E' \in 
(F\text{-})D ^\mathrm{b} _\mathrm{surcoh}  (\P ^N _{\PP}, q ^{-1} (T), \widetilde{X}'/K)
=(F\text{-})D ^\mathrm{b} _\mathrm{surcoh} (\P ^N _{\PP},T'' , \widetilde{X}'/K)$.
Comme $\widetilde{\E}' \riso \E' (\hdag \widetilde{T}')$, il en résulte que 
$\widetilde{\E}'\in (F\text{-})D ^\mathrm{b} _\mathrm{surcoh} (\P ^N _{\PP},T'' , \widetilde{X}'/K)$.
Par stabilité de la surcohérence par l'image directe d'un morphisme propre, 
il en résulte 
$q  _+(\widetilde{\E}')\in (F\text{-})D ^\mathrm{b} _\mathrm{surcoh}  (\PP, T', \widetilde{X}/K)$.
Or, d'après  \cite[5.3.1]{caro-pleine-fidelite}, 
$\widetilde{\E}$ est un facteur direct de $q  _+(\widetilde{\E}')$.
On en déduit que 
$\widetilde{\E}\in (F\text{-})D ^\mathrm{b} _\mathrm{surcoh}  (\PP, T', \widetilde{X}/K)
\subset (F\text{-})D ^\mathrm{b} _\mathrm{surcoh}  (\PP, T',X/K)$.
Or 
$\R \underline{\Gamma} ^\dag _{\widetilde{T}} (\E) \in 
(F\text{-})D ^\mathrm{b} _\mathrm{surcoh}  (\PP, T, X\cap \widetilde{T}/K)
=(F\text{-})D ^\mathrm{b} _\mathrm{surcoh}  (\PP, T',X \cap \widetilde{T}/K) 
\subset
(F\text{-})D ^\mathrm{b} _\mathrm{surcoh}  (\PP, T', X /K) $, 
l'égalité découlant de l'hypothèse de récurrence.
On conclut via le triangle de localisation de $\E$ par rapport à $\widetilde{T}$.

\end{proof}

\begin{lemm}
\label{gen-coh-PXTindtPsurhol}
Soit $\theta =(f, a,Id) \colon  (\PP', T',X',Y)\to (\PP, T,X,Y)$ un morphisme de $d$-cadres 
tel que $Y$ soit dense dans $X'$, $a$ soit propre et $f ^{-1} (T)$ soit un diviseur de $P'$.

\begin{enumerate}

\item  \label{gen-coh-PXTindtPsurhol-i-fact} 
On dispose des deux factorisations de la forme
$\R \underline{\Gamma} ^\dag _{X'}  f ^! 
\colon 
(F\text{-})D ^\mathrm{b} _\mathrm{surcoh}  (\PP, T, X/K)
 \to
(F\text{-})D ^\mathrm{b} _\mathrm{surcoh}  (\PP', T', X'/K)$
et
$f _+\colon 
(F\text{-})D ^\mathrm{b} _\mathrm{surcoh}  (\PP', T', X'/K)
 \to
(F\text{-})D ^\mathrm{b} _\mathrm{surcoh}  (\PP, T, X/K)$.
De même en remplaçant l'indice {\og surcoh\fg} 
par respectivement {\og surhol\fg} ou {\og isoc\fg}.

\item \label{gen-coh-PXTindtPsurhol-i} 
Pour tout $\E \in (F\text{-})\mathrm{Surcoh}
(\PP, T, X/K)$,
pour tout 
$\E '\in (F\text{-})\mathrm{Surcoh}
(\PP', T', X'/K)$,
pour tout $j \in \Z\setminus \{0\}$,
$$\mathcal{H} ^j (\R \underline{\Gamma} ^\dag _{X'} f ^! (\E) ) =0,
\mathcal{H} ^j (f_+(\E')) =0.$$
De même en remplaçant {\og $\mathrm{Surcoh}$\fg} 
par {\og $\mathrm{Surhol}$\fg}.

\item \label{gen-coh-PXTindtPsurhol-iii} 
On suppose $Y$ lisse. 
De plus, 
si $f$ est propre 
ou si $f$ est une immersion ouverte alors les foncteurs $\R \underline{\Gamma} ^\dag _{X'}  f ^! $ et
$f _+$ induisent des équivalences quasi-inverses entre
 $(F\text{-})D ^\mathrm{b} _\mathrm{isoc}  (\PP, T, X/K)$
et $(F\text{-})D ^\mathrm{b} _\mathrm{isoc}  (\PP', T', X'/K)$.

\item \label{gen-coh-PXTindtPsurhol-ii} 
De plus, si $f$ est propre ou si $f$ est une immersion ouverte alors les foncteurs $\R \underline{\Gamma} ^\dag _{X'}  f ^! $ et
$f _+$ induisent des équivalences quasi-inverses entre
 $(F\text{-})D ^\mathrm{b} _\mathrm{surcoh}  (\PP, T, X/K)$ et $(F\text{-})D ^\mathrm{b} _\mathrm{surcoh}  (\PP', T', X'/K)$.
De même en remplaçant les indices {\og $\mathrm{surcoh}$\fg} 
par {\og $\mathrm{surhol}$\fg}.

\end{enumerate}
\end{lemm}

\begin{proof}
0) Comme $a$ est propre, l'immersion ouverte $Y \hookrightarrow X '\setminus f ^{-1} (T) $ est aussi fermée.
Comme $Y$ est dense dans $X'$, on en déduit $X '\setminus f ^{-1} (T) =Y$. 
Ainsi, $f ^{-1} (T)$ est un diviseur de $P'$ tel que $X '\setminus f ^{-1} (T) =Y$.
Grâce au lemme \ref{lemme-coh-PXTindtPGEN}, on se ramène à supposer $T'= f ^{-1} (T)$.

1) 
Les factorisations de \ref{gen-coh-PXTindtPsurhol}.\ref{gen-coh-PXTindtPsurhol-i-fact} 
découlent alors de \ref{surhol-conjA} dans les cas surcohérent ou surholonome
et de \cite[5.4.1]{caro-pleine-fidelite}
dans le dernier cas.

2) Vérifions à présent les annulations de \ref{gen-coh-PXTindtPsurhol}.\ref{gen-coh-PXTindtPsurhol-i}. 
D'après \cite[4.3.12]{Be1}, on se ramène au cas où $T$ et donc $T'$ sont vides. 
Il s'agit alors de reprendre mot pour mot (le début de) la preuve de \cite[3.2.6]{caro_surcoherent}
en y constatant que l'hypothèse que le morphisme soit propre et lisse est inutile grâce au théorème \ref{surhol-conjA} 
qui généralise \cite[3.1.9]{caro_surcoherent}).

3) En supposant $Y$ lisse, prouvons à présent \ref{gen-coh-PXTindtPsurhol}.\ref{gen-coh-PXTindtPsurhol-iii}.
On se ramène facilement au cas où la compactification partielle est lisse, cas traité dans  \ref{coh-PXTindtPbis}.\ref{coh-PXTindtP-iv}.
En effet, soient 
$\E\in (F\text{-})D ^\mathrm{b} _\mathrm{isoc}  (\PP, T, X/K)$
et 
$\E'\in (F\text{-})D ^\mathrm{b} _\mathrm{isoc}  (\PP', T', X'/K)$.
Notons $\U:= \PP \setminus T$, $\U':= \PP' \setminus T'$.

3.i) 
Traitons d'abord le cas où $f$ est propre et lisse.
Comme $f$ est propre, on dispose alors par adjonction des morphismes canoniques 
$f_+ \circ \R \underline{\Gamma} ^\dag _{X'}  f ^!  (\E) \to \E $, 
$\E' \to \R \underline{\Gamma} ^\dag _{X'}  f ^!  \circ f_+ (\E') $.
D'après \cite[4.3.12]{Be1}, pour vérifier que ces morphismes sont des isomorphismes, 
il suffit de l'établir respectivement en dehors des diviseurs $T$ et $T'$,
ce qui nous ramène au cas de la compactification partielle lisse
qui a été traité dans la preuve de \ref{coh-PXTindtPbis}.\ref{coh-PXTindtP-iv}.

3.ii)
Supposons à présent que $f$ soit une immersion ouverte.
Via \ref{gen-coh-PXTindtPsurhol}.\ref{gen-coh-PXTindtPsurhol-i},
on obtient alors les isomorphismes canoniques 
$\R \underline{\Gamma} ^\dag _{X'}\circ   f ^! (\E) \riso f ^{*} (\E)=f ^{-1} (\E)$ et $f _{*} (\E')\riso f _+(\E')$.
On dispose de plus des morphismes d'adjonction :
$\E \to f _{*} f ^{-1} (\E)$ et $f ^{-1}f _{*} (\E') \to  \E'$, dont le dernier est toujours un isomorphisme. 
Il  résulte de \cite[4.3.12]{Be1} que ces morphismes sont des isomorphismes (car ils le sont en dehors des diviseurs, ce qui nous ramène à la preuve de \ref{coh-PXTindtPbis}.\ref{coh-PXTindtP-iv}).

4)  
Vérifions à présent \ref{gen-coh-PXTindtPsurhol}.\ref{gen-coh-PXTindtPsurhol-ii}  
dans le cas surcohérent. On se ramène au cas où $Y$ est lisse par dévissage.
En effet, soient 
$\E \in (F\text{-})D ^\mathrm{b} _\mathrm{surcoh}  (\PP, T, X/K)$ et 
$\E'\in (F\text{-})D ^\mathrm{b} _\mathrm{surcoh}  (\PP', T', X'/K)$.
Par \ref{indt-X}, comme aucun des foncteurs ne fait intervenir $X$, 
quitte à remplacer $X$ par l'adhérence de $Y$ dans $X$, 
on peut supposer $Y$ dense dans $X$.
D'après \cite[6.2.3]{caro-pleine-fidelite}, 
il existe des diviseurs $T _1, \dots,T _{r}$ contenant $T$ avec $T _r =T$ tels que,
en notant $T _0: =X$ et, 
pour $i=0,\dots , r-1$, $X _{i}=T _0 \cap \dots \cap T _i $, 
le schéma $Y _i:= X _{i}\setminus T _{i+1}$
est lisse et 
$  \mathcal{H} ^j (\R \underline{\Gamma} ^\dag _{X _i}  (\hdag T _{i+1}) (\E) )
  \in (F\text{-})\mathrm{Isoc} ^{\dag \dag}( \PP, T _{i+1}, X _i/K)
  =(F\text{-})\mathrm{Isoc} ^{\dag \dag}( Y_i,X _{i}/K) $, pour tout entier $j$.
  De même pour $\E'$.
Il en résulte par dévissage que les morphismes canoniques d'adjonction (voir les étapes $3.i)$ et $3.ii)$)
sont des isomorphismes. D'où le résultat pour le cas surcohérent.
Le cas surholonome se déduit du cas surcohérent.

\end{proof}

\begin{rema}
En procédant de manière analogue à l'étape $1)$ de la preuve de \ref{b+proprelisse-varlisse-gen},
les résultats de \ref{gen-coh-PXTindtPsurhol}.\ref{gen-coh-PXTindtPsurhol-i-fact}
restent valables sans l'hypothèse que 
$f ^{-1} (T)$ soit un diviseur de $P'$.
\end{rema}

\begin{prop}
\label{prop-nota-Dsurcv}  

Soient $X$ une $k$-variété et $Y$ un ouvert de $X$ tels que $(Y,X)$ soit proprement $d$-réalisable (\ref{defi-(d)plongprop}).

Choisissons $\widetilde{\PP} $ un $\V$-schéma formel propre et lisse,
$\widetilde{T}$ un diviseur de $\widetilde{P}$ tels qu'il existe une immersion $X \hookrightarrow \widetilde{\PP}$
induisant l'égalité $Y = X \setminus \widetilde{T}$. 
Choisissons $\PP$ un ouvert de $\widetilde{\PP}$ contenant $X$ tel que
l'immersion $X \hookrightarrow \PP$ canoniquement induite soit fermée. On note alors $T := P \cap \widetilde{T}$ le diviseur de $P$ induit. 
\begin{itemize}

\item Posons
$\mathfrak{C}=(F\text{-})\mathrm{Surcoh}$
ou 
$\mathfrak{C}=(F\text{-})\mathrm{Surhol}$
ou 
$\mathfrak{C}=
(F\text{-})D ^\mathrm{b} _\mathrm{surcoh} $
ou 
$\mathfrak{C}=
(F\text{-})D ^\mathrm{b} _\mathrm{surhol} $.
La catégorie $\mathfrak{C} (\PP, T, X/K)$
ne dépend pas, à isomorphisme canonique près, des choix effectués (mais seulement de $(Y,X)/K$).

\item  Lorsque $Y$ est lisse,  la catégorie $(F\text{-})D ^\mathrm{b} _\mathrm{isoc}  (\PP, T, X/K)$  ne dépend de même canoniquement que de $(Y,X)/K$.
\end{itemize}
\end{prop}

\begin{proof}
Il s'agit alors de reprendre les arguments de \ref{nota-Dsurcv-cpart} en remplaçant la référence 
\ref{coh-PXTindtPbis}.\ref{coh-PXTindtP-iv} par
\ref{gen-coh-PXTindtPsurhol} (qui n'est d'ailleurs utilisé ici que dans le cas où $X'=X$).
\end{proof}

\begin{nota}
\label{nota-Dsurcv}  
On garde les notations de \ref{prop-nota-Dsurcv}.
\begin{itemize}

\item
La catégorie $\mathfrak{C} (\PP, T, X/K)$
se notera alors sans ambiguïté 
$\mathfrak{C} (\D ^\dag _{(Y,X)/K}) $.
Les objets de
$(F\text{-})D ^\mathrm{b} _\mathrm{surcoh} (\D ^\dag _{(Y,X)/K}) $
(resp. $(F\text{-})D ^\mathrm{b} _\mathrm{surhol} (\D ^\dag _{(Y,X)/K}) $)
sont les {\og $(F\text{-})$complexes surcohérents (resp. surholonomes) de $\D ^\dag _{(Y,X)/K}$-modules\fg}
ou {\og les $(F\text{-})$complexes surcohérents (resp. surholonomes) de $\D$-modules arithmétiques sur $(Y,X)/K$ \fg}.
De même, pour les modules. 

\item  Lorsque $Y$ est lisse, la catégorie $(F\text{-})D ^\mathrm{b} _\mathrm{isoc}  (\PP, T, X/K)$ sera notée
$(F\text{-})D ^\mathrm{b} _\mathrm{isoc} (\D ^\dag _{(Y,X)/K})$.
Lorsque $Y=X$, on notera $(F\text{-})D ^\mathrm{b} _\mathrm{cv} (\D ^\dag _{Y/K})$ au lieu de 
$(F\text{-})D ^\mathrm{b} _\mathrm{isoc} (\D ^\dag _{(Y,Y)/K})$.
On remarque, grâce à \cite{caro-Tsuzuki-2008},
qu'avec une structure de Frobenius on obtient l'inclusion canonique
$F\text{-}D ^\mathrm{b} _\mathrm{isoc} (\D ^\dag _{(Y,X)/K}) 
\subset
F\text{-}D ^\mathrm{b} _\mathrm{surhol} (\D ^\dag _{(Y,X)/K}) $.

\end{itemize}
\end{nota}

\begin{rema}
\label{rema-surcv-si-cv}
On reprend les notations de \ref{nota-Dsurcv}. 
\begin{itemize}
\item  Le foncteur pleinement fidèle $(F\text{-})\mathrm{Surcoh} (\PP, T, X/K) \to  (F\text{-})D ^\mathrm{b} _\mathrm{surcoh}  (\PP, T, X/K)$
ne dépend pas des choix effectués.
On obtient ainsi le foncteur canonique pleinement fidèle 
$(F\text{-})\mathrm{Surcoh} (Y, X/K)\to (F\text{-})D ^\mathrm{b} _\mathrm{surcoh} (\D ^\dag _{(Y,X)/K}) $.

\item Supposons $Y$ est lisse. Posons $\U:=\PP\setminus T$.
Un module $\E$ de $(F\text{-})\mathrm{Surcoh} (\PP, T, X/K)$ appartient à $(F\text{-})\mathrm{Isoc}^{\dag \dag} (\PP, T, X/K)$ 
si et seulement si $\E |\U \in (F\text{-})\mathrm{Isoc}^{\dag \dag} (\U, Y/K)$
(cela résulte de la caractérisation de Berthelot énoncée dans \cite[2.2.12]{caro_courbe-nouveau}). 
On en déduit qu'un $F$-complexe $\E$ de 
$(F\text{-})D ^\mathrm{b} _\mathrm{surcoh} (\D ^\dag _{(Y,X)/K})$ appartient à 
$(F\text{-})D ^\mathrm{b} _\mathrm{isoc} (\D ^\dag _{(Y,X)/K})$ si et seulement si 
$\E |(Y,Y)$ appartient à $(F\text{-})D ^\mathrm{b} _\mathrm{isoc} (\D ^\dag _{(Y,Y)/K})$.

\end{itemize}

\end{rema}

\begin{lemm} 
\label{lemm-morph-couple}
Soit  $a \colon (Y',X') \to (Y,X)$ un morphisme de couples de $k$-variétés proprement $d$-réalisables.
Il existe alors un diagramme commutatif de la forme :
\begin{equation}
  \label{diagdefi-morph-couple}
  \xymatrix {
  {X'} \ar[r] ^-{u'} \ar[d] _{a } & {\PP'} \ar[r] ^-{j'} \ar[d] ^{f}&
  {\widetilde{\PP}'} \ar[d] ^{\widetilde{f}}   \\
  {X} \ar[r] ^-{u} & {\PP} \ar[r] ^-{j}& {\widetilde{\PP},} 
  }
  \end{equation}
où $\widetilde{f}$ est un morphisme lisse de $\V $-schémas formels propres et lisses, 
où $u$ et $u ' $ sont des immersions fermées, 
où $j$ et $j'$ sont des immersions ouvertes, 
tels qu'il existe un diviseur $\widetilde{T}$ de $\widetilde{P}$ (resp. un diviseur $\widetilde{T}'$ de $\widetilde{P}'$)
vérifiant 
$Y= X\setminus \widetilde{T}$ (resp. $Y'= X'\setminus \widetilde{T}'$)
et $\widetilde{T} '\supset \widetilde{f} ^{-1} (\widetilde{T})$.

\begin{itemize}
\item 
Si $Y' = a  ^{-1} (Y)$ alors on peut choisir de plus $\widetilde{T} '=\widetilde{f} ^{-1} (\widetilde{T})$.
\item Si $a $ est propre, le diagramme de droite de \ref{diagdefi-morph-couple} peut alors être choisi cartésien.

\end{itemize}

\end{lemm}

\begin{proof}
Soit $\widetilde{\PP} $ un $\V$-schéma formel propre et lisse,
$\widetilde{T}$ un diviseur de $\widetilde{P}$ tels qu'il existe une immersion $X \hookrightarrow \widetilde{\PP}$
induisant l'égalité $Y = X \setminus \widetilde{T}$. 
On procède de même avec des primes. 
Posons $\widetilde{\PP}'':=\widetilde{\PP}\times \widetilde{\PP}'$, 
$\widetilde{q}\colon  \widetilde{\PP}'' \to \widetilde{\PP}$ et $\widetilde{q}'\colon  \widetilde{\PP}'' \to \widetilde{\PP}'$
les projections canoniques, 
$\widetilde{T} '':= \widetilde{q} ^{\prime -1} (\widetilde{T}') \cup \widetilde{q} ^{-1} (\widetilde{T})$.
On dispose de l'immersion canonique $X' \hookrightarrow \widetilde{\PP}''$ induite par le graphe de $a $ et les immersions $X \hookrightarrow \widetilde{\PP}$, $X '\hookrightarrow \widetilde{\PP}'$.
Comme $Y' = X' \setminus \widetilde{T} ''$, 
quitte à remplacer $\widetilde{\PP}'$ par  $\widetilde{\PP}''$ et 
$\widetilde{T}'$ par  $\widetilde{T}''$, 
on peut supposer qu'il existe
un morphisme $\widetilde{f}\colon \widetilde{\PP}' \rightarrow \widetilde{\PP}$ lisse, prolongeant $a $ et tel que $\widetilde{T} '\supset \widetilde{f} ^{-1} (\widetilde{T})$.
Soit $\PP$ un ouvert de $\widetilde{\PP}$ contenant $X$ tel que
l'immersion $X \hookrightarrow \PP$ canoniquement induite soit fermée. 
Comme $\widetilde{f} ^{-1} (\PP) \supset X'$, il existe un ouvert $\PP'$ inclus dans $\widetilde{f} ^{-1} (\PP) $ contenant $X'$
tel que l'immersion induite $X '\hookrightarrow \PP'$ soit fermée. 
Ainsi, $\widetilde{f}$ induit le morphisme
$f\colon \PP' \to \PP$.

-Lorsque $Y' = a  ^{-1} (Y)$, comme $Y' = X' \setminus ( \widetilde{f} ^{-1} (\widetilde{T}) )$, alors on peut choisir 
$\widetilde{T} '= \widetilde{f} ^{-1} (\widetilde{T})$.

-Enfin, lorsque $a $ est propre, le morphisme $X '\to \PP$ l'est aussi. 
Comme le morphisme $\widetilde{f} ^{-1} (\PP) \rightarrow \PP$ induit par $\widetilde{f}$ est propre, 
il en résulte que l'immersion $X' \hookrightarrow \widetilde{f} ^{-1} (\PP)$ est fermée. Dans ce cas $\PP'$ peut être choisi égal à $\widetilde{f} ^{-1} (\PP) $.

\end{proof}

\begin{prop}
\label{operation-cohomoYXsurhol}
Soit $a \colon (Y',X') \to (Y,X)$ un morphisme de couples de $k$-variétés proprement $d$-réalisables.
\begin{enumerate}
\item On dispose alors du foncteur image inverse extraordinaire par $a $ : 
\begin{equation}
\notag 
a  ^{!} \colon  (F\text{-})D ^\mathrm{b} _\mathrm{surcoh} (\D ^\dag _{(Y,X)/K})
\to 
(F\text{-})D ^\mathrm{b} _\mathrm{surcoh} (\D ^\dag _{(Y',X')/K}).
\end{equation}
\item Lorsque que $a $ est un morphisme propre (voir la définition \ref{defi-(d)plongprop}), on dispose du foncteur image directe par $ a $ :
\begin{equation}
\notag 
a _{+}\colon  (F\text{-})D ^\mathrm{b} _\mathrm{surcoh} (\D ^\dag _{(Y',X')/K})
\to 
(F\text{-})D ^\mathrm{b} _\mathrm{surcoh} (\D ^\dag _{(Y,X)/K}).
\end{equation}

\end{enumerate}

On bénéficie des propriétés analogues pour les complexes surholonomes. 
\end{prop}

\begin{proof}
D'après \ref{lemm-morph-couple}, on dispose alors d'un diagramme de la forme \ref{diagdefi-morph-couple}
satisfaisant aux conditions requises de \ref{lemm-morph-couple} et dont on reprendra les notations.
Notons de plus $T := \widetilde{T} \cap \PP$ et $T ':= \widetilde{T}' \cap \PP'$.
\begin{itemize}
\item  Pour tout $\E \in (F\text{-})D ^\mathrm{b} _\mathrm{surcoh}  (\PP, T, X/K)$, 
 le foncteur $a  ^{!}$ est alors défini en posant 
 $a  ^{!}(\E) := (\hdag T') \circ \R \underline{\Gamma} ^\dag _{X'} \circ f ^{!} (\E)$.
\item  Lorsque $ a $ est propre, pour tout $\E '\in (F\text{-})D ^\mathrm{b} _\mathrm{surcoh}  (\PP', T', X'/K)$, le foncteur $a _{+}$ est alors défini en posant 
 $a _{+} (\E') :=  f _{+} (\E')$. On vérifie que le complexe $ f _{+} (\E')\in (F\text{-})D ^\mathrm{b} _\mathrm{surcoh}  (\PP, T, X/K)$ grâce à \ref{surhol-conjA}.
\item  Par transitivité pour la composition des foncteurs images directes et images inverses extraordinaires, 
via les équivalences canoniques de catégories garantissant l'indépendance canonique de 
$(F\text{-})D ^\mathrm{b} _\mathrm{surcoh} (\D ^\dag _{(Y,X)/K})$ (qui font intervenir ces foncteurs),
on vérifie que la définition des foncteurs $a  ^{!}$  et $a _{+}$ ne dépend pas, à isomorphisme canonique près, 
du choix du diagramme de la forme \ref{diagdefi-morph-couple}. 
\end{itemize}
\end{proof}

\begin{nota}
\label{fonct-restri}
Soit $a \colon (Y',X') \to (Y,X)$ un morphisme de couples de $k$-variétés proprement $d$-réalisables
tel que $X' \to X$ soit une immersion ouverte. 
Remarquons que l'on peut obtenir un diagramme \ref{diagdefi-morph-couple}
satisfaisant aux conditions requises avec $\widetilde{f}=Id$. 
Pour tout $\E \in (F\text{-})D ^\mathrm{b} _\mathrm{surcoh} (\D ^\dag _{(Y,X)/K})$,
on notera $\E | (Y',X'):=a  ^{!}(\E)$ ($= j ^{\prime *} (\E)$).
\end{nota}

\begin{lemm}
\label{a0+comm-restrict}
Soit $a \colon (Y',X') \to (Y,X)$ un morphisme propre de couples de $k$-variétés proprement $d$-réalisables tel que $Y' = a  ^{-1} (Y)$.
Soit $b \,: \,(Y',Y') \to (Y,Y)$ le morphisme induit par $a $.
Pour tout $\E ' \in (F\text{-})D ^\mathrm{b} _\mathrm{surcoh} (\D ^\dag _{(Y',X')/K})$, on dispose de l'isomorphisme canonique 
\begin{equation}
\notag
a _{+} (\E') |(Y,Y) \riso b _{+} (\E'|(Y',Y')).
\end{equation}
\end{lemm}

\begin{proof}
D'après \ref{lemm-morph-couple}, on dispose alors d'un diagramme de la forme \ref{diagdefi-morph-couple}
satisfaisant aux conditions requises de \ref{lemm-morph-couple} et dont on reprendra les notations.
Comme $Y' = a  ^{-1} (Y)$, on peut en outre supposer que $\widetilde{T} '=\widetilde{f} ^{-1} (\widetilde{T})$. 
D'où le résultat par définition des foncteurs restriction et image directe.
\end{proof}

Le théorème \ref{b+proprelisse} ci-dessous est une version formelle de la conjecture de Berthelot (dans le cas complet).
Il se déduit aussitôt du théorème \ref{b+proprelisse-varlisse-gen}.

\begin{theo}
\label{b+proprelisse}
Soit $a \colon (Y',X') \to (Y,X)$ un morphisme propre de couples de $k$-variétés proprement $d$-réalisables (voir \ref{defi-(d)plongprop}).
On suppose $Y$ est lisse, $Y' = a  ^{-1} (Y)$ et le morphisme 
induit $Y'\rightarrow Y$ lisse.

Le foncteur image directe par $a $ induit alors la factorisation :
\begin{equation}
  \label{a+isocsurcv-pl-fonct}
a _{+} \colon  (F\text{-})D ^\mathrm{b} _\mathrm{isoc} (\D ^\dag _{(Y',X')/K})
\rightarrow
(F\text{-})D ^\mathrm{b} _\mathrm{isoc} (\D ^\dag _{(Y,X)/K}).
\end{equation}
\end{theo}

\subsection{Indépendance par rapport à la compactification}

\begin{prop}
\label{prop-nota-Dsurcv-proper}  

Soient $Y$ une variété proprement $d$-réalisable (\ref{defi-var-dplong}).

Choisissons $\PP$ un $\V$-schéma formel propre et lisse,
$T$ un diviseur de $P$ tels qu'il existe une immersion fermée $X \hookrightarrow \PP$ 
induisant l'égalité $Y = X \setminus T$. 
\begin{itemize}

\item Posons
$\mathfrak{C}=(F\text{-})\mathrm{Surcoh}$
ou 
$\mathfrak{C}=(F\text{-})\mathrm{Surhol}$
ou 
$\mathfrak{C}=
(F\text{-})D ^\mathrm{b} _\mathrm{surcoh} $
ou 
$\mathfrak{C}=
(F\text{-})D ^\mathrm{b} _\mathrm{surhol} $.
La catégorie $\mathfrak{C} (\PP, T, X/K)$
ne dépend pas, à isomorphisme canonique près, des choix effectués (mais seulement de $Y/K$).
On la notera alors sans ambiguïté 
$\mathfrak{C} (\D ^\dag _{Y/K}) $.

\item  Lorsque $Y$ est lisse,  la catégorie $(F\text{-})D ^\mathrm{b} _\mathrm{isoc}  (\PP, T, X/K)$  ne dépend de même canoniquement que de $Y$.
On la notera alors sans ambiguïté $(F\text{-})D ^\mathrm{b} _\mathrm{isoc} (\D ^\dag _{Y/K})$.
\end{itemize}
\end{prop}

\begin{proof}
Faisons un tel second choix : soient $\PP' $ un $\V$-schéma formel propre et lisse,
$X' \hookrightarrow \PP'$ une immersion fermée, 
$T'$ un diviseur de $P'$ 
induisant l'égalité $Y = X '\setminus T'$. 
Soient $\PP'':=\PP\times \PP'$,
$q\colon \PP''\to \PP$,
$q'\colon \PP''\to \PP'$ les projections canoniques, 
$T'' _{1}:= q ^{-1} (T)$, 
$T'' _{2}:= q ^{\prime -1} (T')$, 
$T'' := T'' _{1} \cup T'' _{2}$,
$X''$ l'adhérence de $Y$ dans $P''$.
On vérifie $q (X'') \subset X$.
Comme $q$ est propre, l'immersion ouverte
$Y \subset X '' \cap q ^{-1} (Y)=X ''\setminus T'' _{1}$ est aussi fermée.
Comme $Y$ est dense dans $X''$, 
on en déduit
$Y = X ''\setminus T'' _{1}$.  
De même, on obtient $Y = X ''\setminus T'' _{2}$ et donc 
$Y = X ''\setminus T'' $. 
Notons $a\colon X''\to X$ le morphisme induit par $q$. 
Comme $a$ et $q$ sont propres,
il résulte alors de \ref{gen-coh-PXTindtPsurhol} 
utilisé pour $(q, a,Id) \colon  (\PP'', T'',X'',Y)\to (\PP, T,X,Y)$
que
les foncteurs $\R \underline{\Gamma} ^\dag _{X''}  q ^! $ et
$q _+$ induisent des équivalences quasi-inverses 
entre les catégories 
$(F\text{-})\mathfrak{C}  (\PP, T, X/K)$ 
et 
$(F\text{-})\mathfrak{C}  (\PP'', T'', X''/K)$.
Symétriquement, on vérifie que les catégories 
$(F\text{-})\mathfrak{C}  (\PP', T', X'/K)$ 
et 
$(F\text{-})\mathfrak{C}  (\PP'', T'', X''/K)$
sont canoniquement équivalentes. 
D'où le résultat. 

\end{proof}

\begin{rema}
$\bullet$ Soit $Y$ une $k$-variété proprement $d$-réalisable. 
On bénéficie alors, pour tout entier $l$, des foncteurs canoniques
  $\mathcal{H} ^l \colon  (F\text{-})D ^\mathrm{b} _\mathrm{surcoh} (\D ^\dag _{Y/K})
  \rightarrow (F\text{-})\mathrm{Surcoh} (Y/K)$
  et 
   $\mathcal{H} ^l \colon  (F\text{-})D ^\mathrm{b} _\mathrm{surhol} (\D ^\dag _{Y/K})
  \rightarrow (F\text{-})\mathrm{Surhol} (Y/K)$.

  En effet, soient $\PP$ un $\V$-schéma formel propre et lisse, $T$ un diviseur de $P$ et
  $X$ un sous-schéma fermé de $P$ tels que $Y = X \setminus T$. 
  Le cas surholonome se traitant de manière analogue (grâce aussi à \ref{surhol-espcoh}), 
  traitons le cas surcohérent. 
  On dispose du foncteur 
  $\mathcal{H} ^l \colon  (F\text{-})D ^\mathrm{b} _\mathrm{surcoh} (\PP, T, X/K)
  \rightarrow (F\text{-})\mathrm{Surcoh} (\PP, T, X/K)$.
Grâce au théorème d'annulation de \ref{gen-coh-PXTindtPsurhol}.\ref{gen-coh-PXTindtPsurhol-i}, 
on vérifie que ce foncteur $\mathcal{H} ^l$ commute aux équivalences canoniques de catégories
d'indépendance par rapport au choix de tel $d$-cadre $(\PP, T,X, Y)$ (ces équivalences sont construites dans la preuve de \ref{prop-nota-Dsurcv-proper}).

$\bullet$
Lorsque $T$ est un sous-schéma fermé, le foncteur $(\hdag T)$ reste bien défini 
(voir \cite{caro_surcoherent}). On peut donc étendre les définitions des catégories 
de $(F\text{-})D ^\mathrm{b} _\mathrm{surhol} (\PP, T, X/K)$ et
$(F\text{-})\mathrm{Surhol} (\PP, T, X/K)$ 
à ce cas (par contre, il est moins clair comment donner un sens à la $\smash{\D} ^{\dag } _{\PP } (\hdag T)_{\Q}$-surcohérence). 
Mais on ne dispose plus de la factorisation de la forme
$\mathcal{H} ^l \colon  (F\text{-})D ^\mathrm{b} _\mathrm{surhol} (\PP, T, X/K)
  \rightarrow (F\text{-})\mathrm{Surhol} (\PP, T, X/K)$.
  Par exemple, si $\PP$ est un $\V$-schéma formel propre et lisse,
  si $T$ est un sous-schéma lisse de codimension $2$ dans $P$ et si $Y = P \setminus T$,
  alors $\O _{\PP} (\hdag T) _\Q \in F\text{-}D ^\mathrm{b} _\mathrm{surhol} (\PP, T, P/K)$
  mais $\mathcal{H}  ^0 (\O _{\PP} (\hdag T) _\Q ) = \O _{\PP,\Q} \not \in
  \mathrm{Surhol} (\PP, T, P/K)$.
\end{rema}

\begin{prop}
\label{operation-cohomoYsurcoh}
Soit $b \colon Y'\to Y$ un morphisme de $k$-variétés proprement $d$-réalisables.
\begin{enumerate}
\item On dispose alors du foncteur image inverse extraordinaire par $a $ : 
\begin{equation}
\notag 
b  ^{!} \colon  (F\text{-})D ^\mathrm{b} _\mathrm{surcoh} (\D ^\dag _{Y/K})
\to 
(F\text{-})D ^\mathrm{b} _\mathrm{surcoh} (\D ^\dag _{Y'/K}).
\end{equation}
\item On dispose du foncteur image directe par $ b $ :
\begin{equation}
\notag 
b _{+}\colon  (F\text{-})D ^\mathrm{b} _\mathrm{surcoh} (\D ^\dag _{Y'/K})
\to 
(F\text{-})D ^\mathrm{b} _\mathrm{surcoh} (\D ^\dag _{Y/K}).
\end{equation}

\end{enumerate}

On bénéficie des factorisations analogues pour les complexes surholonomes. 
\end{prop}

\begin{proof}
On construit les foncteurs de manière identique à ceux de 
\ref{operation-cohomoYXsurhol} et on vérifie qu'ils sont indépendants des choix faits de manière analogues.

\end{proof}

\begin{nota}
\label{fonct-restri-Y}
Soit $b \colon Y' \to Y$ un morphisme de $k$-variétés proprement $d$-réalisables
qui soit une immersion ouverte. 
Pour tout $\E \in  (F\text{-})D ^\mathrm{b} _\mathrm{surcoh} (\D ^\dag _{Y/K})$,
on notera alors $\E | Y':=b  ^{!}(\E)$.
\end{nota}

\subsection{Complexes dont les espaces de cohomologie sont des
isocristaux partiellement surcohérents: cas général}

La définition suivante étend au cas général (i.e. $Y$ n'est plus forcément lisse)
celle de \ref{prop-nota-Dsurcv}.
\begin{defi}
\label{defi-isocdagdag}
 Soit $(Y,X)$ un couple de $k$-variétés proprement $d$-réalisable (\ref{defi-(d)plongprop}).
On désigne par $(F\text{-})D ^\mathrm{b} _\mathrm{isoc} (\D ^\dag _{(Y,X)/K})$ la sous-catégorie pleine de
$(F\text{-})D ^\mathrm{b} _\mathrm{surcoh} (\D ^\dag _{(Y,X)/K})$ 
des objets $\E$ satisfaisant la propriété suivante : pour tout morphisme 
$a \colon (Y',X') \to (Y,X)$ de couples de $k$-variétés proprement $d$-réalisables 
tel que $Y'$ soit lisse, on a alors $a  ^! (\E) \in (F\text{-})D ^\mathrm{b} _\mathrm{isoc} (\D ^\dag _{(Y',X')/K})$, où la catégorie $(F\text{-})D ^\mathrm{b} _\mathrm{isoc} (\D ^\dag _{(Y',X')/K})$
est celle définie en \ref{prop-nota-Dsurcv}.

\end{defi}

\begin{theo}
\label{prop-chgtbase}
Soient $a \colon (Y',X') \to (Y,X)$ un morphisme propre de couples de $k$-variétés proprement $d$-réalisables
et $q \colon (\widetilde{Y},\widetilde{X}) \to (Y,X)$ un morphisme de couples de $k$-variétés proprement $d$-réalisables.
  Soient $\widetilde{Y}  ' := \widetilde{Y} \times _{Y} Y'$, $\widetilde{X}  ' := \widetilde{X} \times _{X} X'$,
  $q '\colon (\widetilde{Y}  ' ,\widetilde{X}  ') \rightarrow (Y',X')$ et
  $\widetilde{a} \colon  (\widetilde{Y}  ' ,\widetilde{X}  ')  \rightarrow (\widetilde{Y},\widetilde{X})  $ les morphismes de couples de $k$-variétés proprement $d$-réalisables induits.   
  On dispose alors de l'isomorphisme de changement de base fonctoriel en $\E'\in (F\text{-})D ^\mathrm{b} _\mathrm{surcoh} (\D ^\dag _{(Y',X')/K}) $ :
  \begin{equation}
    \notag
  q  ^! \circ a _{+} (\E') \riso \widetilde{a} _{+} \circ q  ^{\prime!} (\E').
  \end{equation}
\end{theo}

\begin{proof}
Comme $a $ est propre, 
d'après \ref{lemm-morph-couple},
il existe un morphisme de $d$-cadres de la forme 
$(f, a,b) \colon  (\PP', T',X',Y')\to (\PP, T,X,Y)$
tel que $f$ soit un morphisme propre et lisse, 
$T' \supset f ^{-1}(T)$
et s'inscrivant dans un diagramme de la forme \ref{diagdefi-morph-couple}.
De même, il existe un morphisme de $d$-cadres de la forme 
$(g, q,r) \colon  (\widetilde{\PP} , \widetilde{T} ,\widetilde{X} ,\widetilde{Y} )\to (\PP, T,X,Y)$
tel que $g$ soit un morphisme lisse,
$\widetilde{T} \supset g ^{-1}(T)$
et s'inscrivant dans un diagramme de la forme \ref{diagdefi-morph-couple}.
  On note alors $\widetilde{\PP}' := \widetilde{\PP} \times _{\PP} \PP'$ et
  $g '\colon \widetilde{\PP}' \rightarrow \PP'$,
  $\widetilde{f}\colon \widetilde{\PP}' \rightarrow \widetilde{\PP}$,
  $\widetilde{T}':= \widetilde{f} ^{-1} (\widetilde{T}) \cup g ^{\prime -1} (T')$.
On remarque alors que
$\widetilde{Y}  ' = \widetilde{X}' \setminus \widetilde{T}'$.
On dispose ainsi des morphismes de $d$-cadres de la forme
$(g', q',r') \colon  (\widetilde{\PP} ', \widetilde{T} ',\widetilde{X} ',\widetilde{Y} ')\to (\PP', T',X',Y')$
et
$(\widetilde{f} , \widetilde{a} ,\widetilde{b} ) \colon (\widetilde{\PP} ', \widetilde{T} ',\widetilde{X} ',\widetilde{Y} ')\to (\widetilde{\PP} , \widetilde{T} ,\widetilde{X} ,\widetilde{Y} )$,
le dernier étant propre.

Par définition $q  ^{\prime!}(\E') =
\R \underline{\Gamma} ^\dag _{\widetilde{X}'} \circ (\hdag \widetilde{T}') \circ g ^{\prime !} (\E') $ et
$$\label{prop-chgtbase=1}
\widetilde{a} _{+} \circ q  ^{\prime!} (\E') =
\widetilde{f} _+ \circ \R \underline{\Gamma} ^\dag _{\widetilde{X}'} \circ (\hdag \widetilde{T}') \circ g ^{\prime !} (\E ').$$
D'un autre côté, on a par définition :
$$\label{prop-chgtbase=2} 
q  ^! \circ a _{+} (\E ')  =
\R \underline{\Gamma} ^\dag _{\widetilde{X}} \circ (\hdag \widetilde{T}) \circ g ^! \circ f _+ (\E ').$$
D'après 
l'isomorphisme de changement de base
(voir \cite[5.7]{Abe-Frob-Poincare-dual} ou 
\cite[5.4.6]{caro-stab-sys-ind-surcoh}),
$g ^! \circ f _+ (\E ') \riso \widetilde{f} _+ \circ g ^{\prime !}(\E ')$. D'où : 
$q  ^! \circ a _{+} (\E ')  \riso 
\R \underline{\Gamma} ^\dag _{\widetilde{X}} \circ (\hdag \widetilde{T}) \circ \widetilde{f} _+ \circ g ^{\prime !}(\E ')$.
Par commutation des foncteurs cohomologiques locaux à support strict et des foncteurs de localisation aux foncteurs images directes (voir \cite[2.2.18]{caro_surcoherent}),
il en résulte :
$$\label{prop-chgtbaseiso1}
q  ^! \circ a _{+} (\E ')  \riso
\widetilde{f} _+ \circ
\R \underline{\Gamma} ^\dag _{\widetilde{f} ^{-1} (\widetilde{X})} \circ (\hdag \widetilde{f} ^{-1} (\widetilde{T}))
\circ g ^{\prime !} (\E ').$$
Or, $\E ' \riso \R \underline{\Gamma} ^\dag _{X'} \circ (\hdag T') (\E ')$.
Par commutation des foncteurs cohomologiques locaux à support strict et des foncteurs de localisation aux foncteurs images inverses extraordinaires
(voir \cite[2.2.18]{caro_surcoherent}),
cela entraîne :
$g ^{\prime !} (\E ')
\riso
\R \underline{\Gamma} ^\dag _{g ^{\prime -1} (X')}\circ
(\hdag g ^{\prime -1} (T')) \circ g ^{\prime !} (\E ').$
D'où :
\begin{align}
  \label{prop-chgtbaseiso2}
  \notag
q  ^! \circ a _{+} (\E ')  & \riso
\widetilde{f} _+ \circ
\R \underline{\Gamma} ^\dag _{\widetilde{f} ^{-1} (\widetilde{X})} \circ (\hdag \widetilde{f} ^{-1} (\widetilde{T}))
\circ
\R \underline{\Gamma} ^\dag _{g ^{\prime -1} (X')}\circ
(\hdag g ^{\prime -1} (T')) \circ g ^{\prime !} (\E ')
\\
\notag
& \riso
\widetilde{f} _+ \circ \R \underline{\Gamma} ^\dag _{\widetilde{X}'} \circ (\hdag \widetilde{T}') \circ g ^{\prime !} (\E ')
=\widetilde{a} _{+} \circ q  ^{\prime!},
\end{align}
le deuxième isomorphisme découlant des égalités $\widetilde{f} ^{-1} (\widetilde{X}) \cap g ^{\prime -1} (X') = \widetilde{X}'$
et $\widetilde{T}'= \widetilde{f} ^{-1} (\widetilde{T}) \cup g ^{\prime -1} (T')$.
\end{proof}

\begin{theo}
\label{b+proprelisse-Yqcq}
Soit $a \colon (Y',X') \to (Y,X)$ un morphisme propre de couples de $k$-variétés proprement $d$-réalisables tel que $ a  ^{-1} (Y) =Y'$ et tel que le morphisme induit $Y' \to Y$ soit propre et lisse.
Alors, le foncteur image directe par $a $ se factorise sous la forme :
\begin{equation}
  \label{b+isocsurcv-pl-fonct-Yqcq}
a _{+} \colon  (F\text{-})D ^\mathrm{b} _\mathrm{isoc} (\D ^\dag _{(Y',X')/K})
\rightarrow
(F\text{-})D ^\mathrm{b} _\mathrm{isoc} (\D ^\dag _{(Y,X)/K}).
\end{equation}
\end{theo}

\begin{proof}
Par définition de $(F\text{-})D ^\mathrm{b} _\mathrm{isoc} (\D ^\dag _{(Y,X)/K})$,
on se ramène via le théorème de changement de base de \ref{prop-chgtbase} au cas où $Y$ est lisse.
Enfin, par \ref{b+proprelisse}, ce cas a déjà été traité.
\end{proof}

\begin{theo}
[Surconvergence générique de la surholonomie]
\label{surcv-dense}
Soient $(Y,X)$ un couple de $k$-variétés proprement $d$-réalisables (\ref{defi-(d)plongprop}),
$Y _1$ une composante irréductible de $Y$.
Pour tout $\E \in (F\text{-})D ^\mathrm{b} _\mathrm{surcoh} (\D ^\dag _{(Y,X)/K})$,
il existe alors un morphisme de couples de $k$-variétés proprement $d$-réalisables
de la forme $(\widetilde{Y},X) \to (Y,X)$ tel que $\widetilde{Y} \subset Y _1$ soit non vide et tel que
  $$\E | (\widetilde{Y},X)  \in (F\text{-})D ^\mathrm{b} _\mathrm{isoc} (\D ^\dag _{(\widetilde{Y},X)/K}).$$
\end{theo}

\begin{proof}
Cela résulte aussitôt de \cite[6.2.1]{caro-pleine-fidelite}.
\end{proof}

On déduit aussitôt de \ref{surcv-dense} le corollaire suivant :
\begin{coro}
Soient $a \colon (Y',X') \to (Y,X)$ un morphisme propre de couples de $k$-variétés proprement $d$-réalisables,
$Y _1$ une composante irréductible de $Y$.
  Pour tout $\E \in (F\text{-})D ^\mathrm{b} _\mathrm{isoc} (\D ^\dag _{(Y',X')/K})$,
il existe alors un morphisme de couples de $k$-variétés proprement $d$-réalisables
de la forme $(\widetilde{Y},X) \to (Y,X)$ tel que $\widetilde{Y} \subset Y _1$ soit non vide et tel que 
  $$a _{+} ( \E )|(\widetilde{Y},X) \in (F\text{-})D ^\mathrm{b} _\mathrm{isoc} (\D ^\dag _{(\widetilde{Y},X)/K}).$$
\end{coro}

\bibliographystyle{smfalpha}
\providecommand{\bysame}{\leavevmode ---\ }
\providecommand{\og}{``}
\providecommand{\fg}{''}
\providecommand{\smfandname}{et}
\providecommand{\smfedsname}{\'eds.}
\providecommand{\smfedname}{\'ed.}
\providecommand{\smfmastersthesisname}{M\'emoire}
\providecommand{\smfphdthesisname}{Th\`ese}

\bigskip
\noindent Daniel Caro\\
Laboratoire de Mathématiques Nicolas Oresme\\
Université de Caen
Campus 2\\
14032 Caen Cedex\\
France.\\
email: daniel.caro@unicaen.fr

\end{document}